\documentclass[11pt] {amsart}
\usepackage[margin=1in]{geometry}
\date{October 10, 2014}          
\usepackage{hyperref}
\usepackage{latexsym,amsmath,amsfonts,amscd,amssymb}
\usepackage{lscape}
\usepackage{array}
\DeclareFontFamily{OT1}{rsfs}{}
\DeclareFontShape{OT1}{rsfs}{n}{it}{<->rsfs10}{}
\DeclareMathAlphabet{\curly}{OT1}{rsfs}{n}{it}

\theoremstyle{plain}  
\newtheorem{theorem}{Theorem}[section]

\newtheorem*{theorem*}{Theorem}

\newtheorem{corollary}[theorem]{Corollary}
\newtheorem{lemma}[theorem]{Lemma}
\newtheorem{proposition}[theorem]{Proposition}

\theoremstyle{definition}
\newtheorem{definition}[theorem]{Definition}

\theoremstyle{remark}
\newtheorem{example}[theorem]{Example}

\newtheorem*{notation*}{Notation}
\newtheorem{remark}[theorem]{Remark}

\newtheorem*{claim*}{Claim}
\numberwithin{equation}{section}
\newcommand{\suchthat}{\;\;|\;\;}
\newcommand{\abs}[1]{\lvert#1\rvert}
\renewcommand{\leq}{\leqslant}
\renewcommand{\le}{\leqslant}
\renewcommand{\geq}{\geqslant}
\renewcommand{\ge}{\geqslant}

\newcommand{\into}{\hookrightarrow}

\newcommand{\R}{\mathbb{R}}

\newcommand{\Z}{\mathbb{Z}}
\newcommand{\C}{\mathbb{C}}

\newcommand{\HH}{\mathbb{H}}
\newcommand{\PP}{\mathbb{P}}

\newcommand{\cM}{\mathcal{M}}
\newcommand{\cO}{\mathcal{O}}
\newcommand{\cR}{\mathcal{R}}

\newcommand{\dbar}{\bar{\partial}}

\newcommand{\lra}{\longrightarrow}
\newcommand{\xra}{\xrightarrow}
\newcommand{\PU}{\mathrm{PU}}

\newcommand{\SU}{\mathrm{SU}}
\newcommand{\U}{\mathrm{U}}

\newcommand{\GL}{\mathrm{GL}}
\newcommand{\SL}{\mathrm{SL}}

\newcommand{\SO}{\mathrm{SO}}
\newcommand{\Sp}{\mathrm{Sp}}
\newcommand{\Spin}{\mathrm{Spin}}

\DeclareMathOperator{\Jac}{Jac}
\DeclareMathOperator{\ad}{ad}
\DeclareMathOperator{\Ad}{Ad}

\DeclareMathOperator{\rk}{rk}
\DeclareMathOperator{\rank}{rank}
\DeclareMathOperator{\im}{im}

\DeclareMathOperator{\Hom}{Hom}
\DeclareMathOperator{\End}{End}

\DeclareMathOperator{\Mat}{Mat}
\DeclareMathOperator{\Sym}{Sym}

\DeclareMathOperator{\Vol}{Vol}

\newcommand{\noi}{\noindent}
\newcommand{\st}{\;|\;}

\newcommand{\aut}{\operatorname{aut}}
\newcommand{\Aut}{\operatorname{Aut}}

\newcommand{\norm}[1]{\lVert#1\rVert}

\newcommand{\liem}{\mathfrak{m}}

\newcommand{\liemc}{\mathfrak{m}^{\mathbb{C}}}
\newcommand{\lieh}{\mathfrak{h}}

\newcommand{\lieg}{\mathfrak{g}}
\newcommand{\liegc}{\mathfrak{g}^{\mathbb{C}}}

\newcommand{\CC}{\mathbb{C}}

\newcommand{\RR}{\mathbb{R}}

\newcommand{\ZZ}{\mathbb{Z}}
\newcommand{\lie}{\mathfrak}

\newcommand{\mlie}{\mathfrak{m}}

\newcommand{\mclie}{\mlie^{\CC}}

\newcommand{\VVV}{{\curly V}}

\newcommand{\Tr}{\operatorname{Tr}}

\newcommand{\HC}{H^{\CC}}

\newcommand{\liez}{\mathfrak{z}}

\renewcommand{\phi}{\varphi}  

\begin{document}


\title[Higgs bundles for $\SO^\ast(2n)$]
{Higgs bundles for the non-compact dual of the special orthogonal group}

\author[S. B. Bradlow]{Steven B. Bradlow}
\address{Department of Mathematics \\
University of Illinois \\
Urbana \\
IL 61801 \\
USA }
\email{bradlow@math.uiuc.edu}

\author[O. Garc{\'\i}a-Prada]{Oscar Garc{\'\i}a-Prada}
\address{Instituto de Ciencias Matem\'aticas \\
CSIC-UAM-UC3M-UCM \\ Calle Nicol\'as Cabrera, 13--15 \\28049 Madrid \\ Spain}
\email{oscar.garcia-prada@icmat.es}

\author[P. B. Gothen]{Peter B. Gothen}
\address{Centro de
  Matem\'atica da Universidade do Porto \\
Faculdade de Ci\^encias da Universidade do Porto \\
Rua do Campo Alegre, s/n \\ 4169-007 Porto \\ Portugal }
\email{pbgothen@fc.up.pt}

\thanks{
  Members of the Research Group VBAC (Vector Bundles on Algebraic
  Curves) and the ESF Network ITGP (Interactions of Low-Dimensional
  Topology and Geometry with Mathematical Physics).
Second author is partially supported by  the Spanish Ministerio de 
Ciencia e Innovaci\'on (MICINN) under grants  MTM2007-67623 and
MTM2010-17717.
  Third author partially supported by FCT (Portugal) with EU
  (FEDER/COMPETE) and Portuguese funds under projects
  PTDC/MAT/099275/2008, PTDC/MAT/098770/2008, PTDC/MAT-GEO/0675/2012
  and PEst-C/MAT/UI0144/2013. 
  The authors also acknowledge support from U.S. National Science
  Foundation grants DMS 1107452, 1107263, 1107367 "RNMS: GEometric
  structures And Representation varieties" (the GEAR Network) }

\begin{abstract} Higgs bundles over a closed orientable surface can be defined for any real reductive Lie group $G$. In this paper we examine the case $G=\SO^*(2n)$. We describe a rigidity phenomenon encountered in the case of maximal Toledo invariant. Using this and  Morse theory in the mo\-du\-li space of Higgs bundles, we show that the moduli space is connected in this maximal Toledo case. The Morse theory also allows us to show connectedness when the Toledo invariant is zero. The correspondence between Higgs bundles and surface group representations thus allows us to count the connected components with zero and  maximal Toledo invariant in the moduli space of representations of the fundamental group of the surface in $\SO^*(2n)$. 
\end{abstract}

\maketitle

\section{Introduction}

Higgs bundles over a Riemann surface are intrinsically holomorphic objects.  Their moduli spaces can nevertheless be identified with representation varieties for the fundamental group of the surface even if the target group for the representations, or equivalently the group defining the Higgs bundles,  is a real
reductive Lie group.  If the group, say $G$, is of Hermitian type, i.e. if the homogeneous space $G/H$
(where $H$ is a maximal compact subgroup) is a Hermitian symmetric space, then the 
associated $G$-Higgs bundles have especially rich structure. The real
connected semisimple classical groups with this property are $\SU(p,q),
\Sp(2n,\R)$, $\SO(2,n)$,  and $\SO^*(2n)$. In this paper we examine in detail the case of  
$G=\SO^*(2n)$. In particular, we give proofs of the results that were announced in \cite{bradlow-garcia-prada-gothen:2005}.  

The theory of $G$-Higgs bundles with $G$ a real Lie group goes back to Hitchin's seminal papers \cite{hitchin:1987,hitchin:1992} in which split real forms were considered.  Since then, the $G$-Higgs bundles for many other real forms have been examined.  Among the real groups of Hermitian type, $\Sp(2n,\R)$ is special because it is also a split real form and therefore a
particular case of the situation studied by Hitchin.  Higgs bundles for the groups
$\Sp(2n,\R)$ and $\SU(p,q)$ have been studied (in chronological order) in
\cite{gothen:2001, bradlow-garcia-prada-gothen:2003, garcia-prada-mundet:2004,
  bradlow-garcia-prada-gothen:2005, 
  bradlow-garcia-prada-gothen:2012, gothen-oliveira:2012, GGM, wentworth-wilkin:2013} and also, most recently and from a different point of view in \cite{hitchin:2013, schaposnik:2013}.  In the paper
\cite{bradlow-garcia-prada-gothen:2005} we announced results on
$\SO(2,n)$ and $\SO^*(2n)$; the recent preprint \cite{hitchin-schaposnik:2013} addresses some aspects of the $\SO^*(2n)$ case from a different point of view.  Higgs bundles for all the groups of Hermitian symmetric type, including the two exceptional cases found among the real forms of $E_6$ and $E_7$, have also been studied in \cite{rubio:thesis}, where the first steps towards a unified treatment were taken.  The group $\SO(2,n)$ falls into the more
general (but not in general Hermitian symmetric) case of
$\SO(p,q)$-Higgs bundles, which were studied in \cite{aparicio,
  garcia-prada-aparicio}.  On the other side of the correspondence
between Higgs bundles and surface group representations, the groups of
Hermitian type have been extensively studied, notably recently in \cite{BIW-annals,guichard-wienhard:2010,ben-simon-burger-hartnick-iozzi:2013}.

In some ways the work described in this paper is one more in a series of case-by-case analyses of $G$-Higgs bundles for different $G$.  Adding to its interest, however, is the fact that the analysis of $\SO^*(2n)$-Higgs bundles unavoidably involves other reductive groups.  Any discussion of $\SO^*(2n)$-Higgs bundles is thus a showcase for several types of $G$-Higgs bundles.

The most direct way that other groups enter the picture is through the
structure of  polystable $\SO^*(2n)$-Higgs bundles.  In general (see Theorem
\ref{prop:SO-star-poly-stability-refined}) such Higgs bundles decompose as a
sum of  $G$-Higgs bundles where $G$ can be one of a number of different
groups, including $\SO^*(2m)$ for $m<n$, but also $\U^*(m), \U(p,q)$, and
$\U(m)$ for suitable values of $m,p,q$.  At the level of Lie theory, these are
the groups which appear as factors in Levi subgroups of $\SO(2n,\C)$
intersected with $\SO^*(2n)$.   Note that this list of groups includes both
compact and non-compact real forms. In the latter case the corresponding
symmetric space may be Hermitian or not.

The group $\U^*(m)$ appears in a second way that depends on a key
feature of $G$-Higgs bundles for non-compact real forms of Hermitian
type.  In these cases a discrete invariant known as the Toledo
invariant can be defined.  The invariant has several interpretations
(see
\cite{milnor,dupont,domic-toledo:1987,clerc-orsted:2003,bradlow-garcia-prada-gothen:2005,
  BIW-annals, hitchin-schaposnik:2013, schaposnik:2013}) but all lead
to a bound that generalizes the Milnor inequality on the Euler class
of flat $\SL(2,\R)$-bundles. The $G$-Higgs bundles with maximal Toledo
invariant all have special properties but these fall into two
categories, depending on whether the Hermitian symmetric space is of
tube type or not.  In the tube cases, a correspondence emerges between
polystable $G$-Higgs bundles with maximal Toledo invariant and objects
called $K^2$-twisted $G'$-Higgs bundles, where $G'$ is a new reductive
group.  We call this $G'$ $K^2$-twisted $G'$-Higgs bundle the Cayley
partner to the original $G$-Higgs bundle.  In the non-tube cases, the
maximal $G$-Higgs bundles do not have Cayley partners but decompose
into two parts, one of which has a Cayley partner and the other of
which corresponds to a compact group.  This imposes constraints which
we refer to as `rigidity' on the moduli spaces.  For $G=\SO^*(2n)$ we
see both types of phenomena, depending on whether $n$ is even or odd.
In the odd case, the group is not of tube type and we see rigidity
(see Section \ref{subs:rigidity}).  For $n=2m$, the group is of tube
type and the Cayley partner to $\SO^*(2n)$ is the group $\U^*(n)$.

There is one more group that enters the discussion, namely $\Sp(2n,\R)$.  While the nature of the relation between $\SO^*(2n)$-Higgs bundles and $\Sp(2n,\R)$-Higgs bundles is more subtle than in the case of the groups which appear in Levi subgroups,  the comparison between the two cases is instructive and unavoidable.  In both cases the maximal compact subgroups are isomorphic to $\U(n)$, and the complexified isotropy representations are
\begin{equation}
\begin{cases} \Lambda^2(\C^n)\oplus\Lambda^2((\C^n)^*)\ \phantom{(\C^n)^*)} \mathrm{for}\ \SO^*(2n)\\ 
\Sym^2(\C^n)\oplus \Sym^2((\C^n)^*)\  \mathrm{for}\ \Sp(2n,\R)
\end{cases}
\end{equation}

These structural similarities between  $\SO^*(2n)$ and $\Sp(2n,\R)$ carry over
to the theory of Higgs bundles. 
In both cases a $G$-Higgs bundle over a Riemann surface is defined by 
triple $(V,\beta,\gamma)$ where $V$ is a rank $n$ holomorphic bundle, and $\beta$ and $\gamma$ are homomorphisms
$$
\beta: V^*\lra V \otimes K \;\;\; \mbox{and}\;\;\; \gamma: V\lra V^*
\otimes K.
$$
The difference between the cases $G=\SO^*(2n)$ and $G=\Sp(2n,\R)$ is that
in the former case the maps $\beta$ and $\gamma$ are skew-symmetric, while in
the latter case the maps are symmetric.  However in both cases, the quadruple $(V,V^*,\beta,\gamma)$
defines a $\SU(n,n)$-Higgs bundle.\footnotemark\footnotetext{ This
  corresponds to the fact that both $\SO^*(2n)$ and $G=\Sp(2n,\R)$ embed as
  subgroups in $\SU(n,n)$} (see Section \ref{A:U(p,q)}, where
$\U(n,n)$-Higgs bundles $(V,W,\beta,\gamma)$ are defined. One has here the
extra condition $\det W={(\det V)}^{-1}$ since the group is $\SU(n,n)$).
Indeed both types of Higgs bundles appear in the moduli space of  
$\SU(n,n)$-Higgs bundles as fixed points of involutions, namely
$$(V,W,\beta,\gamma)\mapsto (W^*, V^*,\pm\beta^t,\pm\gamma^t).$$

The similarities between the two cases mean that many of the details
worked out in \cite{GGM} for $\Sp(2n,\R)$-Higgs bundles require only
minor modification in order to be applied to $\SO^*(2n)$-Higgs
bundles.  However, we believe that the
presentation in this paper naturally benefits from a more systematic
understanding of the theory.
Our main results show  that the
outcomes in the two cases are significantly different in at least two
respects. First, the parity of $n$ plays a role if $G=\SO^*(2n)$ (but
not if $G=\Sp(2n,\R)$), and second the moduli space of Higgs bundles
with maximal $\deg(V)$ has just one connected component if
$G=\SO^*(2n)$ but has several connected components distinguished by
`hidden' topological invariants revealed through the Cayley
correspondence in the case $G=\Sp(2n,\R)$. 

We now describe the contents of the paper in a bit more detail.  Let
$X$ be a Riemann surface of genus $g\ge 2$.  After some general
definitions in Section \ref{G-Higgs}, in Section~\ref{sect:SO*Higgs}
we describe the main features of the groups $\SO^*(2n)$ and
$\SO^*(2n)$-Higgs bundles.  We give structure results for stable and
polystable objects.  As in the case $G=\Sp(2n,\R)$, the moduli space
of polystable $\SO^*(2n)$-Higgs bundles, denoted by
$\mathcal{M}(\SO^*(2n))$, is not connected. The Toledo invariant,
which in the case of $\SO^*(2n)$-Higgs bundles corresponds to the the
degree of the bundle $V$, separates the moduli space into components
$\cM_d$ (where $d=\deg V$).  In Section \ref{sec:spn-higgs-moduli} we
establish the bounds on this invariant, namely
\begin{equation}
0\le |d|\le \lfloor\frac{n}{2}\rfloor(2g-2)\ .
\end{equation}
In Section \ref{sect: max} we study the case $d=\lfloor\frac{n}{2}\rfloor(2g-2)$ (the case $d=-\lfloor\frac{n}{2}\rfloor(2g-2)$ is analogous).  The special feature in this maximal situation is that the component 
$$\gamma :V\longrightarrow V^*\otimes K$$
of the Higgs field has maximal rank. Since $\gamma$ is skew-symmetric, this means that it defines a symplectic structure on either $V\otimes K^{-1/2}$ (if $n$ is even) or on a rank $n-1$ quotient of this (if $n$ is odd).  This leads to the Cayley correspondence we describe in Section \ref{subs:cayley} and to the rigidity result in Section \ref{subs:rigidity}.

The moduli spaces of Higgs bundles come equipped with a natural
function that can be used in a Morse-theoretic way to detect
topological properties.  First described by Hitchin
\cite{hitchin:1987}, this function measures the $L^2$-norm of the
Higgs field.  For each $d$, the function provides a proper map from
$\mathcal{M}_d$ to $\R$ and thus attains its minimum on each connected
component.  In Section \ref{sect:morse} we examine the minima and show
that they are precisely the polystable Higgs bundles in which
$\beta=0$ or $\gamma=0$ (depending on the sign of $d$). This reduces
the problem of the connectivity of the components to one of the
connectivity of the locus of minima. Unfortunately for most values of
$d$ this is itself a difficult problem.  The only exceptions are the
cases where $d=0$ or where $|d|$ has its maximum value. In Section
\ref{sect:morse} we also examine these exceptional cases and show the
following.

\begin{theorem} For $d=0$ or $|d|$ maximal, the components
$\mathcal{M}_d(\SO^*(2n))$  of the moduli space of polystable 
$\SO^*(2n)$-Higgs bundles are connected.
\end{theorem}

In Section \ref{sect: surfacegroups} we invoke the non-abelian Hodge theory
correspondence between the moduli space of $\SO^*(2n)$-Higgs bundles over $X$ 
 and the moduli space of representations
of the fundamental group of $X$ in $\SO^*(2n)$ to count the number of connected
components of the latter in the zero and maximal Toledo invariant cases, and
to give a rigidity result for maximal representations when $n$ is odd.  

In Section \ref{sect:lowrank} we examine some special features of $\SO^*(2n)$-Higgs bundles and their moduli spaces in the low rank cases, i.e. for $n=1,2,3$.  These features are mostly reflections of special low rank isomorphisms between Lie groups, but they yield interesting relations between Higgs bundle moduli spaces.  

Finally, in the Appendix we summarize salient features of $G$-Higgs bundles
for the groups other than $\SO^*(2n)$ 
which come up in the discussion of the case $G=\SO^*(2n)$.

We conclude this introduction by pointing out that a number of works which appeared after the first version of this paper was posted on the arXiv exploit an interesting complementary approach to $G$-Higgs bundles via
the Hitchin fibration. These include \cite{hitchin:2013,hitchin-schaposnik:2013,peon:2013,schaposnik:2013}.  The comparison between the two approaches is instructive and deserves further investigation. 

\subsubsection*{Acknowledgments}

The authors thank Olivier Biquard, Ignasi Mundet and Roberto Rubio for
useful discussions.  The authors also thank the following institutions for their
hospitality during various stages of this research: Centre for Quantum
Geometry of Moduli Spaces (Aarhus University), The Institute for
Mathematical Sciences (National University of Singapore), Centro de
Investigaci{\'o}n en Matem{\'a}ticas (Guanajuato) and the Centre de
Recerca Matem{\`a}tica (Barcelona), University of Illinois at Urbana-Champaign, Centro de Matem\'atica da Universidade do Porto, and the Instituto de Ciencias Matem\'aticas, Madrid.

\section{$G$-Higgs bundles}\label{G-Higgs}

The original notion of a $G$-Higgs bundle when $G$ is a real reductive Lie
group can be traced back to \cite{hitchin:1987,hitchin:1992}. For the
convenience of the reader we summarize the basic definitions and
constructions.  These have appeared at various levels of explicitness
in several places including \cite{bradlow-garcia-prada-gothen:2005,GGM,gothen:1995}.

\subsection{Moduli space of $G$-Higgs bundles}\label{moduli-ghiggs}

Let $G$ be a \textbf{real reductive Lie group}. By this we
mean\footnotemark\footnotetext{Our definition follows  Knapp
\cite[p.~384]{knapp:1996}, except that we do not impose the condition that for every $g\in G$ the automorphism $\Ad(g)$ of $\lieg^{\C}$ is inner, i.e. $\Ad(g)=\Ad(x)$ for some $x$ in the identity component of the adjoint form of $G$.  In fact this condition, which plays a role only if non-connected groups must be considered, is automatically satisfied by the groups which appear in  this paper.}
that we are given the data
$(G,H,\theta,B)$, where $H \subset G$ is a maximal compact subgroup, $\theta\colon \lieg
\to \lieg$ is a Cartan involution and $B$ is a
non-degenerate bilinear form on $\lieg$, which is $\Ad(G)$-invariant
and $\theta$-invariant.  The data $(G,H,\theta,B)$ has to satisfy in
addition that
\begin{enumerate}
\item
the Lie algebra $\lieg$ of $G$ is reductive,
\item
$\theta$ gives  a decomposition  (the Cartan decomposition)
\begin{displaymath}
\lie{g} = \lieh \oplus \liem
\end{displaymath}
into its $\pm1$-eigenspaces, where  $\lieh$ is the Lie
algebras of $H$,
\item
$\lieh$ and $\liem$ are orthogonal under $B$, and $B$ is positive definite
on $\liem$ and negative definite on $\lieh$,
\item
multiplication as a map from $H\times \exp\liem$ into $G$ is an onto
diffeomorphism.
\end{enumerate}
We will refer sometimes to the data $(G,H,\theta,B)$ as the 
{\bf Cartan data}.

The group $H$ acts linearly on $\mlie$ through the adjoint
representation of $G$.  Complexifying, we get the \textbf{isotropy
  representation}
\begin{math}
  \iota\colon \HC \to \GL(\mclie).    
\end{math}

\begin{definition}
\label{def:g-higgs}
A \textbf{$G$-Higgs bundle} on $X$ is a pair
$(E,\varphi)$, where $E$ is a holomorphic $\HC$-principal bundle
over $X$ and $\varphi$ is a holomorphic section of $E(\mclie)\otimes
K$,  where $E(\mclie)= E \times_{\HC}\mclie$ is the
$\mclie$-bundle associated to $E$ via the isotropy representation and $K$ is the canonical bundle of $X$.  The section $\varphi$ is called the {\bf Higgs field}.  Two $G$-Higgs bundles $(E,\varphi)$ and
$(E',\varphi')$ are \textbf{isomorphic} if there is an isomorphism
$f\colon E \xra{\simeq} E'$ such that $\varphi = f^*\varphi'$ where $f^*$ is the obvious induced map.
\end{definition}

More generally, replacing $K$ by an arbitrary line bundle on $X$ in
the preceding definition, we obtain the notion of a
\textbf{$L$-twisted $G$-Higgs pair} on $X$.

Just as for vector bundles, there are notions of stability,
semistability and polystability for $G$-Higgs bundles (and more
generally for $L$-twisted Higgs pairs).  In this paper we consider
only the particular cases we need (cf.\ Section~\ref{sec:spnr-higgs})
and refer the reader to \cite{garcia-prada-gothen-mundet:2009a} for
the general definitions.

Henceforth, we shall assume that $G$ is connected. Then the
topological classification of $H^\CC$-bundles $E$ on $X$ is given by a
characteristic class 
$$c(E)\in \pi_1(H^\CC)=\pi_1(H)=\pi_1(G)\ .$$

\begin{definition}\label{Defn: M_d}
For a
fixed $d \in \pi_1(G)$, the \textbf{moduli space of polystable
$G$-Higgs bundles} $\mathcal{M}_d(G)$ is the set of isomorphism
classes of polystable $G$-Higgs bundles $(E,\varphi)$ such that
$c(E)=d$.  
\end{definition}

The moduli space $\cM_d(G)$ has the structure of a complex analytic
variety.  This can be seen by the standard slice method (see, e.g.,
Kobayashi \cite{kobayashi:1987}). Moreover, it is a consequence of the
general constructions of Schmitt \cite{schmitt:2005,schmitt:2008} that
for all the groups which appear in this paper the moduli space
$\cM_d(G)$ is actually algebraic.

\subsection{The Hitchin equation}\label{sec:hitchin-equations}

In general, i.e.\ for any real reductive group $G$, the Hitchin
equations for a $G$-Higgs bundle, say $(E,\varphi)$, can be regarded
as conditions for a reduction of the structure group of $E$. Recall
that $E$ is a principal holomorphic $H^{\C}$-bundle, where $H^{\C}$ is
the complexification of $H$ (a maximal compact subgroup of $G$).  A
reduction of structure group to $H$ defines a principal $H$-bundle,
$E_H$, such that $E=E_H\times_HH^{\C}$.  Then, together with the
holomorphic structure on $E$, the reduction to $E_H$ defines a unique
connection (the Chern connection) on $E$.  We denote the curvature of
this connection by $F_h$.  Assume now that $G$ is a real form of its
complexification $G^{\C}$, and let $\tau:\liegc\longrightarrow\liegc$
denote the involution which defines the compact real form of $G^{\C}$.
The relation between $\tau$, the involution which defines the real
form $G$, and the Cartan involution on $\lieg$, ensures that the
combination $[\mathfrak{m}^{\C},\tau(\mathfrak{m}^{\C})]$ takes values
in $\lieh$.  Using the reduction $E(\liegc)=E_H\times_H\liegc$ we can
extend $\tau$ to a bundle map $\tau_h: E(\liegc)\longrightarrow
E(\liegc)$. Combined with conjugation on the canonical bundle $K$ this
defines a bundle map (also denoted by $\tau_h$) on $E(\liegc)\otimes
K$.  Applying this map to the Higgs field $\varphi$ allows us to form
a $\lieh$-valued (1,1)-form $[\varphi,\tau(\varphi)]$.

\begin{definition}\label{GHitchin}
If $G$ is semisimple the {\bf $G$-Hitchin equation} for a reduction of structure group
to $H$ of  a $G$-Higgs bundle $(E,\varphi)$ is 
\begin{equation}\label{eqn:GHitch}
F_h-[\varphi,\tau_h(\varphi)]= 0
\end{equation}
\noi where $F_h$ and $\tau_h$ are as above.
\end{definition}

The following result can be found in \cite{garcia-prada-gothen-mundet:2009a}.

\begin{theorem}[Theorem 3.21 in \cite{garcia-prada-gothen-mundet:2009a}]
  \label{HKcorrespond}
  Let $(E,\varphi)$ be a $G$-Higgs bundle.  The bundle $E$ admits a
  reduction of structure group from $H^{\C}$ to $H$ satisfying the
  Hitchin equation for a $G$-Higgs bundle if and only if $(E,\varphi)$
  is polystable.
\end{theorem}

\subsection{Deformation theory of $G$-Higgs bundles}
\label{sec:deformation-theory}

In this section we recall some standard facts about the deformation
theory of $G$-Higgs bundles (see \cite{GGM} and
\cite{garcia-prada-gothen-mundet:2009a} for more detail). We also take
care of issues that arise considering general reductive groups, rather
than just semisimple ones. In particular we introduce a reduced
deformation complex which is relevant in analyzing smoothness of the
moduli space.

\begin{definition}\label{def:def-complex}
Let $(E,\varphi)$ be a $G$-Higgs bundle. 
Let $d\iota\colon \lieh^{\CC} \to \End(\liem^{\CC})$ be the derivative
at the identity of the complexified isotropy representation $\iota =
\Ad_{|H^{\CC}}\colon H^{\CC} \to \Aut(\liem^{\CC})$.
The \textbf{deformation complex}
of $(E,\varphi)$ is the following complex of sheaves:
\begin{equation}\label{eq:def-complex}
C^{\bullet}(E,\varphi)\colon E(\lieh^\CC) \xrightarrow{d\iota(\varphi)}
E(\liem^\CC)\otimes K.
\end{equation}
\end{definition}

\begin{proposition}
\label{prop:deform}
The space of infinitesimal deformations of a $G$-Higgs bundle
$(E,\varphi)$ is naturally isomorphic to the hypercohomology group
$\HH^1(C^{\bullet}(E,\varphi))$. The Lie algebra of $\Aut(E,\varphi)$,
denoted by $\aut(E,\varphi)$, can be identified with
$\HH^0(C^{\bullet}(E,\varphi))$.
\end{proposition}




Next we introduce two concepts which are important for understanding
smoothness of the moduli space (cf.\ Proposition~\ref{prop:smoothpoint} below).

\begin{definition}
\label{def:simple}
 A $G$-Higgs bundle $(E,\varphi)$ is called {\bf simple} if 
 \begin{math}
 \mathrm{Aut}(E,\varphi)=Z(H^{\C})\cap \ker(\iota)
\end{math}
where $Z(H^{\C})$ denotes the center.
A $G$-Higgs bundle $(E,\varphi)$ is said to be
\textbf{infinitesimally simple} if the infinitesimal automorphism
space $\aut(E,\varphi)$ is isomorphic to
$ H^0(X,E(\ker d\iota \cap Z(\lieh^{\C}))$ where  $Z(\lieh^{\C})$ denotes the Lie algebra of $Z(H^{\C})$.
\end{definition}

Thus a $G$-Higgs bundle is (infinitesimally) simple if its
(infinitesimal) automorphism group is
as small as possible.

\begin{remark} It is clear that a simple $G$-Higgs bundle is
  infinitesimally simple.  If $G$ is complex then $\iota$ is the
  adjoint representation and $(E,\varphi)$ is simple
  (resp.\ infinitesimally simple) if $\Aut(E,\varphi)=Z(G)$
  (resp.\ $\aut(E,\varphi)=Z(\lieh^{\C})$).
\end{remark}

\begin{example}
  View a $\GL(n,\C)$-Higgs bundle a Higgs vector bundle $(E,\Phi)$
  with $\Phi\in H^0(X,\End(E)\otimes K)$. Then $(E,\Phi)$ is simple if
  its automorphism group is $\Aut(E,\Phi)=\C^*$ and infinitesimally
  simple if its infinitesimal automorphism space $\End(E,\Phi) =
  \C$. In this case the two notions coincide, but this is not the case
  for all groups. Indeed, this phenomenon already occurs for principal
  bundles without a Higgs field: as an example, let $L$ be a line
  bundle of degree zero such that $L^2 \neq \mathcal{O}$. Then the
  $\SO(2,\C)$-bundle $(V,Q)=(L \oplus L^{-1},\left(
  \begin{smallmatrix}
    0 & 1 \\
    1 & 0
  \end{smallmatrix}\right))$
  has $\Aut(V,Q) = \{\pm 1\}$ and $\aut(V,Q) = 0$ so it is
  infinitesimally simple but not simple.
\end{example}

In order to study smoothness of the moduli space in the general case
of reductive groups (i.e.\ for non-semisimple $G$), we introduce a reduced
deformation complex.

\begin{lemma}[{\cite[p.~388]{knapp:1996}}]
  Let $\liez$ be the center of $\lieg$ and $\liez^{\C}$ be the center
  of $\lieg^{\C}$.  There are decompositions
  \begin{math}
    \liez = (\lieh \cap \liez) \oplus (\liem \cap \liez)
  \end{math}
  and 
\begin{math}
    \liez^{\C} = (\lieh^{\C} \cap \liez^{\C}) \oplus (\liem^{\C} \cap \liez^{\C}).
  \end{math}
\end{lemma}

In view of this Lemma, we can decompose as $H$-modules
\begin{displaymath}
  \lieh = (\lieh \cap \liez) \oplus \lieh_0, \qquad
  \liem = (\liem \cap \liez) \oplus \liem_0,
\end{displaymath}
where we have defined
\begin{displaymath}
  \lieh_0 = \lieh / (\lieh \cap \liez), \qquad
  \liem_0 = \liem / (\liem \cap \liez).
\end{displaymath}
Analogously we define $\lieh_0^{\C}$ and $\liem_0^{\C}$ and we have
similar decompositions of $\lieh^{\C}$ and $\liem^{\C}$.
Note also that
\begin{displaymath}
  [\liem^{\C},\lieh^{\C} \cap \liez^{\C}] = 0, \qquad
  [\liem^{\C}_0,\lieh^{\C}_0] \subset \liem^{\C}_0.
\end{displaymath}
We can thus define the following reduced complex.

\begin{definition}\label{def:red-def-complex}
  Let $(E,\varphi)$ be a $G$-Higgs bundle.  The {\bf reduced
    deformation complex} of $(E,\varphi)$ is the following complex of
  sheaves:
\begin{equation}\label{eq:red-def-complex}
C^{\bullet}_0(E,\varphi)\colon E(\lieh_0^\CC) \xrightarrow{\ad(\varphi)}
E(\liem_0^\CC)\otimes K.
\end{equation}
\end{definition}

\begin{remark} If $G$ is semisimple the reduced deformation complex
  (\ref{eq:red-def-complex}) coincides with the non-reduced complex
  (\ref{eq:def-complex}).  If $G$ is a complex reductive group, then
  the reduced complex $C^{\bullet}_0(E,\varphi)$ can be identified
  with the (non-reduced) deformation complex for the $PG$-Higgs bundle
  associated to $(E,\varphi)$, where $PG=G/Z(G)$.
\end{remark}

Let $(E,\varphi)$ be a $G$-Higgs bundle and assume that $G$ is a real form of a complex reductive group $G_{\C}$. Let
\begin{displaymath}
  \tilde{E} = E \times_{H^{\C}}G^{\C}
\end{displaymath}
be the principal $G^{\C}$-bundle associated by extension of structure
group. Note that
\begin{displaymath}
  \tilde{E}(\lieg^{\C}) = E(\lieg^{\C}) = E(\lieh^{\C}) \oplus E(\liem^{\C}).
\end{displaymath}
Hence we can let $\tilde\varphi$ be the image of $\varphi$ under the inclusion
\begin{displaymath}
  H^0(X,E(\liem^{\C})\otimes K) \into H^0(X,\tilde{E}(\lieg^{\C})\otimes K).
\end{displaymath}

\begin{definition} 
The $G^{\C}$-Higgs bundle $(\tilde{E},\tilde{\varphi})$  is called the {\bf $G^{\C}$-Higgs bundle associated to the $G$-Higgs bundle} $(E,\varphi)$.
\end{definition}

\begin{proposition}\label{prop:smoothpoint} Let $(E,\varphi)$ be a $G$-Higgs bundle.

\begin{enumerate}
\item If $(E, \varphi)$ is stable and $\varphi\ne 0$ then it is
  infinitesimally simple.
\item If $(E,\varphi)$ is stable and simple and
  $\HH^2(C^{\bullet}_0(E,\varphi))=0$ then $(E,\varphi)$ represents a
  smooth point in the moduli space.
\item If $G$ is complex and $(E,\varphi)$ is stable and simple then
  $(E,\varphi)$ represents a smooth point in the moduli space.
\item Let $(\tilde E,\tilde\varphi)$ be the $G^{\C}$-Higgs bundle associated to $(E,\varphi)$.  If $(E,\varphi)$ is stable then $(\tilde
  E,\tilde\varphi)$ is polystable. If $(E,\varphi)$ is stable, simple
  and stable as a $G^\CC$-Higgs bundle then it represents a smooth
  point in the moduli space.
\end{enumerate}
\end{proposition}

\begin{proof}
(1) See \cite[Proposition~3.11]{garcia-prada-gothen-mundet:2009a}.

(2) If $(E,\varphi)$ is simple, there are no singularities coming from
automorphisms of the pair. Therefore the obstruction to smoothness
lies in
 $\HH^2(C^{\bullet}(E,\varphi))$. Analyzing the Kuranishi
model (as done in Kobayashi \cite{kobayashi:1987} in the case of
vector bundles on higher dimensional manifolds, cf.\ also
Friedman--Morgan \cite[p.~301]{friedman-morgan:1994}), one sees that
the image of the Kuranishi map in fact lies in 
the hypercohomolgy of the reduced deformation complex, i.e., in
$\HH^2(C^{\bullet}_0(E,\varphi))=0$. The point is that the Kuranishi
map is given by the \emph{quadratic} part of the holomorphicity
condition
\begin{displaymath}
  0 = \dbar_{A+\dot{A}}(\varphi+\dot\varphi)
  = \dbar_A\varphi + \dbar_A\dot{\varphi} + [\dot{A},\varphi]
    + [\dot{A},\dot\varphi],
\end{displaymath}
which lies in $\Omega^{0,1}E(\liem_0^\CC)$.
This leads to the result. (An
alternative method of proof would be to go through the proof of
Theorem~3.1 of \cite{biswas-ramanan:1994} and see that the vanishing
of $\HH^2(C^{\bullet}_0(E,\varphi))=0$ is really what is required in
this case.)

(3) By stability we have the vanishing
$\HH^0(C^{\bullet}_0(E,\varphi))=0$ and Serre duality of complexes
implies $\HH^2(C^{\bullet}_0(E,\varphi))=0$. The result now follows by
(2). 

(4) Stability of $(\tilde{E},\tilde\varphi)$ implies that
it is infinitesimally simple, i.e.,
$\HH^0(C^{\bullet}(\tilde{E},\tilde\varphi)) = Z(\liegc)$, where 
\begin{displaymath}
  C^{\bullet}(\tilde{E},\tilde\varphi)\colon
  \tilde{E}(\lieg^{\C})
  \xra{\ad(\tilde\varphi)}\tilde{E}(\liegc)\otimes K.
\end{displaymath}
It follows that
\begin{math}
  \HH^0(C^{\bullet}_0(\tilde{E},\tilde\varphi)) = 0.
\end{math}
Moreover,
\begin{displaymath}
  C^{\bullet}_0(\tilde{E},\tilde\varphi) =
  C^{\bullet}_0(E,\varphi) \oplus C^{\bullet}_0(E,\varphi)^*\otimes K
\end{displaymath}
and hence, by Serre duality of complexes, we obtain the vanishing
\begin{math}
  \HH^2(C^{\bullet}_0(E,\varphi))=0.
\end{math}
Again the result is now a consequence of (2). 
\end{proof}

\section{$\SO^*(2n)$-Higgs bundles}\label{sect:SO*Higgs}

\subsection{Preliminaries: the group $\SO^*(2n)$}
\label{sec:prel-group-so2n}

In this section we collect together some basic facts about the group $\SO^*(2n)$ (see \cite{helgason} for more details).  We concentrate on the features that are needed to describe $\SO^*(2n)$-Higgs bundles and to understand their relation to $G$-Higgs bundles for related groups such as $\SL(2n,\C)$ and $\SU(n,n)$.  The group $\SO^*(2n)$ may be defined as the the set of matrices $g\in\SL(2n,\C)$ satisfying 
\begin{equation}\label{eqn:SO*defn}
g^tJ_n\bar{g}=J_n\ \mathrm{and} \ g^tg=I_{2n}\ , \mathrm{where}\ J_n=\begin{pmatrix}0&I_n\\-I_n&0\end{pmatrix}.
\end{equation}
It is thus a subgroup of $\SO(2n,\C)$ which leaves invariant a
skew-Hermitian form.  The group is connected, semisimple, and a
non-compact real form of $\SO(2n,\C)$.  The maximal compact subgroups
are isomorphic to $\U(n)$. The choice
\begin{math}
\Theta (g)=J_ngJ^{-1}_n
\end{math}
of Cartan involution on $\SO^*(2n)$ gives the  Cartan decomposition 
$$\mathfrak{so}^*(2n)=\mathfrak{u}(n)+\liem$$
with
\begin{equation}\label{cartan}
\begin{aligned}\mathfrak{u}(n)&=\left\{\begin{pmatrix}X_1&X_2\\ -X_2&X_1
\end{pmatrix} \ | \  X_1,X_2\in \Mat_{n,n}(\R), X_1^t+X_1=0,  X_2^t-X_2=0\right\},\\
\liem &= \left\{i\begin{pmatrix}Y_1&Y_2\\ Y_2&-Y_1\end{pmatrix} \ | \  Y_1,Y_2\in \Mat_{n,n}(\R),  Y_1^t+Y_1=0, Y_2^t+Y_2=0\right\}.
\end{aligned}
\end{equation}

\begin{remark}  It follows immediately from \eqref{cartan} that
\begin{equation*}
\mathfrak{u}(n)+i\liem =\left\{\begin{pmatrix}A&B\\ -B^t&D\end{pmatrix} \ |\ A,B,D\in  \Mat_{n,n}(\C),\ A+A^t=D+D^t=0\right\}  ,
\end{equation*}
 which can be identified with the Lie algebra of $\SO(2n)$.  This shows that the real form $\SO^*(2n)$ is the non-compact dual to the compact real form $\SO(2n)\subset\SO(2n,\C)$.
\end{remark}


The complexification of the Cartan decomposition is
\begin{equation}
\mathfrak{so}^*(2n)\otimes\C=\mathfrak{gl}(n,\C)+\liemc,
\end{equation}
\noi where
\begin{equation}
\begin{aligned}
\mathfrak{gl}(n,\C)&=\left\{\begin{pmatrix}\frac{Z-Z^t}{2}&-\frac{Z+Z^t}{2i}\\\frac{Z+Z^t}{2i}&\frac{Z-Z^t}{2}\end{pmatrix} \ |\ Z \in \Mat_{n,n}(\C)\right\},\\
\liemc &= \left\{\begin{pmatrix}Y_1&Y_2\\ Y_2&-Y_1\end{pmatrix} \ | \ Y_1,Y_2\in \Mat_{n,n}(\C),  Y_1^t+Y_1=0, Y_2^t+Y_2=0\right\}.
\end{aligned}
\end{equation}
%
\noi It follows that if
$T$ is the complex automorphism of $\C^{2n}$ defined by
\begin{equation}
\label{eq:def-T}
T=\begin{pmatrix}I&iI\\I&-iI\end{pmatrix},
\end{equation}
then
\begin{equation}\label{Ccartan}
\begin{aligned}
&T\mathfrak{gl}(n,\C)T^{-1}=\left\{\begin{pmatrix} Z&0\\0&-Z^t\end{pmatrix}
  \ | \ Z \in \Mat_{n,n}(\C)\right\},\\
&T\liemc T^{-1}= \left\{\begin{pmatrix}0&\beta\\ \gamma&0\end{pmatrix} \ |
  \  \beta,\gamma\in \Mat_{n,n}(\C),\  \beta^t+\beta=0,\ \gamma^t+\gamma=0\right\}.
\end{aligned}
\end{equation}
\noi This reflects the following fact. 

\begin{proposition}\label{prop:SUnn} With $T$ defined in \eqref{eq:def-T}, 
\begin{equation}\label{SU(n,n) embed}
T\SO^*(2n) T^{-1}\subset\SU(n,n), 
\end{equation}
\noi where $\SU(n,n)\subset\SL(2n,\C)$ is the subgroup defined by 
\begin{equation}
\SU(n,n)=\{ A\in\SL(2n,\C)\ |\  \bar{A}^tI_{n,n}A=I_{n,n},\ \mathrm{\det}(A)=1\}\ 
\end{equation}
\noi with $I_{n,n}=\begin{pmatrix}I_n&0\\0&-I_n\end{pmatrix}$.

\end{proposition}

\begin{proof} Using $\bar{T}^tI_{n,n}T=2iJ$ it follows that if $g\in\SO^*(2n)$ then $A=TgT^{-1}$ satisfies $\bar{A}^tI_{n,n}A=I_{n,n}$. Also $\det(A)=\det(g)=1$.
\end{proof}

\begin{remark} If $g\in\SO^*(2n)$, i.e.\ $g$ satisfies \eqref{eqn:SO*defn}, and $A=TgT^{-1}$ then a simple calculation shows that $A^tI_{n,n}JA=I_{n,n}J$. Combined with Proposition \ref{prop:SUnn} we can thus identify $\SO^*(2n)\subset\SU(n,n)\subset\SL(2n,\C)$ as the subgroup defined by the relation $A^tI_{n,n}JA=I_{n,n}J$. This is the definition given in \cite{knapp:1996}.
\end{remark}

\subsection{$\SO^*(2n)$-Higgs bundles and stability}
\label{sec:spnr-higgs}

When $H^\CC$ is a classical group we prefer to work with the vector
bundle $V$ associated to the standard representation rather than the
$H^\CC$-principal bundle. We take this point of view for
$\SO^*(2n)$-Higgs bundles, for which $H^{\C}=\GL(n,\C)$ and $V$ is a
rank $n$ vector bundle. In view of (\ref{Ccartan}),
Definition~\ref{def:g-higgs} then becomes the following.

\begin{definition}\label{defn:SO*Higgs}  
  A {\bf $\SO^*(2n)$-Higgs bundle} over $X$ is a pair $(V,\varphi)$ in
  which $V$ is a rank $n$ holomorphic vector bundle over $X$, and the
  Higgs field $\varphi=(\beta,\gamma)$ has components $\beta\in
  H^0(X,\Lambda^2V\otimes K)$ and $\gamma\in H^0(X,\Lambda^2V^*\otimes
  K)$.  We will sometimes write $\varphi=\beta+\gamma$, where the sum
  is interpreted as being in $\End(V\oplus
  V^*)\otimes K$, viewing $\beta$ and $\gamma$ as skew-symmetric maps $\beta\colon V^*\to V\otimes K$ and $\gamma\colon V\to
  V^*\otimes L$. We will also sometimes use the notation
  $(V,\varphi)=(V,\beta,\gamma)$.
\end{definition}

In order to  state the
(semi,poly)stability condition for a
$\SO^*(2n)$-Higgs bundle we need to introduce some notation.

Let $V \to X$ be a holomorphic vector bundle. Then there is an
isomorphism $V \otimes V \simeq \Lambda^2 V \oplus S^2 V$.
Let $U$ and $W$ be subbundles of $V$. We define $U \otimes_{A} W$ to
be the sheaf theoretic kernel of the projection $V \otimes V \to S^2V$
restricted to $ U \otimes V$:
\begin{displaymath}
0 \to U \otimes_A W \to U \otimes W \to S^2 V.
\end{displaymath}
Since $U \otimes W$ is locally free and $X$ is a curve, $U \otimes_A
W$ can be viewed as a subbundle of $\Lambda^2 V$.
We define $U^{\perp} \subset V^*$ to be the kernel of the
restriction map $V^* \to U^*$, i.e.
\begin{displaymath}
0 \to U^{\perp} \to V^* \to U^*\to 0.
\end{displaymath}

\begin{definition}\label{defn:filtrations}
  Let $k$ be an integer satisfying $k\geq 1$. We define a
  \textbf{filtration of} $V$ \textbf{of length} $k-1$ to be any
  strictly increasing filtration by holomorphic subbundles
\begin{displaymath}
\VVV=(0 \subsetneq V_1\subsetneq V_2\subsetneq\dots\subsetneq
V_k=V).
\end{displaymath}

Let $\lambda=(\lambda_1<\lambda_2 <\dots <\lambda_k)$ be a strictly increasing
sequence of $k$ real numbers. Define the subbundle
\begin{equation}
\label{eq:NV-lambda-def}
N(\VVV,\lambda)=
\sum_{\lambda_i+\lambda_j\leq 0} K\otimes V_i\otimes_A V_j \oplus
\sum_{\lambda_i+\lambda_j\geq 0} K\otimes V_{i-1}^{\perp}\otimes_A
V_{j-1}^{\perp} \subset K\otimes (\Lambda^2V\oplus \Lambda^2V^*).
\end{equation}
Define also
\begin{equation}
\label{eq:dV-lambda-alpha}
d(\VVV,\lambda)=
\lambda_k\deg V_k
+\sum_{j=1}^{k-1}(\lambda_j-\lambda_{j+1})\deg V_j.
\end{equation}

We say that the pair $(\VVV,\lambda)$ is \textbf{trivial} if the length of
$\VVV$ is $0$ and $\lambda_1=0$.
We say that the pair $(\VVV,\lambda)$ is  $\varphi$-\textbf{invariant}
if  $\varphi=\beta+\gamma\in H^0(X,N(\VVV,\lambda))$.

\end{definition}

The general results of \cite{garcia-prada-gothen-mundet:2009a} allow
us to express stability, semistability and polystability for
$\SO^*(2n)$-Higgs bundles in terms of filtrations, as follows.

\begin{definition}
\label{prop:sp2n-alpha-stability}
The Higgs bundle $(V,\varphi)$ is \textbf{semistable} if for 
any integer $k\geq 1$, any filtration $\VVV$ of length $k-1$ of $V$ and
any strictly increasing sequence $\lambda$ of $k$ real numbers  such that
$(\VVV,\lambda)$ is $\varphi$-invariant  we have 
\begin{equation}
d(\VVV,\lambda)\geq 0. \label{eq:des}
\end{equation}

The Higgs bundle $(V,\varphi)$ is \textbf{stable} if under the same conditions
as above with the additional condition that $(\VVV,\lambda)$ be non-trivial
we have the strict  inequality 
\begin{equation}
d(\VVV,\lambda)> 0. 
\end{equation}


The Higgs bundle $(V,\varphi)$ is \textbf{polystable} if it is semistable and
for any integer $k\geq 1$, any filtration $\VVV$ of length $k-1$ of $V$ and
any
strictly increasing sequence $\lambda$ of $k$ real numbers such that
$(\VVV,\varphi)$ is $\varphi$-invariant  and  $d(\VVV,\lambda)=0$ there is
an isomorphism of holomorphic bundles
$$V\simeq V_1\oplus V_2/V_1\oplus\dots\oplus V_k/V_{k-1}$$
with respect to which
$$\beta\in H^0(X,\bigoplus_{\lambda_i+\lambda_j=0}K\otimes
V_i/V_{i-1}\otimes_A V_j/V_{j-1})$$ and
$$\gamma\in H^0(X,\bigoplus_{\lambda_i+\lambda_j=0}K\otimes
(V_i/V_{i-1})^*\otimes_A (V_j/V_{j-1})^*).$$
We follow the convention that a direct sum of vector bundles over an
empty indexing set is the zero vector bundle. 
\end{definition}

\begin{remark}\label{remark-parameter}
In general the notion of (semi,poly)stability depend on a real parameter
related to the fact that the center of the maximal compact subgroup of 
$\SO^*(2n)$ is isomorphic to $\U(1)$ (see \cite{GGM}). However, since
our main interest is in relation to representations of the fundamental
group, we have the value of this parameter to be zero. 
\end{remark}

Following the same arguments given in
\cite{garcia-prada-gothen-mundet:2009a} for the group $\Sp(2n,\R)$,
the stability conditions for $\SO^*(2n)$-Higgs bundles can be
simplified. Before we give the simplified conditions are given
in Proposition~\ref{prop:simplifiedss} we need some preliminaries.

\begin{definition}\label{prop:SO-star-semi-stability} 
Let $(V,\varphi)$ be a $\SO^*(2n)$-Higgs bundle with $\varphi=(\beta,\gamma)$.  A filtration of subbundles
\begin{displaymath}\label{eqn:filtr}
0 \subset V_1 \subset V_2 \subset V
\end{displaymath}
such that
\begin{equation} \label{invariance}
\beta\in H^0(X,K\otimes(\Lambda^2V_2 + V_1\otimes_{A} V)),
\quad \gamma\in H^0(X,K\otimes(\Lambda^2V_1^{\perp} + V_2^{\perp}\otimes_A V^*)),
\end{equation}
is called a {\bf $\varphi$-invariant two-step filtration}. 
\end{definition}

\begin{remark}
It is important to note that the summands in the bundles $\Lambda^2V_2
+ V_1\otimes_{A}V$ and $\Lambda^2V_1^{\perp} + V_2^{\perp}\otimes_A
V^*$ intersect non-trivially, so they do not form a direct sum.
\end{remark}

\begin{remark}
We allow equality between the terms of the filtration in order to avoid having to consider separately 
filtrations that are length one or zero.  For example the filtration $0\subset V_1\subset V$ is included as 
the two-step filtration in which $V_1=V_2$.
\end{remark}

\noi It is sometimes convenient to reformulate the $\varphi$-invariance
condition using the following lemma, which is easily proved.

\begin{lemma} \label{lemma:2-step}
Let $(V,\varphi)$ be a $\SO^*(2n)$-Higgs bundle with $\varphi=(\beta,\gamma)$.  A
two-step filtration 
\begin{math}
0 \subset V_1 \subset V_2 \subset V
\end{math}
is $\varphi$-invariant if and only if the following conditions are satisfied:
\begin{align*}
\beta (V_2^{\perp})&\subset V_1\otimes K,
&\gamma (V_2)&\subset V_1^{\perp}\otimes K,\\
\beta (V_1^{\perp})&\subset V_2\otimes K,
&\gamma (V_1)&\subset V_2^{\perp}\otimes K.
\end{align*}
\end{lemma}

There is yet another useful interpretation of the $\varphi$-invariance of
a two step filtration  that will be used later. To explain this, 
let $\Omega_\gamma:V\times V\to K$
be the $K$-twisted skew-symmetric bilinear pairing defined by $\gamma$ as
$$
\Omega_{\gamma}(u, v):=(\gamma(v))(u), \;\;\mbox{for}\;\; u,v\in V,
$$ 
and denote, for a subbundle  $V'\subset V$, 
$$
V'^{\perp_\gamma}:=\{v\in V\;\;|\;\; \Omega_\gamma(u,v)=0\;\;\mbox{for every}\;\;
u\in V'\}.
$$

\noi The following lemma is immediate.

\begin{lemma}\label{K-pairing}
For any  filtration $0\subset V_1\subset V_2\subset V$,
we have that  
$\gamma(V_1)\subset K\otimes V_2^\perp$ is equivalent to 
$V_1\subset V_2^{\perp_\gamma}$. This  is   equivalent
to $V_2\subset V_1^{\perp_\gamma}$ which, in turn,  is equivalent
to $\gamma(V_2)\subset K\otimes V_1^\perp$.
Similar statements apply to $\beta$.
\end{lemma}

\noi The following  simplified version of the stability conditions  
follows in the
same way as the analogous results for $\Sp(2n,\R)$-Higgs bundles 
(see \cite{garcia-prada-gothen-mundet:2009a}).

\begin{proposition}\label{prop:simplifiedss}
A $\SO^*(2n)$-Higgs bundle $(V,\varphi)$ with $\varphi=(\beta,\gamma)$ is semistable if and only 
for every $\varphi$-invariant two-step filtration 
$0 \subset V_1 \subset V_2 \subset V$ we have that
\begin{equation}
\label{eq:spn-simplified-sstab}
\deg(V) - \deg(V_1) - \deg(V_2) \geq 0.
\end{equation}

A $\SO^*(2n)$-Higgs bundle $(V,\varphi)$ with $\varphi=(\beta,\gamma)$ is stable if and
only if for every $\varphi$-invariant two-step filtration 
$0 \subset V_1 \subset V_2 \subset V$ except the 
filtration $0 =V_1 \subset V_2 = V$ we have that
\begin{equation}
\label{eq:spn-simplified-stab}
\deg(V) - \deg(V_1) - \deg(V_2) > 0.
\end{equation}

A $\SO^*(2n)$-Higgs bundle $(V,\varphi)$ with $\varphi=(\beta,\gamma)$ is polystable if is semistable
and  for any $\varphi$-invariant filtration 
$0 \subset V_1 \subset V_2 \subset V$,
distinct from the filtration $0=V_1\subset V_2=V$
such that
$$
  \deg(V) - \deg(V_1) - \deg(V_2) = 0,
$$
there exists an isomorphism of holomorphic vector bundles
$$
 V\simeq  V_1\oplus V_2/V_1\oplus V/V_2
$$
with respect to which we have:

\begin{itemize}
\item[(a)] $V_2\simeq V_1\oplus V_2/V_1$,

\item [(b)] $ \beta\in H^0(X,K\otimes( 
\Lambda^2(V_2/V_1)\oplus 
V_1\otimes_A(V/V_2))$,

\item [(c)] $ \gamma\in H^0(X,K\otimes(\Lambda^2(V_2/V_1)^*\oplus 
V_1^*\otimes_A (V/V_2)^*)$.
\end{itemize}
\end{proposition}

\begin{remark} \label{remark: zero}If $\beta=\gamma=0$ then the semistability
  condition is  equivalent to the requirements that $\deg V =0$ and $V$ is semistable.
\end{remark}


\subsection{The $\SO^*(2n)$-Hitchin equations}\label{hitcheq}
Using the vector bundle picture, in which a $\SO^*(2n)$-Higgs bundle
is specified by data $(V,\beta,\gamma)$, we now make explicit the
Hitchin equations in this case.  A reduction of structure group to
$H=\U(n)$ corresponds to a choice of Hermitian metric $h$ on
the holomorphic bundle $V$.  The Hitchin equations now become
\begin{equation}\label{eqn:SO*Hitch}
F^h_V+\beta\beta^*+\gamma^*\gamma = 0.
\end{equation}
Here we denote the curvature for the Chern connection on $V$ by
$F_V^h$ and the adjoints are with respect to the hermitian metric $h$
(combined with complex conjugation $dz\mapsto d\bar{z}$ on the form
component).

We refer to equation \eqref{eqn:SO*Hitch} as the
\textbf{$\SO^*(2n)$-Hitchin equation}.  Theorem \ref{HKcorrespond}
thus becomes the following.

\begin{theorem}\label{HKforSO*}
Let $(V,\beta,\gamma)$ be a $\SO^*(2n)$-Higgs bundle.  The bundle $V$ admits a metric satisfying the $\SO^*(2n)$-Hitchin equation  \eqref{eqn:SO*Hitch}
if and only if $(V,\beta,\gamma)$ is polystable.
\end{theorem}

\subsection{The moduli spaces}\label{sec:moduli spaces}
The topological invariant attached to a $\SO^*(2n,\RR)$-Higgs bundle
$(V,\beta,\gamma)$ is an element in the fundamental group of $\U(n)$
(see Section \ref{moduli-ghiggs}). Since $\pi_1(U(n))\simeq \ZZ$, this
is an integer.  This integer coincides with the degree of $V$. Under
the correspondence between Higgs bundles and surface group
representations (see Section \ref{sect: surfacegroups}), this integer
corresponds to the Toledo invariant of a representation.\footnotemark\footnotetext{It is interesting to note that this invariant has recently been interpreted in terms of fixed point data on the spectral curve associated to the Higgs bundles - see \cite{hitchin-schaposnik:2013}. This also sheds new light on the bounds described in Proposition \ref{non-emptiness-higgs} } Following
Definition \ref{Defn: M_d} we let $\cM_d(\SO^*(2n))$ denote the {\bf
  moduli space of polystable $\SO^*(2n)$-Higgs bundles
  $(V,\beta,\gamma)$ with $\deg(V) = d$}. For brevity we shall
sometimes write simply $\cM_d$ for this moduli space.  We have the
following result (cf.\ \cite[Theorem~3.4 and Proposition~3.19]{garcia-prada-gothen-mundet:2009a}).



\begin{proposition}\label{prop:modspace} Assume $n\geq 2$. 
  The moduli space $\cM_d$ of $\SO^*(2n)$-Higgs bundles over a compact Riemann 
  surface $X$ of genus $g\ge 2$ is a complex algebraic variety of
  expected $n(2n-1)(g-1)$ (where $g$ is the genus of $X$).  The
  dimension is exactly $n(2n-1)(g-1)$ if the stable locus is nonempty.
\end{proposition}

The reason for excluding $n=1$ in the preceding proposition is that
$\SO^*(1) \simeq \SO(2) $ which is not semisimple. In this case the
dimension of the moduli space is $g$ (cf.\ Section~\ref{sec:case-n=1}).

One has the following easily proven duality result.

\begin{proposition}\label{prop:duality-iso}
The map $(V,\beta,\gamma)\mapsto (V^*,\gamma,\beta)$ gives an
isomorphism $\cM_d\simeq \cM_{-d}$.
\end{proposition}

\subsection{Structure of stable $\SO^*(2n)$-Higgs bundles}\label{sec:SO*stable}
The kernel of the isotropy representation
$$\iota\colon \GL(n,\C) \to \Aut(\Lambda^2(\C^n) \oplus \Lambda^2
(\C^n)^*)$$ for $\SO^*(2n)$ is formed by the central subgroup $\{\pm I
\} \subset \GL(n,\C)$. Moreover the infinitesimal isotropy
representation has injective differential: $\ker(d\iota) =0$.  Thus
Definition~\ref{def:simple} specializes to the following.

\begin{definition}
 A $\SO^*(2n)$-Higgs bundle $(V,\beta,\gamma)$ is \textbf{simple}  if $\Aut(V,\beta,\gamma) = \{\pm I \}$ and it is 
 \textbf{infinitesimally simple} if $\aut(V,\beta,\gamma)=0$.
 \end{definition}

Contrary to the cases of vector bundles and $\U(p,q)$-Higgs bundles, stability of an
$\SO^*(2n)$-Higgs bundle does not imply that it is simple. However,
we have the following.

\begin{theorem}
\label{thm:stable-not-simple-spnr-higgs}
Let $(V,\varphi)$ be a stable $\SO^*(2n)$-Higgs bundle. If
$(V,\varphi)$ is not simple, then one of the following alternatives occurs:

\begin{enumerate}
\item The bundle $V$ is a stable vector bundle of degree zero and $\varphi=0$.  In this case $\Aut(V,\varphi)\simeq\C^*$.
\item  There is a nontrivial decomposition, unique up to reordering,
  $$(V,\varphi) = \Bigl(\bigoplus_{i=1}^{k} V_i,\sum_{i=1}^{k}
  \varphi_i\Bigr)$$
  with $\varphi_i = \beta_i + \gamma_i \in H^0(X,K \otimes (\Lambda^2V_i \oplus
  \Lambda^2V_i^*))$, such that each $(V_i,\varphi_i)$ is a stable and simple
  $\SO^*(n_i)$-Higgs bundle. Furthermore, each $\varphi_i \neq 0$
  and $(V_i,\varphi_i) \not\simeq (V_j,\varphi_j)$ for $i \neq j$.
  The automorphism group of $(V,\varphi)$ is
  \begin{displaymath}
    \Aut(V,\varphi) \simeq \Aut(V_1,\varphi_1) \times \dots \times
    \Aut(V_k,\varphi_k) \simeq (\ZZ/2)^k.
  \end{displaymath}
  \label{item:phi-neq-0}
  
  \end{enumerate}
\end{theorem}

\begin{proof} The proof is precisely the same as for the corresponding result for $\Sp(2n,\R)$-Higgs bundles (Theorem 3.17 in \cite{GGM}).  \end{proof}

In view of Theorem \ref{thm:stable-not-simple-spnr-higgs} we can shift our attention to $\SO^*(2n)$-Higgs bundles which are stable and simple.  Unlike in the case of $G$-Higgs bundles for complex reductive $G$, the combination of stability and simplicity is not necessarily sufficient to guarantee smoothness in the moduli space.  Our analysis involves the relation between $\SO^*(2n)$-Higgs bundles and $G$-Higgs bundles 
for various other\footnotemark\footnotetext{See Appendix~\ref{appendix:G-Higgs} for a summary of results for the relevant groups}   groups $G$.  We begin by noting that a $\SO^*(2n)$-Higgs bundle can be viewed as a Higgs bundle for the
larger complex groups $\SO(2n,\CC)$ and $\SL(2n,\CC)$.

\begin{theorem}\label{thm:stability-equivalence}
Let $(V,\varphi)$ be a $\SO^*(2n)$-Higgs bundle with $\varphi=(\beta,\gamma)$. Let $(E,\Phi)$ be the $\SL(2n,\C)$-Higgs bundle  given by
\begin{displaymath}
      E = V \oplus V^*,\quad
            \Phi =
    \begin{pmatrix}
      0 & \beta \\
      \gamma & 0
    \end{pmatrix}
\end{displaymath}

\noi and let $((E,Q),\Phi)$ be the $\SO(2n,\C)$-Higgs bundle  given by
$E$ and $\Phi$ as above and with $Q$ defined by
\begin{displaymath}
    Q\bigl((v,\xi),(w,\zeta)\bigr) = \xi(w) + \zeta(v),\quad
    \text{for $v,w\in V$ and $\xi,\zeta\in V^*$.}
\end{displaymath}
\noi Then
%
%
\begin{enumerate}
\item The following are equivalent:
\begin{enumerate}
\item $(E,\Phi)$ is semistable (resp.\ polystable).
\item $((E,Q),\Phi)$ is semistable (resp.\ polystable). 
\item $ (V,\varphi)$is  semistable (resp.\ polystable).
\end{enumerate}

  \item If $(E,\Phi)$ is stable then $((E,Q),\Phi)$ is stable.
  \item If
    $((E,Q),\Phi)$is stable then $(V,\varphi)$ is stable.

\item  If $(V,\varphi)$ is stable and simple then
\begin{enumerate}
\item $(E,\Phi)$ is stable unless there is an isomorphism
$f:V\xrightarrow{\simeq} V^*$ such that $\beta f=f^{-1}\gamma$;
\item $((E,Q),\Phi)$ is stable unless there is an isomorphism
$f:V\xrightarrow{\simeq} V^*$ which is skew-symmetric  and with $\beta f=f^{-1}\gamma$.
\end{enumerate}
\end{enumerate}
\end{theorem}

\begin{proof}
  The equivalences in (1) can be proved in exactly the same way as
  done for $\Sp(2n,\R)$-Higgs bundles in \cite{GGM} (see Theorems 3.26
  and 3.27). Although the equivalence analogous to the equivalence
  between (a) and (b) is not explicitly stated in \cite{GGM} in the case of
  semistability, it is implicit in the proof of the equivalence
  analogous to the equivalence between (a) and (c).  

  The implication
  in (2) follows directly from the stability conditions.  

  For the implication in (3) note that a $\varphi$-invariant
  two-step filtration $0 \subset V_1 \subset V_2 \subset V$ gives rise
  to an isotropic subbundle $V_1 \oplus V_2^{\perp}$ of $(E,Q)$ which, by
  Lemma~\ref{lemma:2-step}, is $\Phi$-invariant. These are exactly the
  subbundles which enter the stability condition for $\SO(2,\C)$-Higgs
  bundles (see Proposition~\ref{thm:sp(2n,C)-stability}). Note that
  $V_1 \oplus V_2^{\perp} \subset E$ is non-zero and proper if and
  only if the filtration $0 \subset V_1 \subset V_2 \subset V$ is
  distinct from the filtration $0 =V_1 \subset V_2 = V$. Moreover,
  \begin{displaymath}
    \deg(V_1 \oplus V_2^{\perp}) = \deg(V_1)+\deg(V_2)-\deg(V),
  \end{displaymath}
  so the stability conditions coincide.

  The statements in (4) can be proved in the same way as the analogous
  result for $\Sp(2n,\R)$-Higgs bundles (see Theorem 3.27 in
  \cite{GGM}).
\end{proof}

\begin{remark}
  If $\deg V\neq 0$, then it follows from (3) of
  Theorem~\ref{thm:stability-equivalence} that $(E,\Phi)$ (and hence
  $((E,Q),\Phi)$) is stable if
  $(V,\varphi)$ is stable and simple. Similarly, if the rank $n$ is
  odd, then $((E,Q),\Phi)$ is stable if $(V,\varphi)$ is stable and
  simple.
  On the other hand, in the situation described in Theorem~3.19(2),
  the $\SO(2n,\mathbb{C})$-Higgs bundle $((E,Q),\Phi)$ is not stable,
  because $V_i \oplus V_i^*\subset E=V \oplus V^*$ is an isotropic
  $\phi$-invariant subbundle of degree $0$).
\end{remark}

\begin{proposition}\label{cor:smooth-points}
Let $(V,\varphi)$ be a $\SO^*(2n)$-Higgs bundle which is stable
and simple and assume that there is no skewsymmetric isomorphism
$f\colon V \xrightarrow{\simeq} V^*$ intertwining $\beta$ and
$\gamma$ (i.e.\ such that $\gamma=(f\otimes 1_K)\circ \beta\circ f$).  
Then $(V,\varphi)$ represents a smooth point of the
moduli space of polystable $\SO^*(2n)$-Higgs bundles.  
In particular, if $d=\deg V$ is not zero or $n$ is odd, then all stable and simple 
   $\SO^*(2n)$-Higgs bundles represent smooth points of the
moduli space $\cM_d$.
\end{proposition}

\begin{proof}   By 
(3b) of Theorem \ref{thm:stability-equivalence} the $\SO(2n,\C)$-Higgs bundle 
corresponding to $(V,\varphi)$ is stable and hence by (4) in Proposition 
\ref{prop:smoothpoint} it represents a smooth point in $\cM_d$.
\end{proof}

It remains to analyze the case in which $(V,\varphi)$ is stable
and simple but admits a skewsymmetric isomorphism
$f\colon V \xrightarrow{\simeq} V^*$ intertwining $\beta$ and
$\gamma$.  By (3b) of Theorem \ref{thm:stability-equivalence}  this is equivalent to the associated $\SO(2n,\C)$-Higgs bundle being non-stable. Furthermore
$d=\deg V=0$ and $n$ is even.  

\begin{proposition}\label{thm:stability-equivalence-sp2nC}
Let $(V,\varphi)$ be a  $\SO^*(2n)$-Higgs
bundle with $\phi=(\beta,\gamma)$ which admits a skewsymmetric isomorphism $f\colon V
\xrightarrow{\simeq} V^*$ such that $\beta f =
f^{-1}\gamma$. Then  with  $\psi: = \beta f$, the data
$((V,f),\psi)$ defines a
$\U^*(n)$-Higgs bundle (as defined in Section \ref{sect: U*}).

Let $(V,\varphi)$ be stable. Then $((V,f),\psi)$ is stable.  Assume moreover 
that $(V,\varphi)$ is simple. Then $((V,f),\psi)$ is stable
and simple and the corresponding $\GL(n,\CC)$-Higgs bundle $(V,\psi)$
is stable. Hence $((V,f),\psi)$ represents a smooth point in the
moduli space of $\U^*(n)$-Higgs bundles.
\end{proposition}

\begin{proof}
The fact that $((V,f),\psi)$ defines a $\U^*(n)$-Higgs bundles follows
directly from the definition given in Section \ref{sect: U*}.  
The argument to prove the stability result is similar to the one 
given in the proof of Theorem 3.22 in \cite{GGM}.
The statement about simplicity follows directly from the fact that for
both $\SO^*(2n)$- and $\U^*(n)$-Higgs bundles simplicity means that
the only automorphisms are $\pm$ Identity.
\end{proof}

\begin{notation*}
  We shall, somewhat imprecisely, say that a $\SO^*(2n)$-Higgs bundle
  of the form described in
  Proposition~\ref{thm:stability-equivalence-sp2nC} is a
  \emph{$\U^*(n)$-Higgs bundle}.
\end{notation*}


\subsection{Structure of polystable $\SO^*(2n)$-Higgs bundles}\label{sec:refined}

A general structure theorem for poly\-stable $G$-Higgs bundles was
given in \cite{garcia-prada-gothen-mundet:2009a}, where it is shown
that any strictly polystable $G$-Higgs bundle admits a reduction to a
stable $G'$-Higgs bundle for a uniquely determined reductive subgroup
$G' \subset G$. Here we give an elementary argument in the case
$G=\SO^*(2n)$, identifying explicitly this reduction, without recourse
to Lie theory. Our result is the following.

\begin{proposition}
\label{prop:SO-star-poly-stability}
A $\SO^*(2n)$-Higgs bundle $(V,\varphi)$ with $\varphi=(\beta,\gamma)$
is polystable if
and only if there are decompositions
\begin{align*}
  V &= V_1 \oplus \dots \oplus V_k, \\
  \varphi &= \varphi_1 + \dots +\varphi_k,
\end{align*}
such that each $(V_i,\varphi_i)$ is a $\SO^*(2n_i)$-Higgs bundle  i.e. $\varphi_i=(\beta_i,\gamma_i)$ with $\beta_i \in
  H^0(X,\Lambda^2V_i \otimes K)$ and $\gamma_i \in H^0(X,\Lambda^2V^*_i
  \otimes K)$, and is of one of the following mutually exclusive types:
  
\begin{enumerate}
\item a stable
 $\SO^*(2n_i)$-Higgs bundle with $\varphi\ne 0$;

\item $V_i = \tilde V_i \oplus
  \tilde W^*_i$, with respect to this decomposition $\beta_i= \begin{pmatrix}
    0 & \tilde\beta_i \\
    -\tilde\beta_i^t & 0
  \end{pmatrix}$ and $\gamma_i= \begin{pmatrix}
    0 & -\tilde\gamma_i^t \\
    \tilde\gamma_i& 0
  \end{pmatrix}$where $\tilde\beta_i \in H^0(X,\Hom(\tilde W_i,\tilde
  V_i)\otimes K)$ and $\tilde\gamma_i \in H^0(X,\Hom(\tilde V_i,\tilde
  W_i)\otimes K)$, and $(\tilde V_i,\tilde
  W_i,\tilde\beta_i,\tilde\gamma_i)$ is a stable $\U(p_i,q_i)$-Higgs
  bundle in which $p_iq_i\ne
  0$, $\deg \tilde V_i +\deg \tilde W_i = 0$ and at least one of $\tilde{\beta}_i, \tilde{\gamma}_i$ is non-zero.
  \item $\varphi_i=0$ and $V_i$ is
  a degree zero stable vector bundle.
\end{enumerate}
\end{proposition}

\begin{proof}

Suppose $(V,\beta,\gamma)$ is polystable.  If it is stable then the
result is trivially true (with $k=1$).  Suppose that
$(V,\beta,\gamma)$ is not stable.  Then by Definition
\ref{prop:sp2n-alpha-stability} we can find a non trivial filtration
(i.e.\ with $l \geq 2$)
$\VVV=(0\subsetneq V'_1\subsetneq V'_2\subsetneq\dots\subsetneq
V'_l=V)$ and a sequence of weights
$\lambda=(\lambda_1<\lambda_2<\dots<\lambda_l)$ such that
\begin{itemize}
\item $\varphi\in H^0(X,N(\VVV,\lambda))$
\item  $d(\VVV,\lambda)=0$
\item  there is a
splitting of vector bundles
$$V\simeq V'_1\oplus V'_2/V'_1\oplus\dots\oplus V'_l/V'_{l-1}$$
with respect to which
$$\beta\in H^0(X,\bigoplus_{\lambda_i+\lambda_j=0}K\otimes
V'_i/V'_{i-1}\otimes_A V'_j/V'_{j-1})$$ and
$$\gamma\in H^0(X,\bigoplus_{\lambda_i+\lambda_j=0}K\otimes
(V'_i/V'_{i-1})^*\otimes_A (V'_j/V'_{j-1})^*).$$
\end{itemize}
We can write the set of weights as a disjoint union
\begin{displaymath}
  \{\lambda_1,\dots,\lambda_l\} = I_1 \cup I_2 \cup I_3, 
\end{displaymath}
where each of the sets, if non-empty, can be written as follows:
\begin{align*}
  I_1 &= \{0\},\\
  I_2 &= \{\mu_1,-\mu_1,\dots,\mu_{r},-\mu_{r}\},\\
  I_3 &= \{\eta_1,\dots,\eta_{s}\},
\end{align*}
where $\mu_i >0$ and $\eta_i \neq 0$ for all $i$, and $\abs{\eta_i} \neq
\abs{\eta_j}$ for $i \neq j$. In other words, $I_2$ contains pairs of
non-zero weights $\pm \mu_i$ and $I_3$ contains non-zero weights that
cannot be paired. Note that $I_2 \cup I_3 \neq \emptyset$ since at
least one weight is non-zero.

We can now rewrite the splitting of $V$ as
\begin{equation}
  \label{eq:V=U_i}
  V \simeq U_0 \oplus (U_{-\mu_1} \oplus U_{\mu_1}) \oplus \dots \oplus
  (U_{-\mu_r} \oplus U_{\mu_r}) \oplus U_{\eta_1}\oplus\dots\oplus U_{\eta_s}, 
\end{equation}
where $U_{\nu} = V'_i / V'_{i-1}$ if $\nu=\lambda_i$ for some
$i=1,\dots,l$ and zero otherwise.

If $I_1$ is not empty, let $\beta_{0}$ be the component of $\beta$ in
$H^0(X,K\otimes U_0\otimes_A U_0)$ and
similarly define $\gamma_0$.
If both $\beta_0=0$ and $\gamma_0=0$ then
the vector bundle $U_{0}$ is a $\U(n_0)$-Higgs bundle.  
Otherwise,
$(U_0,\beta_{0},\gamma_{0})$ defines an
$\SO^*(2n_{0})$-Higgs bundle, where $n_{0} =
\rk(U_0)$.

For each positive element $\mu_i \in I_2$, let $\tilde\beta_{i}$ be
the component of $\beta$ in $H^0(X,K\otimes U_{\mu_i}\otimes_A
U_{-\mu_i})$ and similarly define $\tilde\gamma_i$. If both
$\tilde\beta_i=0$ and $\tilde\gamma_i=0$ then the vector bundles
$U_{\mu_i}$ and $U_{-\mu_i}$ are $\U(p_i)$- and $\U(q_i)$-Higgs
bundles respectively, where $p_i=\rk(U_{\mu_i})$ and
$q_i=\rk(U_{-\mu_i})$. Otherwise,
\begin{displaymath}
  (\tilde{V}_i,\tilde{W}_i,\tilde\beta_i,\tilde\gamma_i))
  =(U_{\mu_i} \oplus U_{-\mu_i}^*,\tilde\beta_i,\tilde\gamma_i)
\end{displaymath}
defines a $\U(p_i,q_i)$-Higgs bundle, where $p_i=\rk(U_{\mu_i})$ and
$q_i=\rk(U_{-\mu_i})$. In order  to see that $\deg \tilde V_i
+\deg \tilde W_i = 0$, we note that we can write the decomposition
(\ref{eq:V=U_i}) as
\begin{displaymath}
  V = U_{-\mu_i} \oplus V' \oplus U_{\mu_i},
\end{displaymath}
where we have pulled out $U_{-\mu_i}$ and $U_{\mu_i}$
and we denote the  rest by $V'$.
Now consider the induced filtration $\VVV'$ of $V$ with the weights
$\lambda' = (-1<0<1)$. Clearly $\varphi \in H^0(X,N(\VVV',\lambda'))$. Hence
semistability implies that
\begin{displaymath}
  d(\VVV',\lambda') = \deg(U_{\mu_i}) - \deg(U_{-\mu_i}) \geq 0.
\end{displaymath}
Similarly, considering the filtration induced by 
\begin{math}
  V = U_{\mu_i} \oplus V' \oplus U_{-\mu_i}
\end{math}
with weights $(-1<0<1)$ we obtain $\deg(U_{-\mu_i}) - \deg(U_{\mu_i})
\geq 0$, and hence we conclude that
\begin{displaymath}
  \deg \tilde V_i +\deg \tilde W_i
  = \deg(U_{\mu_i}) - \deg(U_{-\mu_i}) = 0
\end{displaymath}

Finally, for each $\eta_i \in I_3$, the vector bundle $U_{\eta_i}$ is a
$\U(n_i)$-Higgs bundle and we see that $\deg(U_{\eta_i})=0$ by a
similar argument, using the decomposition $V = U_{\eta_i} \oplus V'$.

Altogether, this leads to a decomposition with summands of the type in
the statement of the Proposition. Now we show that each summand is
polystable as a $G$-Higgs bundle, where $G$ is the appropriate group,
i.e, $G=\SO^*(2n_0)$, $G=\U(p_i,q_i)$ or $G=\U(n_i)$. By
Proposition~\ref{prop:Upq-simplification}, it follows that the
$\U(p_i,q_i)$ and $\U(n_i)$ summands are direct sums of stable
ones. Suppose one of the $\SO^*(2n_0)$ summands is not
polystable. Then there is a filtration and weight system violating
polystability of this summand.  This filtration and weight system can
be extended by adding the remaining summands in $V$ to each term and
by taking the same weights. The resulting filtration and weight system
violates polystability for the original $\SO^*(2n)$-Higgs bundle
$(V,\varphi)$.  Moreover, $n_0 < n$ because $I_2 \cup I_3 \neq
\emptyset$. Hence we can iterate the procedure until all summands are
stable.

Finally, we show that the three types are mutually exclusive. The
conditions on $\varphi$ clearly make (1) and (3) mutually exclusive.
Suppose that $(V_i,\beta_i,\gamma_i)$ is of type (2).  Since it is
stable, it must have $\varphi_i\ne 0$ and hence cannot be of type (3).
Suppose that $(V_i,\beta_i,\gamma_i)$ is also stable as a
$\SO^*(2n)$-Higgs bundle.  Then it is infinitesimally simple and thus
$\aut(V_i,\beta_i,\gamma_i)=0$.  But if $(V_i,\beta_i,\gamma_i)$ is of
type (2) then $\C^*\subset \aut(V_i,\beta_i,\gamma_i)$.  Thus cases
(1) and (2) are mutually exclusive.
\end{proof}

\begin{notation*}
  We shall write $(V,\varphi) = (V,\varphi_1) \oplus \dots \oplus
  (V,\varphi_k)$ for a $\SO^*(2n)$-Higgs bundle of the kind described
  in Proposition~\ref{prop:SO-star-poly-stability}.
Moreover, somewhat imprecisely, we shall say that a
  $\SO^*(2n)$-Higgs bundle of the form described in (2) of
  Proposition~\ref{prop:SO-star-poly-stability} is a \textbf{$\U(p,q)$-Higgs
  bundle} (here $n=p+q$).
\end{notation*}

By Theorem~\ref{thm:stable-not-simple-spnr-higgs} and Propositions~\ref{cor:smooth-points} and \ref{thm:stability-equivalence-sp2nC}, case (1) in Proposition~\ref{prop:SO-star-poly-stability} divides further into two cases.  The resulting refinement, given in the next theorem, will be
essential for proving our connectedness results in Section~\ref{sect:morse}.

\begin{theorem}
\label{prop:SO-star-poly-stability-refined}
A $\SO^*(2n)$-Higgs bundles $(V,\varphi=\beta+\gamma)$ is polystable
if and only if there is a decomposition $(V,\varphi) = (V_1,\varphi_1)
\oplus \dots \oplus (V_k,\varphi_k)$ such that each $(V_i,\varphi_i)$
is a $\SO^*(2n_i)$-Higgs bundle of one of the following mutually
exclusive types:
  
\begin{enumerate}
\item $(V_i,\varphi_i)$ is a stable and simple 
 $\SO^*(2n_i)$-Higgs bundle with $\varphi_i\ne 0$ which is stable as an
 $\SO(2n_i,\C)$-Higgs bundle;
\item $(V_i,\varphi_i)$ is a stable and simple 
 $\SO^*(2n_i)$-Higgs bundle with $\varphi_i\ne 0$ which admits a skewsymmetric isomorphism as in Proposition~\ref{thm:stability-equivalence-sp2nC} and thus defines a stable $\U^*(n_i)$-Higgs bundle;
 \item $(V_i,\varphi_i)$ is as described in (2) of
  Proposition~\ref{prop:SO-star-poly-stability}) and thus defines a stable $\U(p_i,q_i)$-Higgs bundle where $p_iq_i\ne 0$, $\deg \tilde V_i +\deg  \tilde W_i = 0$ and  $\varphi_i\ne 0$;
\item $\varphi_i=0$ and $V_i$ defines a degree zero stable vector bundle.
\end{enumerate}
\end{theorem}


\subsection{Bounds on $d=\deg(V)$.}
\label{sec:spn-higgs-moduli}
In this section we give an inequality which bounds the number of
non-empty moduli spaces $\cM_d=\cM_d(\SO^*(2n))$.  The inequality corresponds to the Milnor-Wood
inequality for surface group representations into $\SO^*(2n)$ (see
Section \ref{sect: surfacegroups}).

\begin{proposition}\label{mw-higgs}
Let $(V,\beta,\gamma)$ be a semistable $\SO^*(2n)$-Higgs bundle. Then
\begin{equation}\label{eqn: range}
\rank(\beta)(1-g)\le \deg(V)\le \rank(\gamma)(g-1).
\end{equation}
In particular, 
\begin{equation}
\abs{\deg(V)} \le n(g-1) 
\end{equation}
where  $\deg(V)=n(g-1)$ if and only if $\gamma$ is an isomorphism, and
$\deg(V)=-n(g-1)$ if and only if $\beta$ is an isomorphism.
\end{proposition}

\begin{proof} This is proved by first using the equivalence between
  the semistability of $(V,\beta,\gamma)$ and the $\SL(2n,\CC)$-Higgs
  bundle $(W,\Phi)$ associated to it (see (1) in Theorem \ref
  {thm:stability-equivalence}), and then applying the semistability
  numerical criterion to special Higgs subbundles defined by the
  kernel and image of $\Phi$ (see Section 3.4 in 
  \cite{bradlow-garcia-prada-gothen:2003}, and also \cite{gothen:2001}).
\end{proof}
Notice that since $\beta$ and $\gamma$ are skew-symmetric, they
cannot be isomorphisms if $n$ is odd. If $n=2m+1$ then $2m$ is the
upper bound on $\rank(\beta)$ and $\rank(\gamma)$.  Denote by
$\left\lfloor\frac{n}{2}\right\rfloor$ the integer part of $\frac{n}{2}$. As a
corollary of Proposition \ref{mw-higgs}, we obtain the following.

\begin{proposition}\label{non-emptiness-higgs}
The moduli space $\cM_d$ is empty unless
\begin{equation}\label{eqn: mwHiggs}
\abs{d} \leq \left\lfloor\frac{n}{2}\right\rfloor(2g-2).
\end{equation}
\end{proposition}

In view of this result, we say that $d = \deg(V)$ is \textbf{maximal}
when equality holds in (\ref{eqn: mwHiggs}).

\section{The case of maximal $d$}
\label{sect: max}

\subsection{Cayley correspondence for $n=2m$}\label{subs:cayley}

In this section we will assume that $n=2m$ is even and we will describe
the $\SO^*(2n)$ moduli space for the extreme value $\abs{d} = 2m(g-1)$.
In fact, for the rest of this section we shall assume that $d =
2m(g-1)$.  This involves no loss of generality, since, by Proposition
\ref{prop:duality-iso} there is an isomorphism between the moduli
spaces for $d$ and $-d$. The main result is Theorem~\ref{thm:cayley},
which we refer to as the \emph{Cayley correspondence}.

Let  $(V,\beta,\gamma)$ 
be a $\SO^*(4m)$-Higgs bundle such that $\gamma\in H^0(X,K\otimes
\Lambda^2V^*)$ 
is an isomorphism. 
Let $L_0=K^{-1/2}$ be a fixed  square root of $K^{-1}$, and
define $ W:=V\otimes L_0$. Then $\omega:=\gamma\otimes I_{L_0}:
W \to W^*$ is a skew-symmetric isomorphism defining a non-degenerate 
symplectic  $\Omega$ on $W$,  in other words,  $(W,\Omega)$ is a
$\Sp(2m,\CC)$-holomorphic bundle. The $K^2$-twisted endomorphism
$\psi:W\to W \otimes K^2$ defined by $\psi: = 
\beta \otimes I_{L_0^{-1}}\circ  (\gamma\otimes I_{L_0})$ is $\Omega$-skewsymmetric and
hence $(W,\Omega,\psi)$ defines a $K^2$-twisted $\U^\ast(2m)$-Higgs
pair (in the sense of Section \ref{sect: U*}, suitably modified to incorporate a twisting by an arbitrary line bundle), from which we can recover the original $\SO^*(4m)$-Higgs bundle.

\begin{definition}\label{defn: cayley} With  $(V,\beta,\gamma)$ and
  $(W,\Omega,\psi)$ 
as above, we say that $(W,\Omega,\psi)$ is the \textbf{Cayley partner} to $(V,\beta,\gamma)$.
\end{definition}

\begin{theorem}\label{equivalence-stability}
Let  $(V,\beta,\gamma)$ be a   $\SO^*(4m)$-Higgs bundle with $d=2m(g-1)$
such that $\gamma$ is an isomorphism. Let $(W,\Omega,\psi)$ be the corresponding
$K^2$-twisted $\U^*(n)$-Higgs pair. Then $(V,\beta,\gamma)$ is semistable
(resp.\ stable, polystable)
if and only if $(W,\Omega,\psi)$ is semistable
(resp.\ stable, polystable).
\end{theorem}

\begin{proof}
The proof is similar to that of Theorem 4.2 in \cite{GGM}, so we will just 
sketch the main arguments. We will used the simplified 
stability notions given in Propositions 
\ref{prop:simplifiedss} and 
\ref{prop:orthogonal-stability}.
We first show that if $(V,\beta,\gamma)$ is semistable then the corresponding
$\U^*(2m)$-Higgs pair is semistable. 
Suppose otherwise, then  there exists
an isotropic $\psi$-invariant subbundle $W'\subset W$ such that $\deg W'>0$.
Let $V_1:=W'\otimes L_0^{-1}$ and let $V_2= V_1^{\perp_\gamma}$ (see Lemma
\ref{K-pairing} for the definition of $\perp_\gamma$). 
We can check that the filtration 
$0 \subset V_1 \subset V_2 \subset V$ is $\varphi$-invariant 
and  $\deg(V) - \deg(V_1) - \deg(V_2) < 0$, 
contradicting the semistability
of $(V,\beta,\gamma)$. 

To prove the converse, i.e., that  $(V,\beta,\gamma)$ is semistable 
if the corresponding $\U^*(2m)$-Higgs pair is semistable, suppose that
there is a $\varphi$-invariant filtration 
$0\subset V_1\subset V_2\subset V$ such that
$\deg(V) - \deg(V_1) - \deg(V_2) < 0$.  From this filtration we cannot
immediately obtain a destabilizing isotropic subbundle of the 
$\U^*(2m)$-Higgs pair, but we can construct an appropriate filtration 
giving the destabilizing subobject of the $\U^*(2m)$-Higgs pair.
To do this, we first  observe that the 
$\varphi$-invariance condition for $\gamma$ (second condition in
(\ref{invariance})) is equivalent, by Lemma \ref{K-pairing}, to 
$V_2\subset V_1^{\perp_{\gamma}}$.  We define  two new  filtrations 
as follows:
$$(0\subset V_1'\subset V_2'\subset V):=(0\subset V_1\subset V_1^{\perp_{\gamma}}\subset V)$$
(we indeed have $V_1\subset V_1^{\perp_{\gamma}}$ because $V_1\subset V_2$
and $V_2\subset V_1^{\perp_{\gamma}}$) and
$$(0\subset V_1''\subset V_2''\subset V):=(0\subset V_2\cap V_2^{\perp_{\gamma}}
\subset V_2+V_2^{\perp_{\gamma}}\subset V).$$
One can check (see Theorem 4.2 in \cite{GGM}) that these two filtrations
are $\varphi$-invariant and that one of  the two inequalities 
$$\deg V-\deg V_1-\deg V_1^{\perp_{\gamma}}<0,\qquad\qquad
\deg V-\deg (V_2\cap V_2^{\perp_{\gamma}})-\deg(V_2+V_2^{\perp_{\gamma}})<0$$
holds. These two filtrations give $\psi$-invariant isotropic 
subbundles $W':=V_1'\otimes L_0$ and  $W'':=V_1''\otimes L_0$  such that
either $\deg W'>0$ or  $\deg W''>0$, contradicting the
semistability of  $(W,\Omega,\psi)$. 

The proof of the statement for stability is basically the same, observing
that the trivial filtration 
$0=V_1\subset V_2=V$ corresponds to the trivial subbundle $0\subset W$.
The proof of the equivalence of the polystability
conditions follows word by word the argument for $\Sp(2n,\R)$ 
given in Theorem 4.2 in \cite{GGM}.
\end{proof}

\begin{theorem}
\label{thm:cayley}
Let $\cM_{\max}(\SO^*(4m))$ be the moduli space of polystable
$\SO^*(4m)$-Higgs bundles with $d=2m(g-1)$ and let $\cM_{K^2}(\U^*(2m))$ be
the moduli space of polystable $K^2$-twisted $\U^*(2m)$-Higgs pairs.
The map $(V,\beta,\gamma)\mapsto (W,\Omega,\psi)$ defines an
isomorphism of complex algebraic varieties
$$
\cM_{{\mathrm{max}}}(\SO^*(4m))\simeq \cM_{K^2}(\U^*(2m)).
$$
\end{theorem}

\begin{proof}
  Let $(V,\beta,\gamma)$ be a semistable $\SO^*(4m)$-Higgs bundle
  with $d=2m(g-1)$.  By Proposition \ref{mw-higgs}, $\gamma$ is an
  isomorphism and hence the map $(V,\beta,\gamma)\mapsto (W,\Omega,\psi)$
  is well defined.  The result follows now from
  Theorem~\ref{equivalence-stability} 
and the existence of local
  universal families (see \cite{schmitt:2008}).
\end{proof}

\begin{remark}
  \label{rem:cayley-zero}
  Note that a maximal  $\SO^*(2n)$-Higgs bundle $(V,\beta,\gamma)$ has
  $\beta=0$ if and only if the Cayley partner $(W,\Omega,\psi)$ has
  $\psi=0$. Thus, in particular, Theorem~\ref{equivalence-stability}
  implies that a maximal $\SO^*(2n)$-Higgs bundle of the form
  $(V,0,\gamma)$ is polystable if and only if the corresponding
  $\Sp(n,\C)$-bundle $(W,\Omega)$ is polystable. Hence, the isomorphism of
  Theorem~\ref{thm:cayley} restricts to an isomorphism between the
  subspace of $\SO^*(2n)$-Higgs bundles with $\beta=0$ in
  $\mathcal{M}_{\max}(\SO^*(2n))$ and the moduli space of polystable
  $\Sp(n,\C)$-bundles (note that there is only one topological class of
  such bundles, since $\Sp(n,\C)$ is simply connected.)  This will be important in the proof of Theorem \ref{theorem: main1}.
\end{remark}

\subsection{Rigidity for $n=2m+1$}\label{subs:rigidity}
In this section we consider the case in which $n=2m+1$ and describe
the $\SO^*(2n)$ moduli space for the extreme value $\abs{d} =
2m(g-1)$. As in Section \ref{subs:cayley}, we assume without loss of
generality that $d$ is positive. The main result is the following
Theorem\footnotemark\footnotetext{Announced without proof as
  Theorem~4.8 in \cite{bradlow-garcia-prada-gothen:2005}.}.

\begin{theorem}\label{rigidityso*2n} Let $\mathcal{M}_{\max} (\SO^*(4m+2))$ be the moduli space of polystable $\SO^*(2n)$-Higgs bundles with $n=2m+1$ and $d=2m(g-1)$. If $m>0$ and $g\ge 2$ then the stable locus of $\mathcal{M}_{\max} (\SO^*(4m+2))$ is empty and
$$
\mathcal{M}_{\max} (\SO^*(4m+2))\simeq \mathcal{M}_{\max}(\SO^*(4m))\times \Jac(X),
$$
where $\Jac(X)$ is the Jacobian of $X$.
\end{theorem}

\begin{proof}
 Let $(V,\beta,\gamma)$ be a polystable $\SO^*(2n)$-Higgs bundle with
 $n=2m+1$.  The map $\gamma:V\longrightarrow V^*\otimes K$ defines kernel and image 
sheaves:
\begin{equation}\label{ker}
0\longrightarrow \ker (\gamma)\longrightarrow V\longrightarrow \im(\gamma)\longrightarrow 0.
\end{equation}
\noi The kernel $ \ker(\gamma)$ is a subbundle of $V$, while $\im(\gamma)$ is
in general a subsheaf of $V^*\otimes  K$. Let $W_{\gamma}$ denote the
saturation of $\im(\gamma)\otimes K^{-1}\subset V^*$, so that we have
\begin{equation}\label{Im}
0\longrightarrow \im(\gamma)\otimes K^{-1}\longrightarrow W_{\gamma}\longrightarrow T\longrightarrow 0,
\end{equation}
\noindent where $T$ is a torsion sheaf.  

\noi Let $\ker(\gamma)^{\perp}$ denote the annihilator of $\ker(\gamma)$, i.e. let it be defined by
\begin{equation}\label{kerperp}
0\longrightarrow \ker(\gamma)^{\perp}\longrightarrow V^*\longrightarrow \ker(\gamma)^*\longrightarrow 0
\end{equation}

\noi The skew-symmetry of $\gamma$ implies the following:
\begin{align} 
\ker(\gamma)^{\perp}=W_{\gamma}, \label{skewconseq}\\
\rank(\gamma)\le 2m\label{skewconseq2}
\end{align}

 




\noi Combining \eqref{skewconseq} with (\ref{kerperp}), we get
\begin{equation}\label{degW}
\deg(\ker(\gamma))-\deg(W_{\gamma})= d
\end{equation} 
\noi In addition, we get linear relations from (\ref{ker}) and
(\ref{Im}), namely
\begin{equation}\label{ker-deg}
\deg(\ker(\gamma))+\deg(\im(\gamma))=d
\end{equation}
and
\begin{equation}\label{Im-deg}
\deg(\im(\gamma))-\deg(W_{\gamma})=l(2g-2)-t
\end{equation}
\noi where $t=\deg(T)$ and $l=\rank(\gamma)$.  The system (\ref{degW}),
(\ref{ker-deg}), 
(\ref{Im-deg}) can be solved, giving in particular
\begin{equation}\label{degK}
\deg(\ker(\gamma))=d+\deg(W_{\gamma})=d-l(g-1)+\frac{t}{2}.
\end{equation}

Consider now the subobject $V\oplus W_{\gamma}\subset V\oplus V^*$. This clearly satisfies

\begin{enumerate}
\item $W_{\gamma}^{\perp}\subset V$,
\item $\beta(W_{\gamma})\subset V\otimes K$,
\item $\gamma(V)\subset W_{\gamma}\otimes K$.
\end{enumerate}
\noi Thus, setting $V_1=W_{\gamma}^{\perp}$ and $V_2=V$, we get 
a filtration which  is $\varphi$-invariant, i.e.
satisfies condition (\ref{invariance}) in 
Definition \ref{prop:SO-star-semi-stability}. The semistability 
condition thus yields the inequality $\deg(W_{\gamma}^{\perp})\le 0$ or, equivalently,
\begin{equation}
d+\deg(W_{\gamma})\le 0.
\end{equation}
\noi Combined with (\ref{degK}) this gives 
 \begin{equation}\label{8}
d-l(g-1)+\frac{t}{2}\le 0.
\end{equation}
\noi It follows immediately from (\ref{8}) and (\ref{degK}) --- and the non-negativity of $t$ ---  that if $d=2m(g-1)=l(g-1)$ then $T=0$, i.e. $\im(\gamma)\otimes K^{-1}$ is a subbundle of $V^*$, and $\deg(\ker(\gamma))=0$.

By Theorem \ref{HKcorrespond} the $\SO^*(2n)$-Higgs bundle
$(V,\beta,\gamma)$ is polystable if and only if $V$ admits a Hermitian
metric $h$ satisfying the $\SO^*(2n)$-Hitchin equations. As described in
Section \ref{hitcheq}, these equations take the form
\begin{equation}\label{SO*Hitchin}
F_V+\beta\beta^*+\gamma^*\gamma=0
\end{equation}
\noi where $F_V$ is the curvature of the metric connection determined by $h$,
and the adjoints $\beta^*$ and $\gamma^*$ are with respect to $h$.  
Fix a local frame for $V$ and take the dual frame for $V^*$.  With respect 
to these frames, $\beta$ and $\gamma$ are represented by a skew-symmetric
matrices. If the frame for $V$ is compatible with the smooth decomposition 
$V=\ker(\gamma)\oplus V_{\perp}$, where $V_{\perp}$ denotes the complement to 
$\ker(\gamma)$, then the matrices have the form
\begin{equation}\label{GBmatrices}
\gamma=\begin{pmatrix}0&0\\0&\gamma\end{pmatrix}\ ,\  \beta=\begin{pmatrix}\beta1&\beta_2\\-\beta_2&\beta_3\end{pmatrix}
\end{equation}
with respect to the decompositions $V=\ker(\gamma)\oplus V_{\perp}$ and
$V^*=(\ker(\gamma))^*\oplus (V_{\perp})^*$. 

The metric connection  decomposes as
\begin{equation}
D_{V}=\begin{pmatrix}D_{\ker}&A\\-\bar{A}^T&D_{\perp}\end{pmatrix}\end{equation}
\noi where $A\in\Omega^{0,1}(\Hom(V_{\perp},\ker(\gamma)))$ is the second
fundamental form for the embedding of the subbundle $\ker(\gamma)\subset
V$.  The corresponding decomposition of the curvature is
\begin{equation}
F_{V}=\begin{pmatrix}F_{\ker}-A\wedge\bar{A}^T&*\\ *& F_{V_{\perp}}-\bar{A}^T\wedge A\end{pmatrix} \ .
\end{equation}

Applying  $i\Lambda \Tr$ to equation (\ref{SO*Hitchin}), and using  (\ref{GBmatrices}) thus yields
\begin{align}
& \deg(\ker(\gamma))+\Pi+||\beta_1||^2+||\beta_2||^2=0\label{1}\\
& \deg(V_{\perp})-\Pi+||\beta_{2}||^2+||\beta_3||^2-||\gamma||^2=0\label{2}
\end{align}
\noi where $\Pi=-i\Lambda \Tr(A\wedge\bar{A}^T)$.  Notice that, since the second fundamental form is of type $(0,1)$, we get that 
\begin{equation}\Pi\ge 0\ .
\end{equation}
\noi But if $d=2m(g-1)$ and $\rank(\gamma)=2m$ then $\deg(\ker(\gamma))=0$. It thus follows from (\ref{1}) that $\Pi=0$ and also that $\beta_1=\beta_2=0$.  This immediately implies that the $\SO^*(2n)$-Higgs bundle $(V,\beta,\gamma)$ decomposes as a sum
\begin{equation}
(V,\beta,\gamma)=(\ker(\gamma),0,0)\oplus (V_{\perp},\beta_3,\gamma).
\end{equation}
\noi Notice that with $V_1=0$ and $V_2= V_{\perp}$ we get a $\varphi$-invariant two-step filtration (see definition \ref{prop:SO-star-semi-stability}) with
\begin{equation}
\deg(V)-\deg(V_1)-\deg(V_2) = 0
\end{equation}
\noi By Proposition \ref{prop:simplifiedss} $(V,\beta,\gamma)$ is thus not stable. Moreover, $\ker(\gamma)$ is a holomorphic line bundle, while $(V_{\perp},\beta_3,\gamma)$ is a
$\SO^*(4m)$-Higgs bundle.  The data thus define a Higgs bundle with structure group
\begin{displaymath}
\SO^*(4m) \times \SO(2) = \SO^*(4m) \times \U(1)\ .
\end{displaymath}
This completes the proof of Theorem \ref{rigidityso*2n}.
\end{proof}

\begin{remark} It follows from Theorem \ref{rigidityso*2n} that $\mathcal{M}_{\max} (\SO^*(4m+2))$ has dimension $2m(2m-1)(g-1)+g$. Comparing with the expected dimension given in Proposition\ref{prop:modspace} we see that $\dim(\mathcal{M}_{\max} (\SO^*(4m+2)))$ is smaller than expected if $g\ge 2$ and $m>0$.  This explains why we refer to Theorem \ref{rigidityso*2n} as a rigidity result.
\end{remark}

\section{Connected components of the moduli space}\label{sect:morse}

\subsection{The Hitchin functional and connected components of the
  moduli space}

The method we shall use for studying the topology of the moduli space
goes back to Hitchin \cite{hitchin:1987}. In the following, we very
briefly outline the general aspects of this approach, applied to the
count of connected components (more details can be found in, for
instance,
\cite{hitchin:1992,bradlow-garcia-prada-gothen:2003,bradlow-garcia-prada-gothen:2005,GGM}). We
then apply this programme (in Theorem~\ref{theorem: main1} below) to show that $\cM_d$ is connected for $d=0$ and the maximal value of $|d|$ (where $\cM_d=\cM_d(\SO^*(2n)$, as in Section \ref{sec:moduli spaces}).

The method rests on the gauge theoretic interpretation of the moduli
space (provided by Theorem~\ref{HKcorrespond}) as the moduli space of
solutions to the Hitchin equations (\ref{eqn:GHitch}).  Given defining
data for a $\SO^*(2n)$-Higgs bundle, namely $(V,\beta,\gamma)$, the
solution to the equations is a Hermitian metric on the vector bundle $V$.
Thus it makes sense to define the \textbf{Hitchin function}
\begin{equation}
\label{eq:def-hitchin-functional}
\begin{aligned}
 f\colon \mathcal{M}_d&\to \R \\
 (V,\beta, \gamma) &\mapsto \norm{\beta}^2 + \norm{\gamma}^2
\end{aligned}
\end{equation}
\noi where the $L^2$-norms of $\beta$ and $\gamma$ are computed using
the metric which satisfies the Hitchin equation. The function $f$ is
proper and therefore attains a minimum on each connected component of
$\mathcal{M}_d$. Hence, if the subspace of local minima of
$f$ restricted to $\mathcal{M}_{d}$ can be shown to be
connected, then it will follow that $\mathcal{M}_{d}$
itself is connected.

\begin{theorem}\label{thm:SO*-minima}Let $(V,\beta,\gamma)$ be a poly-stable $\SO^*(2n)$-Higgs bundle.
\begin{itemize}
\item[(1)] If $d > 0$, then $(V,\beta,\gamma)$ represents a local
minimum on $\mathcal{M}_d$ if and only if $\beta=0$.
\item[(2)] If $d < 0$, then $(V,\beta,\gamma)$ represents a local
minimum on $\mathcal{M}_d$ if and only if $\gamma=0$.
\item[(3)] If $d = 0$, then $(V,\beta,\gamma)$ represents a local minimum
 on $\mathcal{M}_d$ if and only if $\beta=0$ and $\gamma=0$.
\end{itemize}
\end{theorem}

Before giving the proof of this result (at the end of
Section~\ref{sec:minima-hitchin-functional} below), we apply it to
prove our main theorem on the connectedness of
$\mathcal{M}_{0}$ and $\mathcal{M}_{\max}$.

\begin{theorem}\label{theorem: main1}
  The moduli space $\cM_d$ is non-empty\footnote{Non-emptiness, also
    for non-maximal components, follows from the results of
    \cite{hitchin-schaposnik:2013} which appeared after the present paper.} and connected if $d=0$ or
  $|d|=\left\lfloor\frac{n}{2}\right\rfloor(2g-2)$.
\end{theorem}

\begin{proof}
  Consider first the case $d=0$. From (3) of
  Theorem~\ref{thm:SO*-minima} it is immediate that the subspace of
  local minima of the Hitchin function on $\cM_0$ consists of
  polystable $\SO^*(n)$-Higgs bundles $(V,\beta,\gamma)$ with
  $\beta=\gamma=0$. Furthermore, we conclude from
  Theorem~\ref{prop:SO-star-poly-stability-refined} that such an
  $\SO^*(2n)$-Higgs bundle is polystable if and only if $V$ is
  a polystable vector bundle. Therefore, the subspace of local minima of
  the Hitchin function on $\cM_0$ can be identified with the moduli space
  of polystable vector bundles of degree zero, which is known to be
  connected. This completes the proof of the case $d=0$.

  Next we turn to the case $\abs{d}=\lfloor\frac{n}{2}\rfloor(2g-2)$, i.e., the
  proof of connectedness of $\mathcal{M}_{\max}$. By
  Proposition~\ref{prop:duality-iso} we may assume, without loss of
  generality, that $d$ is positive. From (1) of Theorem~\ref{thm:SO*-minima}, we
  have that the subspace of local minima of the Hitchin function on
  $\mathcal{M}_{\max}$ can be identified with the subspace of
  $(V,\beta,\gamma)$ with $\beta=0$. Suppose now that $n$ is
  even. Then, using Remark~\ref{rem:cayley-zero}, we have that this
  subspace is isomorphic to the moduli space of polystable
  $\Sp(n,\C)$-bundles. This space is connected by Ramanathan
  \cite[Proposition~4.2]{ramanathan:1975} and hence
  $\mathcal{M}_{\max}$ is connected when $n$ is even. The
  connectedness of $\mathcal{M}_{\max}$ for odd $n$ now follows from
  the rigidity result of Theorem~\ref{rigidityso*2n} and the
  connectedness of $\mathcal{M}_{\max}$ for even $n$.

Finally, non-emptiness of the moduli spaces follows from the
non-emptiness of the subspaces of local minima of the Hitchin
functional, which in turn follows from the identifications given in
the course of the present proof.
\end{proof}

\subsection{Minima of the Hitchin functional}
\label{sec:minima-hitchin-functional}

The purpose of this section is to prove
Theorem~\ref{thm:SO*-minima}. For this we need to show various
preliminary results and, using these, we give the proof of the Theorem
at the end of the section.

The following result is completely analogous to
\cite[Proposition~4.5]{bradlow-garcia-prada-gothen:2003}.

\begin{proposition}
  \label{prop:absolute-minimum}
  The absolute minimum of the Hitchin functional restricted to
  $\mathcal{M}_d$ is $\abs{d}$. This minimal value is
  attained at a point represented by $(V,\beta,\gamma)$ (with
  $\deg(V)=d$) if and only
  if $\beta=0$ (if $d\ge 0$) or $\gamma=0$ (if $d\le 0$).
\end{proposition}


\begin{proof} Using the Hitchin equation and Chern--Weil theory we get
that
\begin{equation}
d+\norm{\beta}^2 - \norm{\gamma}^2=0
\end{equation}
and hence the Hitchin function can be expressed as
\begin{equation}\label{eqn:hitchinmin}
f(V,\beta,\gamma)=\begin{cases}d+2\norm{\beta}^2\\
-d+ 2\norm{\gamma}^2
\end{cases}
\end{equation}
The result follows immediately from \eqref{eqn:hitchinmin}.
\end{proof}

Of course not all local minima are necessarily absolute minima.  We thus need to examine more closely the structure of the local minima.

On the smooth locus of $\mathcal{M}_d$, the Hitchin
functional $f$ arises as the moment map of the $S^1$-action given by
multiplication of the Higgs field $\varphi$ by complex numbers of modulus
one. Considering the moduli space from the algebraic or
holomorphic point of view, this action extends to the $\C^*$-action
given by $(V,\phi) \mapsto (V,w \phi)$ for $w \in \C^*$. The moment
map interpretation shows that, on the smooth locus of
$\mathcal{M}_d(\SO^*(2n))$, the critical points of $f$ are exactly the
fixed points of the $\C^*$-action. On the full moduli space, the fixed
point locus of the $\C^*$-action coincides with the locus of
\emph{Hodge bundles} (this can be easily seen by arguments like the
ones used in \cite{hitchin:1987,hitchin:1992, simpson:1992}), which are defined as follows.

\begin{definition}
 A $\SO^*(2n)$-Higgs bundle $(V,\beta,\gamma)$ is called a
 \textbf{Hodge bundle} if 
\begin{itemize}
\item there is a decomposition of $V$ into holomorphic subbundles 
\begin{equation}\label{eqn:Vdecomp}
V= \bigoplus_{i}F_{i}
\end{equation}
and, with respect to this decomposition,
\item $\beta: F^*_{-i}\longrightarrow
F_{i+1}\otimes K$, and $\gamma: F_{i}\longrightarrow
F^*_{-i+1}\otimes K$
\end{itemize}
Here $F^*_i\subset V^*$ is the dual of $F_i$.  

The \textbf{weight of $F_{i}$} is $i$ and the
\textbf{weight of $F^*_{i}$} is $-i$.
\end{definition}

Thus, in view of (4) of Proposition~\ref{prop:smoothpoint}, we have
the following characterization of the critical points of $f$.

\begin{proposition}\label{prop:critical_point}
 A simple $\SO^*(2n)$-Higgs bundle, which is stable as an
 $\SO(2n,\C)$-Higgs bundle, represents a critical
 point of $f$ if and only if it is a Hodge bundle. 
\end{proposition}

The deformation complex (\ref{eq:def-complex}) for a
$\SO^*(2n)$-Higgs bundle $(E,\varphi)$ is
\begin{equation}
\label{eq:spnr-def-complex}
\begin{aligned}
  C^\bullet(V,\varphi)\colon \End(V) &\xrightarrow{\ad(\varphi)} \Lambda^2 V
  \otimes K \oplus \Lambda^2V^* \otimes K\ . \\
  \psi & \mapsto (-\beta\psi^t-\psi\beta,\gamma\psi+\psi^t\gamma).
\end{aligned}
\end{equation}
If $(V,\phi)$ is a Hodge bundle, then the decomposition
(\ref{eqn:Vdecomp}) of $V$ induces corresponding weight decompositions
\begin{displaymath}
 \End(V) = \bigoplus U_k^+\qquad\text{and}\qquad
 \Lambda^2 V\oplus \Lambda^2 V^*
 = \bigoplus U_k^-
\end{displaymath}
where
\begin{equation}
\label{eqn:U+-k}
\begin{aligned}
U^+_k &=\bigoplus_{j-i=k}F^*_i\otimes
F_j,\qquad\text{and}\qquad
U^-_k &=\bigoplus_{i+j=k}F_i\otimes_A F_j \oplus
\bigoplus_{i+j=-k}F^*_i\otimes_A F^*_j.
\end{aligned}
\end{equation}
Moreover, since the Higgs field $\phi$ has weight one, the deformation
complex (\ref{eq:spnr-def-complex}) decomposes accordingly as
\begin{displaymath}
 C^{\bullet}(V,\phi) = \bigoplus_k C^{\bullet}_k(V,\phi),
\end{displaymath}
where we let
\begin{math}
 C^{\bullet}_k(V,\phi)\colon U^+_k \xrightarrow{\ad(\phi)} U^-_{k+1}
 \otimes K.
\end{math}
If we write $C^{\bullet}_{-}(V,\phi) =
\bigoplus_{k>0}C^{\bullet}_k(V,\phi)$ we then have the corresponding
\textbf{positive weight subspace}
\begin{displaymath}
 \HH^1(C^{\bullet}_{-}(V,\phi)) \subset \HH^1(C^{\bullet}(V,\phi))
\end{displaymath}
of the infinitesimal deformation space. When $(V,\phi)$ represents a
smooth point of the moduli space, the hypercohomology
$\HH^1(C^{\bullet}_{-}(V,\phi))$ is the negative eigenvalue subspace
of the Hessian of $f$ and so $(V,\phi)$ is a local minimum of $f$ if
and only if $\HH^1(C^{\bullet}_{-}(V,\phi)) = 0$.

The key result we need for identifying the minima of $f$ on the smooth
locus of the moduli space is the following
(\cite[Corollary~5.8]{bradlow-garcia-prada-gothen:2005}).

\begin{proposition}
\label{cor:adjoint-minima}
Assume that $(V,\phi)$ is a $\SO^*(2n)$-Higgs bundle which is
  stable as a $\SO(2n,\C)$-Higgs bundle. Then $(V,\phi)$ represents
  a local minimum of $f$ in $\mathcal{M}_d$ if and only if
  it is a Hodge  bundle and 
$$
\ad(\phi)\colon U^+_k \lra U^-_{k+1}\otimes K
$$
is an isomorphism for all
$k>0$.
\end{proposition}

Using this result, we can prove the following lemma.

\begin{lemma}\label{lem:smooth-SO*-minima}
  Let $(V,\beta,\gamma)$ be a simple $\SO^*(2n)$-Higgs bundle which is
  stable as a $\SO(2n,\C)$-Higgs bundle and assume that
  $(V,\beta,\gamma)$ represents a local minimum of $f$ on
  $\mathcal{M}_d$. Then, if $d = \deg(V) \geq 0$ the vanishing
  $\beta=0$ holds and, if $d = \deg(V) \leq 0$ the vanishing $\gamma=0$
  holds.
\end{lemma}

\begin{proof} 
  Let $(V,\beta,\gamma)=(V,\phi)$ be a minimum. Then
  Proposition~\ref{prop:critical_point} implies that
  $(V,\beta,\gamma)$ is a Hodge bundle. Moreover, arguing as in
  \cite[Section~6]{GGM}, we see that $(V,\beta,\gamma)$ being simple
  implies the following: there is a decomposition of $V$ into $2p+1$
  non-zero holomorphic subbundles (for some $p \in \frac{1}{2}\ZZ$),
  which is either of the form:
 \begin{equation}
   \label{eq:type1}
   \begin{aligned}
     &V=F_{-p+\frac{1}{2}}\oplus F_{-p+2+\frac{1}{2}}\oplus\dots\oplus
     F_{p-2+\frac{1}{2}}\oplus F_{p+\frac{1}{2}},\\
    &\beta\colon F^*_{p-2j+\frac{1}{2}}\longrightarrow
     F_{-p+2j+\frac{1}{2}}\otimes K,\quad \text{for $0\le j\le p$, and}\\
    &\gamma\colon F_{-p+2j+\frac{1}{2}}\longrightarrow
     F^*_{p-2(j+1)+\frac{1}{2}}\otimes K,\quad \text{for $0\le j\le p$}.
   \end{aligned}
 \end{equation}
or of the form
 \begin{equation}
   \label{eq:type2}
   \begin{aligned}
     &V=F_{-p-\frac{1}{2}}\oplus F_{-p+2-\frac{1}{2}}\oplus\dots\oplus
     F_{p-2-\frac{1}{2}}\oplus F_{p-\frac{1}{2}},\\
    &\beta\colon F^*_{p-2j-\frac{1}{2}}\longrightarrow
     F_{-p+2j-\frac{1}{2}}\otimes K,\quad \text{for $0\le j\le p$, and}\\
    &\gamma\colon F_{-p+2j-\frac{1}{2}}\longrightarrow
     F^*_{p-2(j+1)-\frac{1}{2}}\otimes K,\quad \text{for $0\le j\le p$}.
   \end{aligned}
 \end{equation}

 Let $k_0$ be the largest index such that $U^+_{k_0}\ne 0$. Since
 otherwise there is nothing to prove, we may assume that $k_0 >0$. For
 definiteness, assume that the decomposition of $V$ is of the form
 \eqref{eq:type1} --- a similar argument applies when $V$ is of the form
 \eqref{eq:type2}. Using (\ref{eqn:U+-k}), we see that $k_0=2p$ and
 thus (by Proposition \ref{cor:adjoint-minima}) we have an isomorphism
\begin{equation}\label{eqn:k=2p}
\ad(\varphi)\colon F^*_{-p+\frac{1}{2}}\otimes F_{p+\frac{1}{2}}\lra
\Lambda^2 F_{p+\frac{1}{2}}\otimes K.
\end{equation}
In this case, since $\gamma=0$ on $F_{p+\frac{1}{2}}$, the map
$\ad(\phi)$ is given explicitly by
\begin{displaymath}
x\mapsto \phi\circ x-x\circ\phi = -x\circ\beta,
\end{displaymath}
where
\begin{equation}
 \label{eq:beta-p+half}
 \beta\colon F_{p+\frac{1}{2}}^* \to F_{-p+\frac{1}{2}}
\end{equation}
for a local section $x:F^*_{-p+\frac{1}{2}}\lra F_{p+\frac{1}{2}}$.
Denote the ranks of $F_{p+\frac{1}{2}}$ and $F_{-p+\frac{1}{2}}$ by
$a$ and $b$ respectively. Then (\ref{eqn:k=2p}) implies that
$ab=\frac{a(a-1)}{2}$ and hence that
\begin{equation}
 \label{eq:a>b}
 a=2b+1>b.
\end{equation}
But then the map $\beta$ in (\ref{eq:beta-p+half})
must have a non-trivial kernel and, therefore, the map
$$
-x\circ\beta: F^*_{p+\frac{1}{2}}\lra F_{-p+\frac{1}{2}}\lra
F_{p+\frac{1}{2}}
$$
vanishes on $\ker(\beta)$ for any local section $x$.  Now,
(\ref{eq:a>b}) implies that
\begin{displaymath}
 a=\rk(F_{p+\frac{1}{2}}) \geq 2.
\end{displaymath}
Hence there are non-zero antisymmetric local sections $y$ of
$\Lambda^2 F_{p+\frac{1}{2}}\otimes K$ which do not vanish on the
kernel of $\beta$. This is in contradiction with the existence of the
isomorphism (\ref{eqn:k=2p}).
\end{proof}

In order to show that certain singular points  of the
moduli space are not minima, we need the following lemma (cf.\ Hitchin
\cite[\S 8]{hitchin:1992}).

\begin{lemma}
 \label{prop:non-smooth-non-minima}
 Let $(V,\varphi)$ be a polystable $\SO^*(2n)$-Higgs bundle which is a
 Hodge bundle. Suppose
 there is a family $(V_t,\varphi_t)$ of polystable $\SO^*(2n)$-Higgs
 bundles, pa\-ra\-me\-trized by $t$ in the open unit disk $D \subset \C$, such
 that $(V_0,\varphi_0) = (V,\varphi)$ and the corresponding infinitesimal
 deformation is a non-zero element of $\HH^1(C^\bullet_-(V,\varphi))$.
 Then $(V,\varphi)$ is not a local minimum of $f$ on
 $\mathcal{M}_d$.
\end{lemma}

Using this criterion and
Theorem~\ref{prop:SO-star-poly-stability-refined}, we can now
extend the result of Lemma~\ref{lem:smooth-SO*-minima} to cover all
polystable $\SO^*(2n)$-Higgs bundles.

\begin{lemma}
  \label{lem:non-smooth-minima}
  Let $(V,\beta,\gamma)$ be a polystable $\SO^*(2n)$-Higgs bundle and
  assume that $(V,\beta,\gamma)$ represents a local minimum of $f$ on
  $\mathcal{M}_d$. Then, if $d = \deg(V) \geq 0$ the vanishing
  $\beta=0$ holds and, if $d = \deg(V) \leq 0$ the vanishing $\gamma=0$
  holds.
\end{lemma}

\begin{proof}
  Let $(V,\varphi) = (V_1,\varphi_1) \oplus \dots \oplus
  (V_k,\varphi_k)$ be the decomposition given in
  Theorem~\ref{prop:SO-star-poly-stability-refined}. 
  As observed by Hitchin \cite{hitchin:1992}, the Hitchin function
  \eqref{eq:def-hitchin-functional} is additive in the sense that
  \begin{displaymath}
    f(V,\varphi) =\sum_{i=1}^k f (V_i,\varphi_i).
  \end{displaymath}
  It follows that each summand $(V_i,\varphi_i)$ represents a local
  minimum for the Hitchin functional on its own moduli space.

  If a summand $(V_i,\varphi_i)$ is of type (1) in
  Theorem~\ref{prop:SO-star-poly-stability-refined}, then
  Lemma~\ref{lem:smooth-SO*-minima} shows that $\beta_i=0$ or
  $\gamma_i=0$. Similarly, if a summand $(V_i,\varphi_i)$ is of type
  (3), then it is shown in
  \cite[Theorem~4.6]{bradlow-garcia-prada-gothen:2003} that
  $\beta_i=0$ or $\gamma_i=0$. With regard to summands of type (2), it
  is shown in \cite[Proposition~4.6]{garcia-prada-oliveira:2010} that
  a stable $\U^*(n_i)$-Higgs bundle $(V_i,\varphi_i)$ representing a
  local minimum on the corresponding moduli space has $\varphi_i=0$.
  Finally we note that the summands $(V_i,\varphi_i)$ of type (4) have
  $\varphi_i=0$.

  Thus each of the summands $(V_i,\varphi_i)$ of type (1) or (3) has
  either $\beta_i=0$ or $\gamma_i=0$ and each of the summands of type
  (2) or (4) has $\varphi_i=0$.
  
  To complete the proof, assume that there are summands
  $(V',\beta',\gamma')$ and $(V'',\beta'',\gamma'')$ with $\beta'=0$,
  $\gamma'\neq 0$, $\beta''\neq 0$ and $\gamma''=0$, and that each of
  these summands is either of type (1) or of type (3). If we can
  construct a family $(V_t,\varphi_t)$ of polystable $\SO^*(2n)$-Higgs
  bundles such that
  \begin{displaymath}
    (V_0,\varphi_0) = (V',\beta'+\gamma') \oplus (V'',\beta''+\gamma'')
  \end{displaymath}
  and satisfying the hypothesis of
  Lemma~\ref{prop:non-smooth-non-minima}, this proposition
  guarantees that $(V',\beta'+\gamma') \oplus (V'',\beta''+\gamma'')$
  is not a minimum (on its own moduli space) and hence $(V,\varphi)$ cannot
  be a minimum. In the analogous case of $\Sp(2n,\R)$-Higgs bundles,
  such a family is constructed in Lemmas~7.2 and 7.3 of
  \cite{GGM}. Inspection of the proofs of these two lemmas shows that
  they are not sensitive to the symmetry properties of $\beta$ and
  $\gamma$ and so go through unchanged in the present case of
  $\SO^*(2n)$-Higgs bundles. This completes the proof.
\end{proof}

\noi Finally we are in a position to prove Theorem~\ref{thm:SO*-minima}.

\begin{proof}[Proof of Theorem~\ref{thm:SO*-minima}]
  The ``if'' part is immediate
  from Proposition~\ref{prop:absolute-minimum}.  In the case $|d|= \lfloor \frac{n}{2}\rfloor(2g-2)$, the ``only
  if'' part follows from Lemma~\ref{lem:non-smooth-minima}. In the
  case $d=0$ the result follows from the observation that if one of the Higgs fields $\beta$ and $\gamma$ vanishes, then polystability of $(V,\beta,\gamma)$ forces
  the other Higgs field to vanish.  
\end{proof}

\section{Representations of $\pi_1(X)$ in $\SO^*(2n)$}
\label{sect: surfacegroups}

Let $X$ be a compact Riemann surface of genus $g$ and let
\begin{displaymath}
  \pi_{1}(X) = \langle a_{1},b_{1}, \dotsc, a_{g},b_{g} \suchthat
  \prod_{i=1}^{g}[a_{i},b_{i}] = 1 \rangle
\end{displaymath}
be its fundamental group.  
By a representation of $\pi_1(X)$ in
$\SO^*(2n)$ we mean  a homomorphism $\rho\colon \pi_1(X) \to \SO^*(2n)$.
The set of all such homomorphisms,
$$\Hom(\pi_1(X),\SO^*(2n)),$$ can be naturally identified with the subset
of $\SO^*(2n)^{2g}$ consisting of $2g$-tuples
$$(A_{1},B_{1}\dotsc,A_{g},B_{g})$$ 
satisfying the algebraic equation
$\prod_{i=1}^{g}[A_{i},B_{i}] = 1$.  This shows that
$\Hom(\pi_1(X),\SO^*(2n))$ is a  real algebraic   variety.

The group $\SO^*(2n)$ acts on $\Hom(\pi_1(X),\SO^*(2n))$ by conjugation:
\[
(g \cdot \rho)(\gamma) = g \rho(\gamma) g^{-1}
\]
for $g \in \SO^*(2n)$, $\rho \in \Hom(\pi_1(X),\SO^*(2n))$ and $\gamma\in
\pi_1(X)$. Recall that a representation is reductive if its composition 
with the adjoint representation is semisimple. 
If we restrict the action to the subspace
$\Hom^{\mathrm{red}}(\pi_1(X),\SO^*(2n))$ consisting of reductive
representations, the orbit space is Hausdorff.  
By a reductive representation we mean one
for which the Zariski closure of the
image of $\pi_1(X)$ in $\SO^*(2n)$ is  a reductive group. 
Define the
{\bf moduli space of representations} of $\pi_1(X)$ in $\SO^*(2n)$ to be
the orbit space
\[
\mathcal{R} = \Hom^{\mathrm{red}}(\pi_1(X),\SO^*(2n)) / \SO^*(2n). \]





Since
$\U(n)\subset \SO^*(2n)$ is a maximal compact subgroup, we have
$$
\pi_1(\SO^*(2n))\simeq \pi_1(\U(n))\simeq \ZZ,
$$
and there is a topological invariant attached to a representation
$\rho\in\cR$ given by  an element $d=d(\rho)\in\ZZ$. This integer is
called the \textbf{Toledo invariant} and coincides with the first
Chern class of a reduction to a $\U(n)$-bundle of the flat
$\SO^*(2n)$-bundle associated to $\rho$.

Fixing the invariant $d\in \ZZ$ we consider, 
$$
\mathcal{R}_d:=\{\rho \in \mathcal{R}\;\;\;\mbox{such that}
\;\;\; d(\rho)=d\}.
$$

\begin{proposition}\label{rep-duality}
The transformation  $\rho\mapsto {(\rho^t)}^{-1}$ in $\cR$ induces an
isomorphism of the moduli  spaces $\cR_d$ and $\cR_{-d}$.
\end{proposition}

As shown by  Domic--Toledo
\cite{domic-toledo:1987}, the Toledo invariant $d$ of a representation
satisfies the Milnor--Wood type inequality:

\begin{proposition}
The moduli space $\cR_d$ is empty unless
$$
\abs{d} \leq \left\lfloor\frac{n}{2}\right\rfloor(2g-2).
$$
\end{proposition}


As a special case of of the non-abelian Hodge theory correspondence
(see \cite[Theorem~3.32]{garcia-prada-gothen-mundet:2009a}) we have
the following.

\begin{proposition}\label{Md-Rd}
The moduli spaces $\cR_d$ and $\cM_d$ are homeomorphic.
\end{proposition}

{}From Proposition \ref{Md-Rd} and Theorem
\ref{theorem: main1} we have the main
result of this paper regarding the connectedness properties of $\cR$
given by the following.

\begin{theorem}\label{theorem:main-rep}
  The moduli space $\cR_d$ is non-empty and connected if $d=0$ or
  $|d|=\lfloor \frac{n}{2}\rfloor(2g-2)$.
\end{theorem}

{}From Proposition \ref{Md-Rd} and Theorem \ref{theorem:main-rep} we also have
the following rigidity result for maximal representations.

\begin{theorem}\label{rigidityso*2n-rep} Let $\mathcal{R}_{\max}(\SO^*(4m+2))$
  be the moduli space of maximal representations in $\SO^*(2n)$ with
  $n=2m+1$ and $d=2m(g-1)$. If $m>0$ and $g\ge 2$ then the locus of
  irreducible representations of $\mathcal{R}_{\max} (\SO^*(4m+2))$ is
  empty and
$$
\mathcal{R}_{\max} (\SO^*(4m+2))\simeq
\mathcal{R}_{\max}(\SO^*(4m))\times \Hom(\pi_1(X),\U(1)).
$$
\end{theorem}

\section{Low rank cases}\label{sect:lowrank}
In this section we exploit well known Lie-theoretic isomorphisms to examine $\SO^*(2n)$-Higgs bundles for low values of $n$. 

\subsection{The case $n=1$}
\label{sec:case-n=1}

The group $\SO^*(2)$ is isomorphic to $\SO(2)$ and hence, in
particular, it is compact. A $\SO^*(2)$-Higgs bundle is thus simply a
bundle (with zero Higgs field).  Identifying the maximal compact
subgroup (in this case the group itself) with $\U(1)$, we see that a
$\SO^*(2)$-Higgs bundle consists of a $\GL(1,\C)$-bundle, or
equivalently, a holomorphic line bundle.  Using the usual
identification $\GL(1,\C)\simeq\SO(2,\C)$, we see that the associated
$\SO(2,\C)$-Higgs bundle is equivalent to the vector bundle $L\oplus
L^{-1}$ with the standard off-diagonal quadratic form.

\begin{proposition} \label{prop: 2ss} As a $\SO^*(2)$-Higgs bundle, a line bundle $L$ is semistable if and only if $\deg(L)=0$. Moreover, semistability implies stability for $\SO^*(2)$-Higgs bundles.
\end{proposition}

\begin{proof}We apply Proposition \ref{prop:simplifiedss}. The only two-step filtrations are:
\begin{align*}
0 \subset 0 &\subset 0 \subset L\\
0 \subset 0 &\subset L \subset L\\
0 \subset L &\subset L \subset L
\end{align*}
\noi All are $\varphi$-invariant since the Higgs field is zero. Applying \eqref{eq:spn-simplified-sstab} to these filtrations in turn yields $\deg(L)\le 0$, $0\le 0$, and $\deg(L)\ge 0$.  The first result follows from this.  The second result is a consequence of the fact that there are no $\varphi$-invariant two-step filtrations in which at least one of the subbundles is proper.
\end{proof}

\begin{remark}  Since  $L$ and $L^{-1}$ are isotropic subbundles of $L\oplus L^{-1}$, it follows that $L\oplus L^{-1}$ is semistable as a $\SO(2,\C)$-bundle if and only if $\deg(L)=0$. This gives an alternative proof for Proposition \ref{prop: 2ss}.
\end{remark}
It follows that the moduli space of $\mathcal{M}_d(\SO^*(2))$ is non-empty only for $d=0$, in which case we can identify
\begin{equation*}
\mathcal{M}_0(\SO^*(2))\simeq \Jac^0(X)
\end{equation*} 
\noi where  $\Jac^0(X)$ denotes the Jacobian of degree zero line bundles over $X$.

\begin{remark}
It may look paradoxical that we do not obtain the whole moduli space of line
bundles of arbitrary degree over $X$. This is because, as
indicated in Remark \ref{remark-parameter}, we are fixing the parameter of
stability to be zero. In order to obtain the other components of the moduli space we
have to consider stability for other integral values of the parameter.
\end{remark}

\subsection{The case $n=2$}  In this section we examine the $\SO^*(2n)$-Higgs bundles $(V,\beta,\gamma)$ in which $\rank(V)=2$.  The low rank and the isomorphism 
\begin{equation}\label{so4isom}
\mathfrak{so}^*(4)\simeq \mathfrak{su}(2)\oplus\mathfrak{sl}(2,\R)
\end{equation}
\noi lead us to descriptions that are more explicit than in the general case.

\subsubsection{Stability conditions} If $\rank(V)=2$ there are no
two-step filtrations $0\subset V_1\subset V_2\subset V$ in which all
the inclusions are strict.  The two-step filtrations with at
least one non-zero proper subbundle are thus of one of the following types:
\begin{enumerate}
\item  $V_1=0$ and $V_2=L$ where $L$ is a  line subbundle, or 
\item  $V_2=V$ and $V_1=L$ where $L$ is a  line subbundle, or
\item $V_1=V_2=L$ where $L$ is a  line subbundle.
\end{enumerate}
The corresponding conditions in Lemma \ref{lemma:2-step} for such two-step
filtration to be $\varphi$-invariant are:
\begin{enumerate}
\item $\beta(L^{\perp})=0$ if $V_1=0$ and $V_2=L$,
\item $\gamma(L)=0$ if $V_1=L$ and $V_2=V$, and
\item $\beta(L^{\perp})\subset L\otimes K$ and $\gamma(L)\subset
  L^{\perp}\otimes K$ if $V_1=V_2=L$.
\end{enumerate}

\begin{remark}\label{remark:conditions}
  In case (1) the condition $\beta(L^{\perp})=0$ implies that
  $\beta:V^*\rightarrow V\otimes K$ has rank less than two.  The skew
  symmetry of $\beta$ thus forces $\beta=0$.  Similarly, in case (2),
  $\gamma(L)=0$ implies that $\gamma=0$. In case (3), the skew symmetry
  of $\beta$ and $\gamma$ ensure that the conditions
  $\beta(L^{\perp})\subset L\otimes K$ and $\gamma(L)\subset
  L^{\perp}\otimes K$ apply {\it for all} line subbundles $L\subset
  V$.
\end{remark}
The stability condition for $\SO^*(4)$-Higgs bundles thus reduces to
the following.

\begin{proposition}\label{prop:simplifiedss2}
  A $\SO^*(4)$-Higgs bundle $(V,\beta,\gamma)$ with $\deg(V)>0$ is
  (semi)stable if and only if $V$ is (semi)stable as a bundle and
  $\gamma\ne 0$.

A $\SO^*(4)$-Higgs bundle $(V,\beta,\gamma)$ with $\deg(V)<0$ is (semi)stable if and only if $V$ is (semi)stable as a bundle and $\beta\ne 0$.

A $\SO^*(4)$-Higgs bundle $(V,\beta,\gamma)$ with $\deg(V)=0$ is (semi)stable if and only if $V$ is (semi)stable as a bundle.
\end{proposition}

\begin{proof} Suppose that $(V,\beta,\gamma)$ is a (semi)stable
  $\SO^*(4)$-Higgs bundle with $\deg(V)=d$. By \eqref{eqn: range}, if
  $d>0$ then $\gamma$ cannot be zero and if $d<0$ then $\beta$ cannot
  be zero. If $d=0$ then ( see Remark \ref{remark: zero}) there is no
  restriction on $\beta$ or $\gamma$.  Any line subbundle $L\subset V$
  defines a $\varphi$-invariant two-step filtration in which
  $V_1=V_2=L$. Applying Proposition (\ref{prop:simplifiedss}) we see
  that if $(V,\beta,\gamma)$ is semistable then $\deg(L)\le\deg(V)/2$,
  and the inequality is strict if $(V,\beta,\gamma)$ is stable.  This
  proves the `only if' direction.

To prove the converse it remains to check that the inequalities
\eqref{eq:spn-simplified-sstab} and \eqref{eq:spn-simplified-stab} are
satisfied by $\varphi$-invariant two-step filtrations of the form  (a)
$V_1=0$, $V_2=L$ or  (b) $V_1=L$, $V_2=V$. By Remark
\ref{remark:conditions}, the first case occurs only if $\beta=0$ and hence,
by \eqref{eqn: range}, $\deg(V)\ge 0$.  Thus in this case
\begin{displaymath}
  \deg(L)\le\deg(V)/2\implies \deg(L)\le\deg(V).
\end{displaymath}
\noi Similarly, the second case occurs only if $\gamma=0$ and hence $\deg(V)\le 0$. Thus
\begin{displaymath}
\deg(L)\le\deg(V)/2\implies \deg(L)\le 0.
\end{displaymath}
\noi The requisite inequalities thus follow from the (semi)stability of $V$.
\end{proof}

From Proposition \ref{prop:SO-star-poly-stability} we have the following.

\begin{proposition}\label{prop:poly4} A $\SO^*(4)$-Higgs bundle $(V,\beta,\gamma)$ is polystable if and only if
\begin{enumerate}
\item it is stable with $\varphi\ne 0$, or
\item $V$ decomposes as a sum of two line bundles of degree zero and $\beta=\gamma=0$, or
\item $V=L_1\oplus L^*_2$ with $\deg(L_1)=-\deg(L_2)$ and with respect to this decomposition $\beta=\begin{pmatrix}0&\tilde{\beta}\\-\tilde{\beta}&0 \end{pmatrix}$ and $\gamma=\begin{pmatrix}0&\tilde{\gamma}\\-\tilde{\gamma}&0 \end{pmatrix}$.
\end{enumerate}
\end{proposition}

\begin{corollary} Let $M_d(2)$ denote the moduli space of rank 2,
  degree $d$ semistable bundles and let $M^s_d(2)\subset M_d(2)$ be
  the stable locus.  There is a map
\begin{align}\label{eqn:so4map}
\mathcal{M}_d(\SO^*(4))&\longrightarrow M_d(2)\\
[V,\beta,\gamma]& \mapsto [V]\nonumber
\end{align}
\begin{enumerate}
\item If $d>0$ then the image of the map is the locus of bundles for which
  $h^0(\det(V)^{-1}\otimes K)$ is greater than zero. The fiber over $[V]\in
  M^s_d(2)$ can be identified with $\mathcal{O}_{\PP^s}(1)^{\oplus r}$ where
  $r=h^0(\det(V)\otimes K)$ and $s=h^0(\det(V)^{-1}\otimes K)$.
\item If $d<0$ then the image is the locus of bundles for which
  $h^0(\det(V)\otimes K)$ is greater than zero.  The fiber over $[V]\in
  M^s_d(2)$ can be identified with 
$\mathcal{O}_{\PP^r}(1)^{\oplus s}$ where $r=h^0(\det(V)\otimes K)$ and
  $s=h^0(\det(V)^{-1}\otimes K)$.
\item If $d=0$ then the map is surjective.  
\end{enumerate}
\end{corollary}

\begin{proof} Everything is immediate from Propositions \ref{prop:simplifiedss2} and \ref{prop:poly4} except for the description of the fibers. 

Suppose that $d>0$ and consider the fiber over a point in ${M}_d(2)$
represented by the bundle $V$. The $\SO^*(4)$-Higgs bundles $(V,\beta,\gamma)$
are semistable for all $(\beta,\gamma)\in H^0(X,\det(V)\otimes K)\oplus
(H^0(X,\det(V)^{-1}\otimes K)-\{0\}$. However, since the points in $\mathcal{M}_d(\SO^*(4))$ are isomorphism classes of objects, we need to consider when two objects, say $(V,\beta,\gamma)$ and $(V,\beta',\gamma')$, are isomorphic as $\SO^*(4)$-Higgs bundles. By definition the object are isomorphic if there exists a bundle automorphism $f:V\rightarrow V$ such that $f^*(\beta')=\beta$ and $f^*(\gamma')=\gamma$.  But if $V$ is stable, then the only automorphisms are multiples of the identity, say $f=tI$, and the induced map on $\beta$ and $\gamma$ is 
\begin{equation}
f^*(\beta)=t^2\beta\ ,\ f^*(\gamma)=t^{-2}\gamma
\end{equation}
\noi The fiber over $[V]\in M_d(2)$ is thus given by $(H^0(X,\det(V)\otimes
K)\oplus (H^0(X,\det(V)^{-1}\otimes K)-\{0\}))/\C^*$ where the $\C^*$-action is
given by $t(\beta,\gamma)=(t^2\beta,t^{-2}\gamma)$. 
The results follows from this. 

The description of the fibers in the $d<0$ case is similar.
\end{proof}

\begin{remark}\hfill

\begin{enumerate}
\item Brill-Noether theory shows that in fact the map is surjective for all $d<(g-1)$.
\item If $\deg(V)$ is odd then $M_d(2)=M^s_d(2)$, so all fibers are direct sums of copies of the degree one line bundle over a suitable projective space. Note, though, that the number of summands and the dimension of the projective space need not be constant.
\item  In the case $d=0$, the fiber over a point $[V]\in M_d(2)$ is the quotient 
$$(H^0(X,\det(V)\otimes K)\oplus H^0(X,\det(V)^{-1}\otimes K))/\C^*\ .$$
\end{enumerate}
\end{remark}

\subsubsection{Simplicity and smoothness in $\mathcal{M}_d(\SO^*(4))$}  Applying Theorem \ref{thm:stable-not-simple-spnr-higgs} to the case of $\SO^*(4)$-Higgs bundles yields:

\begin{theorem}
Let $(V,\varphi)$ be a stable $\SO^*(4)$-Higgs bundle. If
$(V,\varphi)$ is not simple, then  $V$ is a stable vector bundle of degree zero and $\varphi=0$.  In this case $\Aut(V,\varphi)\simeq\C^*$.
\end{theorem}

\begin{proof} Theorem~\ref{thm:stable-not-simple-spnr-higgs} says that
  there are two alternatives for stable $\SO^*(2n)$-Higgs bundle which
  are not simple and we wish to show that alternative (1) occurs when
  $n=2$. To exlude alternative (2) we note that it requires
  $(V,\varphi)$ to decompose into two $\SO^*(2)$-Higgs bundles with
  non-zero Higgs fields. This is impossible since the Higgs field
  necessarily vanishes in a $\SO^*(2)$-Higgs bundle.
\end{proof}

By Proposition \ref{cor:smooth-points} a stable and simple $\SO^*(4)$-Higgs bundle $(V,\beta,\gamma)$ represents a smooth point in $\mathcal{M}_d(\SO^*(4))$ (where $d=\deg(V)$) unless $d=0$ and there is a skewsymmetric isomorphism $f\colon V \xrightarrow{\simeq} V^*$ intertwining $\beta$ and
$\gamma$.  By Lemma \ref{lemma:SL2} such an isomorphism can exist only if $\det(V)=\mathcal{O}$.  We thus get:

\begin{proposition} 
\begin{enumerate}
\item If $d$ is odd then $\mathcal{M}_d(\SO^*(4))$ is smooth.
\item If $d$ is even and $d\ne 0$ then $\mathcal{M}_d(\SO^*(4))$ is
  smooth except possibly at points represented by $\SO^*(4)$-Higgs
  bundles $(V,\beta,\gamma)$ of the form:
  \begin{equation}
    \label{eq:1}
    \begin{aligned}
      V&=L_1\oplus L^*_2, \quad \text{with $\deg(L_1)=-\deg(L_2)$ and, with respect to this decomposition,} \\
      \beta&=\begin{pmatrix}0&\tilde{\beta}\\-\tilde{\beta}&0 \end{pmatrix}
      \text{ and }
       \gamma=\begin{pmatrix}0&\tilde{\gamma}\\-\tilde{\gamma}&0 \end{pmatrix}
    \end{aligned}
  \end{equation}
  \item If $d=0$ then $\mathcal{M}_d(\SO^*(4))$ is smooth except
    possibly at points represented by  $\SO^*(4)$-Higgs bundles
    $(V,\beta,\gamma)$ such that 
\begin{enumerate}
\item $\beta=\gamma=0$, or
\item $(V,\beta,\gamma)$ is of the form \eqref{eq:1}, or
\item $\det(V)=\mathcal{O}$ and $f\beta=f^{-1}\gamma$ where $f:V\xra{\simeq} V^*$ is a skew-symmetric isomorphism.
\end{enumerate}
\end{enumerate}
\end{proposition}

\begin{proof} 
(1) If $d$ is odd then all semistable and polystable Higgs bundles are stable, simple and do not admit a skew-symmetric isomorphism intertwining the components of the Higgs field.

(2) If $d$ is even and $d\ne 0$ then all stable Higgs bundles are simple and do not admit a skew-symmetric isomorphism intertwining the components of the Higgs field. The non-smooth points can occur only at points represented by polystable Higgs bundles.

(3) The cases (a)-(c) correspond to polystable Higgs bundles (cases (a) and (b)), stable but not simple Higgs bundles (case (a)), or stable and simple bundles which admit a skew-symmetric isomorphism intertwining the components of the Higgs field (case (c)).
\end{proof}

\subsubsection{The even degree case} Notice that if $V$ is a rank 2 bundle, then $\Lambda^2(V)=\det(V)$. Furthermore if $\deg(V)$ is even then  $V$ can be decomposed as
\begin{equation}\label{eq:2}
V=U\otimes L\ ,\ \mathrm{with}\ \begin{cases}\det(U) \simeq \mathcal{O}\\
L^2=\det(V)\end{cases}.
\end{equation}

\begin{lemma}\label{lemma:SL2} If $U$ is a rank 2 holomorphic bundle then the following are equivalent: 
\begin{enumerate}
\item  $\det(U)\simeq \mathcal{O}$,
\item the structure group of $U$ reduces to $\SL(2,\C)$,
\item  $U^*\simeq U$, with the isomorphism defined by a symplectic
  form $\Omega\in
  H^0(X,\Lambda^2U^*)$.
\end{enumerate}
\end{lemma}

\begin{proof} The equivalence of (1) and (2) is straightforward.  The equivalence of (2) and (3) follows from the fact that $\SL(2,\C)\simeq\Sp(2,\C)$.
\end{proof}

\begin{lemma}\label{lemma:tilde}
  Let $(V,\beta,\gamma)$ be a $\SO^*(4)$-Higgs bundle with $\deg(V)$
  even. Let $V=U\otimes L$ as in \eqref{eq:2}, and let $\Omega\in
  H^0(X,\Lambda^2U^*)$ be the symplectic form on $U$ given by (3) of
  Lemma~\ref{lemma:SL2}, with induced symplectic form $\Omega^*\in
  H^0(X,\Lambda^2U)$ on $U^*$.  Then we can write
  \begin{equation}
\begin{aligned}
\beta&=\Omega\otimes\tilde{\beta}\ ,\ \mathrm{where}\ \tilde{\beta}\in
H^0(X,L^2\otimes K),\\
\gamma&=\Omega^*\otimes\tilde{\gamma}\ , \ \mathrm{where}\ \tilde{\gamma}\in 
H^0(X,L^{-2}\otimes K).
\end{aligned}
\end{equation}
\end{lemma}

\begin{proof} Immediate from $\Lambda^2
  V \simeq \Lambda^2U \otimes L^2\simeq L^2$ and the existence of the
  nowhere vanishing
  section $\Omega$ of $\Lambda^2U$. 
\end{proof}

Applying Definition \ref{def:g-higgs} to the case $G=\SL(2,\R)$, a
$\SL(2,\R)$-Higgs bundle can be described as a triple
$(L,\beta,\gamma)$ where $L$ is a line bundle and $\beta\in
H^0(X,L^-2K), \gamma\in H^0(X,L^2K)$.  We denote by
$\mathcal{M}_l(\SL(2,\R))$ the component of the moduli space of
polystable $\SL(2,\R)$-Higgs bundles in which $\deg(L)=l$.

The following result shows that $\SO^*(4)$-Higgs bundles of even
degree are intimately related to $\SL(2,\R)$-Higgs bundles.

\begin{proposition}\label{prop:even} Let $(V,\beta,\gamma)$ be a
  $\SO^*(4)$-Higgs bundle with $deg(V)$ even.   Pick $L$ such that $L^2=
\det(V)$ and define $U=V\otimes L^{-1}$.  Then 

\begin{enumerate}
\item $U$ is a $\SL(2,\C)$-bundle and 
\item  $(L,\tilde{\beta},\tilde{\gamma})$ defines a $\SL(2,\R)$-Higgs bundle 
\end{enumerate}
\noindent where  $\tilde{\beta},\tilde{\gamma}$ are as in Lemma \ref{lemma:tilde}. The $\SO^*(4)$-Higgs bundle $(V,\beta,\gamma)$ is (semi)stable if and only if $U$ is (semi)stable as a bundle and $(L,\tilde{\beta},\tilde{\gamma})$ is (semi)stable as a $\SL(2,\R)$-Higgs bundle. \end{proposition}

\begin{proof} Properties (1) and (2) follow from Lemmas \ref {lemma:SL2} and \ref{lemma:tilde}, and the fact that a triple $(L,\tilde{\beta},\tilde{\gamma})$ (as in Lemma \ref{lemma:tilde}) defines a $\SL(2,\R)$-Higgs bundle. The statement about (semi)stability follows from Proposition \ref{prop:simplifiedss2} and the fact that (semi)stability for a $\SL(2,\R)$-Higgs bundle $(L,\tilde{\beta},\tilde{\gamma})$ with $\deg(L)\ge 0$ is equivalent to the condition that $\tilde{\gamma}\ne 0$ (if $\deg(L)>0$).
\end{proof}

\begin{remark}The isomorphism \eqref{so4isom} is the infinitesimal version of a 2:1 homomorphism
\begin{equation}
\eta: \SU(2)\times\SL(2,\R)\longrightarrow\SO^*(4)\ .
\end{equation}

\noi Proposition \ref{prop:even} shows that if $\deg(V)$ is even then
the structure group of the $\SO^*(4)$-Higgs bundle lifts via $\eta$ to
$\SU(2)\times\SL(2,\R)$. If $\deg(V)$ is odd, then the structure group
does not lift. The obstruction to the lift can be viewed as an element of
$H^2(X,\Z/2)$. In fact, the homomorphism $\eta$ is induced by the homomorphism  
$\Spin(4,\C)\longrightarrow\SO(4,\C)$. To see this, recall that 
$$ 
\Spin(4,\C) \simeq \Spin(3,\C) \times  \Spin(3,\C) \simeq \SL(2,\C)\times
\SL(2,\C).
$$
Under this homomorphism, the real form $\SU(2)\times\SL(2,\R)$ of $\SL(2,\C)\times
\SL(2,\C)$ maps to $\SO^*(4)$.

\end{remark}

\subsubsection{The Cayley partner} Applying Proposition \ref{mw-higgs} with $n=2$, we see that 
$$|\deg(V)|\le 2g-2$$
\noi and that $\gamma$ is an isomorphism if (and only if) $\deg(V)=2g-2$.   As
in Proposition \ref{prop:even} we write $V=U\otimes L$ with
$\det(U)=\mathcal{O}$ and $L^2=\det(V)$. In particular, if $\deg(V)=2g-2$ then
$\deg(L^{-2}\otimes K)=0$. Moreover,  since $\gamma$ is an isomorphism, it
follows that $\tilde{\gamma}$ is 
a non-zero section of $L^{-2}\otimes K$ and thus $L^2=K$.  Proposition \ref{prop:even} thus becomes the following.

\begin{proposition}\label{prop:cayley2} Let $(V,\beta,\gamma)$ be a $\SO^*(4)$-Higgs bundle with $\deg(V)=2g-2$.   Pick $L$ such that $L^2=K$ and define $U=V\otimes L^{-1}$.  Then 
\begin{enumerate}
\item $U$ is a $\SL(2,\C)$-bundle and 
\item  $(L,\tilde{\beta},\tilde{\gamma})$ defines a $\SL(2,\R)$-Higgs bundle where $\tilde{\gamma}$ is a non-zero section in $H^0(X,\mathcal{O})$, and $\tilde{\beta}\in H^0(X,K^2)$.  In particular, $(L,\tilde{\beta},\tilde{\gamma})$ defines a Higgs bundle in a Teichm\"uller component of $\mathcal{M}_{g-1}(\SL(2,\R))$.
\end{enumerate}
\noindent    
Moreover, the polystability of  $(V,\beta,\gamma)$ is equivalent to the polystability of $U$.
\end{proposition}

\begin{remark}  With $\Omega$ as  Lemma \ref{lemma:tilde}, the data
  $(U,\Omega; \tilde{\beta})$ as in Proposition \ref{prop:cayley2}  defines  a
  $K^2$-twisted $\U^*(2)$-Higgs bundle.
Indeed if $(V,\Omega; \varphi)$ is a $L$-twisted $U^*(2)$-Higgs bundle then we
can assume that $\Omega=J$ with respect to suitable local frames. Since, by
definition of a $\U^*(2n)$-Higgs bundle, $\varphi^t\Omega=-\Omega\varphi$, we get
that $\varphi=\tilde{\varphi}I$ with respect to the same frames. It follows
that locally $\varphi=\tilde{\varphi}I$, where $\tilde{\varphi}\in H^0(X,L)$ 
(see Appendix~\ref{appendix:G-Higgs} and
\cite{garcia-prada-oliveira:2010} for details on $\U^*(2n)$-Higgs
bundles). 
 The polystability of the $(U,\tilde{\beta})$ as a $K^2$-twisted $\U^*(2)$-Higgs bundle is equivalent to the polystability of $U$.  
\end{remark}

\begin{remark} The ambiguity in the decomposition $V=U\otimes L$ corresponds, in this case, to the choice of a square root of $K$. This is the same choice as the one which distinguishes the Teichm\"uller component of $\mathcal{M}_{g-1}(\SL(2,\R))$.
\end{remark}

Combining Propositions \ref{prop:even} and \eqref{prop:simplifiedss2}  gives rise to a $2^{2g}:1$ map
\begin{align}\label{n=2map}
T: M_0(2)\times \mathcal{M}_l(\SL(2,\R))&\longrightarrow \mathcal{M}_{2l}(\SO^*(4))\nonumber\\
([U],[L,\tilde{\beta},\tilde{\gamma}])&\mapsto [U\otimes L, \beta,\gamma]
\end{align}
\noi where $M_0(2)$ denotes the moduli space of polystable rank 2 bundles with trivial determinant. This is the Higgs bundle manifestation of isomorphism \eqref {so4isom}.

\begin{proposition} For each $0\le l\le g-1$ the moduli space  $\mathcal{M}_{2l}(\SO^*(4))$ is connected.
\end{proposition}
\begin{proof}\label{rem: even}  Under the map $T$, the $2^{2g}$ Teichm\"uller components in $\mathcal{M}_{g-1}(\SL(2,\R))$ are all identified in the component $\mathcal{M}_{2g-2}(\SO^*(4))$. 
For $0 \le l < g-1$ the moduli spaces  $\mathcal{M}_{l}(\SL(2,\R))$ are connected.
\end{proof}





\subsection{The case $n=3$}
The Lie algebra of $\SO^*(6)$ is isomorphic to $\mathfrak{su}(1,3)$,
the Lie algebra of $\SU(1,3)$. The groups differ because they have
different centers, with $Z(\SO^*(6))\simeq\Z/2$ and $Z(\SU(1,3))\simeq
\Z/4$. Both groups are finite covers of $\PU(1,3)$, the adjoint form
of the Lie algebra.  The relationships among the groups
$\SO^*(6),\SU(1,3)$, and $\PU(1,3)$ leads to relations among the
corresponding Higgs bundles for the groups (see Proposition
\ref{Higgs-lifts}). As in the case of $\SO^*(4)$, the relation can be
explained in terms of the spin group. Namely, the $2:1$ homomorphism
$\Spin(6,\C) \longrightarrow \SO(6,\C)$ restricts to a $2:1$
homomorphism $\Spin^*(6) \longrightarrow \SO^*(6)$. But under the
isomorphism $\Spin(6,\C)\simeq \SL(4,\C)$, one has the isomorphism of
the corresponding real forms $\Spin^*(6)$ and $\SU(1,3)$.

The key to understanding the relation between the Higgs bundles is the isomorphism 
$$\Lambda^k(\mathbf{V}^*)\otimes\Lambda^n(\mathbf{V})\longrightarrow \Lambda^{n-k}(\mathbf{V}).$$
\noi where $\mathbf{V}$ is a vector space of dimension $n>k$, and the
map is defined by the interior product.  This extends to exterior
powers of vector bundles of rank $n$.  In particular, if $n=3$ and
$k=2$ we get $\Lambda^2 V^*\otimes\det(V)\simeq V$ or equivalently 
\begin{equation}
\Lambda^2 V^*\simeq\det(V)^*\otimes V\simeq\Hom(\det(V),V).
\end{equation}
Hence sections $\gamma\in H^0(X,\Lambda^2V^*\otimes K)$ and $\beta\in
H^0(X,\Lambda^2V\otimes K)$ define holomorphic
bundle maps $\tilde{\gamma}:\det(V)\rightarrow V\otimes K$ and $\tilde{\beta}:V\rightarrow
\det(V)\otimes K$ by
\begin{equation}\label{tildeg}
  \begin{aligned}
    \tilde{\gamma}(\omega)&=\iota_{\gamma}(\omega), \\
    \tilde{\beta}(v)&=\beta\wedge v,
  \end{aligned}
\end{equation}
\noi where $\iota_{\gamma}$ denotes interior
product.

\begin{proposition}\label{prop:SO6toU13}
A $\SO^*(6)$-Higgs bundle defines a $\U(1,3)$-Higgs bundle via the map
\begin{equation}\label{eqn:SO6toU13}
(V,{\beta},{\gamma})\mapsto (\det(V),V, \tilde{\beta},\tilde{\gamma})
\end{equation}
\noi where $\tilde{\beta}$ and $\tilde{\gamma}$ are related to $\beta$
and $\gamma$ as in (\ref{tildeg}).
\end{proposition}

\begin{proof}This follows immediately from the definitions. In general, a
  $\U(p,q)$-Higgs bundle is defined by a tuple $(V,W,\beta,\gamma)$ where $V$
  and $W$ are bundles of rank $p$ and $q$ respectively, and $\beta,\gamma$ are
  maps $\beta:V\rightarrow W\otimes K$ and $\gamma:W\rightarrow V\otimes K$ (see \cite{bradlow-garcia-prada-gothen:2003} and Section \ref{A:U(p,q)} for more details). \end{proof}

\begin{remark}\label{rem:defs} We refer the reader to \cite{bradlow-garcia-prada-gothen:2003} and Section \ref{A:U(p,q)} for more details but note here the following key features:
\begin{enumerate}
\item The tuple $(V,W,\beta,\gamma)$ represents a $\SU(p,q)$-Higgs bundle if
  it satisfies the determinant condition $\det(V\oplus W)=\mathcal{O}$. In
  particular, $\SU(1,3)$-Higgs bundles are represented by tuples
  $(L,W,\tilde{\beta},\tilde{\gamma})$ with $L$ a line bundle, $W$ a rank
  three bundle, $\tilde{\beta}:W\rightarrow L \otimes K$ and
  $\tilde{\gamma}:L\rightarrow W \otimes K$ and such that $\det(L\oplus W)$ is trivial. 

\item While a $\PU(p,q)$-Higgs bundle is defined by a principal $\mathrm{P}(\U(p)\times\U(q))$-bundle together with an appropriate Higgs field, the structure group of the bundle can always be lifted to $\U(p)\times\U(q)$. Together with the Higgs field, the  principal $\U(p)\times\U(q)$-bundle defines a $\U(p,q)$-Higgs bundle. The lifts are defined up to a twisting by a line bundle.
\item The notion of polystability and the corresponding Hitchin equations for
  $\U(p,q)$-Higgs bundles are described in 
Section \ref{A:U(p,q)} and in \cite{bradlow-garcia-prada-gothen:2003}. The notions for $\SU(p,q)$ and $\PU(p,q)$ are similar. 

\item 
\begin{enumerate}
\item The components of the moduli space of polystable $\U(p,q)$-Higgs bundles are labeled by the integer pair $(a,b)$ where $a=\deg(V)$ and $b=\deg(W)$. We will denote these components by $\mathcal{M}_{a,b}(\U(p,q))$.  
\item For a $\PU(p,q)$-Higgs bundle,  the components of the moduli spaces are labeled by the combination $\tau=2\frac{aq-bp}{p+q}$, where $(V,W,\beta,\gamma)$ represents a $\U(p,q)$-Higgs bundles obtained by lifting the structure group. This combination, known as  the Toledo invariant, is independent of the lifts to $\U(p,q)$.  We will denote the components with Toledo invariant $\tau$ by $\mathcal{M}_{\tau}(\PU(p,q))$.  
\item For $\SU(p,q)$-Higgs bundles, for which $\deg(V)=-\deg(W)$, the components of the moduli space can be labeled by the single integer $a=\deg(V)$.  We will denote these components by $\mathcal{M}_{a}(\SU(p,q))$. 
\end{enumerate}
\end{enumerate}
\end{remark}

\begin{proposition}\label{prop:eqequiv} Let $(V,{\beta},{\gamma})$ and
  $(\det(V),V, \tilde{\beta},\tilde{\gamma})$ be a $\SO^*(6)$-Higgs
  bundle and corresponding $\U(1,3)$-Higgs bundle, as in (\ref{eqn:SO6toU13}). Then the following are equivalent:
\begin{itemize}
  \item [(A)] The bundle $V$ admits a metric, say $H$, satisfying
    the $\SO^*(6)$-Hitchin equation on $(V, \beta, \gamma)$, namely
    (see \eqref{eqn:SO*Hitch})
\begin{equation}\label{eqn:SO*6Hitch}
  F^H_{V}+\beta\beta^{*_H}+\gamma^{*_H}\gamma = 0.
\end{equation}
\item  [(B)] The bundles $V$ and $\det(V)$ admit metrics, say $K$ and $k$,
satisfying the $\U(1,3)$-Hitchin equation on $(\det(V), V,
\tilde{\beta}, \tilde{\gamma})$, namely (see
\cite{bradlow-garcia-prada-gothen:2003})
\begin{equation}\label{eqn:U13Hitch}
  \begin{aligned}
    F^K_{V}+\tilde{\beta}^{*_{K,k}}\tilde{\beta}+\tilde{\gamma}\tilde{\gamma}^{*_{K,k}} &=-\sqrt{-1} \mu\mathbf{I}_{V}\omega,\\
    F^k_{\det(V)}+\tilde{\beta}\tilde{\beta}^{*_{K,k}}+\tilde{\gamma}^{*_{K,k}}\tilde{\gamma}
    &= -\sqrt{-1}\mu\omega.
  \end{aligned}
\end{equation}
\end{itemize}
\noi In these equations 
\begin{itemize}
\item the first terms denote the curvature of the Chern connection with respect to the indicated metrics, 
\item the adjoints in \eqref{eqn:SO*6Hitch} are with respect to $H$ and the metric it induces on $V^*$, 
\item the adjoints in \eqref{eqn:U13Hitch} are with respect to $K$ and $k$
\item $\mu=\dfrac{\sqrt{-1}\int_X\Tr(F^H_V)}{2\mathrm{Vol}(X)}=\dfrac{\pi\deg(V)}{\mathrm{Vol}(X)}$, 
\item $\mathbf{I}_V$ is the identity map on $V$, and
\item $\omega$ denotes the K\"ahler form of the metric on the Riemann surface $X$.
\end{itemize}
\end{proposition}

The proof of Proposition \ref{prop:eqequiv} uses the following technical Lemma.

\begin{lemma}\label{lemma: confchange} Let $(\det(V),V,
  \tilde{\beta},\tilde{\gamma})$ be a $\U(1,3)$-Higgs bundle, as
  in (\ref{eqn:SO6toU13}). Let $H$ and $h$ be any metrics on $V$ and $\det(V)$
  respectively. Let $K$ be a metric on $V$ which is related to $H$ by
  a conformal factor $e^u$, i.e. $K(\phi,\psi)=e^uH(\phi,\psi)$ for
  any sections $\phi$ and $\psi$ of $V$.  Similarly let $k$ be a
  metric on $\det(V)$ which is related to $h$ by the same conformal
  factor $e^u$. Then (in the notation of Proposition
  \ref{prop:eqequiv}, and denoting by $*_{H,h}$ adjoints with respect
  to $H$ and $h$)
\begin{enumerate}
\item $\tilde{\gamma}^{*_{K,k}}=\tilde{\gamma}^{*_{H,h}}$,
\item $\tilde{\beta}^{*_{K,k}}=\tilde{\beta}^{*_{H,h}}$,
\item $F^K_V=F^H_V-\sqrt{-1}\Delta(u)\omega\mathbf{I}_V$, and
\item $ F^{k}_{\det(V)}=F^{h}_{\det(V)}-\sqrt{-1}\Delta(u)\omega$.
\end{enumerate}
\noi where in (3) and (4) $\omega$ denotes the K\"ahler form on $X$.
\end{lemma}

\begin{proof}
Let $a$ be a point in the fiber of $V$ over a point $x\in X$ and let $b$ be a point in the fiber over $x$ of $\det(V)\otimes\bar{K}$.  Then
\begin{align*}
h(b,\tilde{\gamma}^{*_{K,k}}(a))=&e^{-u(x)}k(b,\tilde{\gamma}^{*_{K,k}}(a))\\
=&e^{-u(x)}K(\tilde{\gamma}(b),a)\\
=&e^{-u(x)}e^{u(x)}H(\tilde{\gamma}(b),a)=h(b,\tilde{\gamma}^{*_{H,h}}(a)).
\end{align*}
\noi This proves (1). The proof of (2) is similar.  The proof of (3) and (4) follows directly from the definition of the Chern connection.  Indeed, if metrics $H_1$ and $H_2$ on a holomorphic bundle $E$ are related by $H_1=H_2s$ where $s$ is a (positive definite) automorphism of $E$, then the curvatures of the Chern connections are related by
\begin{equation}
F_{H_1}=F_{H_2}+\bar{\partial}_E(s^{-1}D'_{H_1}(s))
\end{equation}
\noi where $\bar{\partial}_E$ and $D'_{H_1}$ are the antiholomorphic and holomorphic parts of the Chern connection for $H_1$.  If $s=e^u\mathbf{I}$ then the second term reduces to $-\sqrt{-1}\Delta(u)\omega$.
\end{proof}

We now prove Proposition \ref{prop:eqequiv} .

\begin{proof}[Proof of Proposition \ref{prop:eqequiv}]
Fix a local frame for $V$ and use the dual frame for $V^*$.  Also, fix a local
complex coordinate on the base. Then $\gamma$, as a map from $V$ to
$V^*\otimes K$ is given locally by a matrix of holomorphic 1-forms, which we write as 
\begin{equation}
\gamma=\begin{bmatrix}0&\gamma_1&\gamma_2\\-\gamma_1&0&\gamma_3\\ -\gamma_2&-\gamma_3&0\end{bmatrix}dz\ .
\end{equation}
\noi Using the induced frame for $\det(V)$, the map $\tilde{\gamma}$ is then given by
\begin{equation}
\tilde{\gamma}=\begin{bmatrix}\gamma_3\\-\gamma_2\\ \gamma_1\end{bmatrix}dz\ .
\end{equation}

Similarly, if $\beta$ as a map from $V^*$ to $V\otimes K$ is given locally by a matrix of holomorphic 1-forms of the form
\begin{equation}
\beta=\begin{bmatrix}0&\beta_1&\beta_2\\-\beta_1&0&\beta_3\\ -\beta_2&-\beta_3&0\end{bmatrix}dz\ .
\end{equation}
\noi then the map $\tilde{\beta}$ is then given by
\begin{equation}
\tilde{\beta}=\begin{bmatrix}\beta_3&-\beta_2& \beta_1\end{bmatrix}dz\ .
\end{equation}

Given a metric, say $H$, on $V$, we can pick the local frame to be unitary with respect to $h$.  Then locally
\begin{equation}
\gamma^{*_H}=\begin{bmatrix}0&-\bar{\gamma}_1&-\bar{\gamma}_2\\\bar{\gamma}_1&0&-\bar{\gamma}_3\\ \bar{\gamma}_2&\bar{\gamma}_3&0\end{bmatrix}d\bar{z}\ .
\end{equation}
\noi The metric $H$ induces a metric on $\det(V)$, which we denote by $h$. With respect to the metrics $H$ on $V$ and $h$ on $\det(V)$, the adjoint of $\tilde{\gamma}$ is given locally by
\begin{equation}
\tilde{\gamma}^{*_{H,h}}=\begin{bmatrix}\bar{\gamma}_3&-\bar{\gamma}_2&\bar{\gamma}_1\end{bmatrix}d\tilde{z}\ .
\end{equation}

Using the metrics $H$ and $h$, and taking into account that the entries in the matrix are 1-forms,  we get that 
\begin{equation}\label{eqn:reltns}
  \begin{aligned}
    \gamma^{*_H}\gamma &=\tilde{\gamma}\tilde{\gamma}^{*_{H,h}}+\tilde{\gamma}^{*_{H,h}}\tilde{\gamma}\mathbf{I}_V,\\
    \beta\beta^{*_H}&=\tilde{\beta}^{*_{H,h}}\tilde{\beta}+\tilde{\beta}\tilde{\beta}^{*_{H,h}}\mathbf{I}_V,
  \end{aligned}
\end{equation}
\noi and also
\begin{equation}\label{eqn:Tracereltns}
  \begin{aligned}
    \Tr(\tilde{\gamma}\tilde{\gamma}^{*_{H,h}})&=-\tilde{\gamma}^{*_{H,h}}\tilde{\gamma},\\
    \Tr(\tilde{\beta}^{*_{H,h}}\tilde{\beta})&=-\tilde{\beta}\tilde{\beta}^{*_{H,h}}.
  \end{aligned}
\end{equation}
\noi Suppose that $V$ admits a metric which satisfies the
$\SO^*(6)$-Hitchin equations for $(V,\beta,\gamma)$, namely equation
\eqref{eqn:SO*6Hitch}.  Because of \eqref{eqn:reltns} this is
equivalent to
\begin{equation}\label{eqn:U13HitchA}
F^H_V+\tilde{\beta}^{*_{H,h}}\tilde{\beta}+\tilde{\gamma}\tilde{\gamma}^{*_{H,h}} =-(\tilde{\gamma}^{*_{H,h}}\tilde{\gamma}+\tilde{\beta}\tilde{\beta}^{*_{H,h}})\mathbf{I}_V.
\end{equation}
\noi Taking the trace of this, and using \eqref{eqn:Tracereltns}, we also get
\begin{equation}\label{eqn:U13HitchB}
\Tr(F^h_V)+\tilde{\gamma}^{*_{H,h}}\tilde{\gamma}+\tilde{\beta}\tilde{\beta}^{*_{H,h}} =-(\tilde{\gamma}^{*_{H,h}}\tilde{\gamma}+\tilde{\beta}\tilde{\beta}^{*_{H,h}}).
\end{equation}
\noi We can write the $(1,1)$-form $\tilde{\gamma}^{*_{H,h}}\tilde{\gamma}+\tilde{\beta}\tilde{\beta}^{*_{H,h}}$ as
\begin{equation}\label{eqn:RHS}
\tilde{\gamma}^{*_{H,h}}\tilde{\gamma}+\tilde{\beta}\tilde{\beta}^{*_{H,h}}=\sqrt{-1}t\omega=-(\sum_{i=1}^3|\tilde{\gamma}_i|^2-\sum_{i=1}^3|\tilde{\beta}_i|^2)dz\wedge d\bar{z}
\end{equation}
\noi where the last expression is in local coordinates. Notice that by \eqref{eqn:U13HitchB} we get
\begin{equation}
-2\sqrt{-1}\int_Xt\omega=\int_X\Tr(F^h_V)=-2\pi\sqrt{-1}\deg(V).
\end{equation}
\noi Since $\Tr(F^H_V)=F^{h}_{\det(V)}$, equations \eqref{eqn:U13HitchA} and \eqref{eqn:U13HitchB} can thus be written as
\begin{equation}\label{eqn:U13Hitch-t}
  \begin{aligned}
    F^H_V+\tilde{\beta}^{*_{H,h}}\tilde{\beta}+\tilde{\gamma}\tilde{\gamma}^{*_{H,h}} &=-\sqrt{-1}t\omega\mathbf{I}_V\\
    F^{h}_{\det(V)}+\tilde{\gamma}^{*_{H,h}}\tilde{\gamma}+\tilde{\beta}\tilde{\beta}^{*_{H,h}}&
    =-\sqrt{-1}t\omega
  \end{aligned}
\end{equation}
\noi where
\begin{equation}\label{eqn:mu}
\frac{\int t\omega}{\Vol(X)}=\frac{\pi\deg(V)}{\Vol(X)}=\mu.
\end{equation}

Equations \eqref{eqn:U13Hitch-t} differ from the required $\U(1,3)$-Hitchin equations only in that the right hand side is not constant, but instead involves a function whose average value is the required constant.  Lemma \ref{lemma: confchange} allows us to remove this discrepancy by rescaling the metrics on $V$ and $\det(V)$. Indeed if we pick a function $u$ such that it satisfies the condition
$$ \Delta(u)=t-\mu $$
\noi and define metrics $K=He^u$ on $V$ and $k=he^u$ on $\det(V)$ then 
\begin{align*}\label{eqn:U13Hitch-mu}
F^K_V+\tilde{\beta}^{*_{K,k}}\tilde{\beta}+\tilde{\gamma}\tilde{\gamma}^{*_{K,k}} &=-\sqrt{-1}\mu\omega\mathbf{I}_V\nonumber\\
F^{k}_{\det(V)}+\tilde{\gamma}^{*_{K,k}}\tilde{\gamma}+\tilde{\beta}\tilde{\beta}^{*_{K,k}}& =-\sqrt{-1}\mu\omega
\end{align*}
\noi as required.

Conversely, suppose that $V$ and $\det(V)$ admit metrics $K$ and $k$
which satisfy the $\U(1,3)$-Hitchin equations on $(\det(V), V,
\tilde{\beta},\tilde{\gamma})$, namely \eqref{eqn:U13Hitch}.  In
general $k$ will differ from the metric induced by $K$ on $\det(V)$.
Denoting the latter by $\det(K)$, we can write
\begin{equation}
k=\det(K)e^u
\end{equation}
\noi where $u$ is a smooth function on $X$.  Now define new metrics on
$V$ and $\det(V)$ which are related to $K$ and $k$ by the conformal
factor $e^{u/2}$, i.e.\ set
\begin{equation}\label{Kandk}
H=Ke^{u/2}\quad\text{and}\quad h=ke^{u/2}.
\end{equation}
\noi Notice that $\det(H)=\det(K)e^{3u/2}=h$, where $\det(H)$ denotes
the metric induced by $H$ on $\det(V)$.  Moreover, since both
metrics are modified by the same conformal factor, the adjoints
$\tilde{\beta}^*$ and $\tilde{\gamma}^*$ are unaffected (see Lemma
\ref{lemma: confchange}). By parts (3) and (4) of Lemma \ref{lemma:
  confchange} and the fact that $K$ and $k$ satisfy the
$\U(1,3)$-Hitchin equations, we thus get
\begin{align*}\label{eqn:U13Hitch-mu}
  F^H_V+\tilde{\beta}^{*_{H,h}}\tilde{\beta}+\tilde{\gamma}\tilde{\gamma}^{*_{H,h}}
  &=-\sqrt{-1}(\mu-\frac{\Delta(u)}{2})\omega\mathbf{I}_V=-\sqrt{-1}t\omega\mathbf{I}_V\\
  F^{h}_{\det(V)}+\tilde{\gamma}^{*_{H,h}}\tilde{\gamma}
  +\tilde{\beta}\tilde{\beta}^{*_{H,h}}& =
  -\sqrt{-1}(\mu-\frac{\Delta(u)}{2})\omega=-\sqrt{-1}t\omega
\end{align*}
\noi where $t= \mu-\frac{\Delta(u)}{2}$ and $h=\det(H)$.  Exactly as above (see equation \eqref{eqn:U13HitchA}-\eqref{eqn:mu} ) we find that these two equations combine to yield
\begin{equation*}\label{eqn:SO*6Hitch2}
F^H_{V}+\beta\beta^{*_H}+\gamma^{*_H}\gamma = 0
\end{equation*}
\noi as required.
\end{proof}

\begin{corollary}\label{cor:SO*toU13}  Let $(V, {\beta},{\gamma})$ and
  $(\det(V),V, \tilde{\beta},\tilde{\gamma})$ be as in (\ref{eqn:SO6toU13}). Then $(V, {\beta},{\gamma})$ defines a polystable $\SO^*(6)$-Higgs bundle if and only if $(\det(V),V, \tilde{\beta},\tilde{\gamma})$ defines a polystable $\U(1,3)$-Higgs bundle.  Moreover, the map \eqref{eqn:SO6toU13} defines an embedding
\begin{equation}\label{modSO6toU13}
\mathcal{M}_d(\SO^*(6))\hookrightarrow\mathcal{M}_{d,d}(\U(1,3))
\end{equation}
\noi where $\mathcal{M}_{d,d}(\U(1,3))$ denotes the component in the moduli space of polystable $\U(1,3)$-Higgs bundles in which the bundles both have degree $d$.
\end{corollary}

\begin{proof} The first part follows immediately  from Proposition \ref{prop:eqequiv} because of the Hitchin-Kobayashi correspondence for $G$-Higgs bundles, i.e.\ Theorem \ref{HKcorrespond}.  The map defined by \eqref{eqn:SO6toU13} is clearly injective, with image given by the subvariety in which the $\U(1,3)$-Higgs bundles are defined by tuples $(L,V,\beta,\gamma)$ in which $L=\det(V)$.
\end{proof}

\begin{remark} By Proposition \ref{prop:modspace}  the dimension of $\mathcal{M}_d(\SO^*(6))$ is $15(g-1)$, while the dimension of $\mathcal{M}_{d,d}(\U(1,3))$ is $16(g-1)+1$ (see \cite{bradlow-garcia-prada-gothen:2003}).  The image of the embedding given by \eqref{modSO6toU13} thus has codimension $g$ in  $\mathcal{M}_{d,d}(\U(1,3))$.
\end{remark}

\begin{proposition} \label{lemma:SU13HK}Let  $(\det(V),V, \tilde{\beta},\tilde{\gamma})$ be a  $\U(1,3)$-Higgs bundle in which $\deg(V)$ is even.  Pick $L$ such that $L^2=\det(V)$ and define maps
\begin{equation}\label{betaLgammaL}
  \begin{aligned}
    \tilde{\beta}_L&=\tilde{\beta}\otimes 1_{L}:V\otimes
    L^{-1}\rightarrow L
    \otimes K\\
    \tilde{\gamma}_L&=\tilde{\gamma}\otimes 1_{L}:L\rightarrow V\otimes
    L^{-1}\otimes K
  \end{aligned}
\end{equation}
\noi where $1_{L}:L^{-1}\rightarrow L^{-1}$ is the identity map. Then
$(L,V\otimes L^{-1},\tilde{\beta}_L,\tilde{\gamma}_L)$ defines an
$\SU(1,3)$-Higgs bundle and, with the same notation as in Proposition
\ref{prop:eqequiv}, the following are equivalent:

\begin{itemize}
\item [(A)] The bundles $V$ and $\det(V)$ admit metrics, say $H$
  and $h$, satisfying
  \begin{equation}\label{eqn:SU13Hitch1}
    \begin{aligned}
      F^H_{V}+\tilde{\beta}^{*_{H,h}}\tilde{\beta}+\tilde{\gamma}\tilde{\gamma}^{*_{H,h}} &=-\sqrt{-1} \mu\mathbf{I}_{V}\omega\\
      F^h_{\det(V)}+\tilde{\beta}\tilde{\beta}^{*_{H,h}}+\tilde{\gamma}^{*_{H,h}}\tilde{\gamma}
      &= -\sqrt{-1}\mu\omega\ .
    \end{aligned}
  \end{equation}

\item  [(B)] The bundles $V\otimes L^{-1}$ and $L$ admit metrics, say
  $K$ and $k$, satisfying
  \begin{equation}\label{eqn:SU13Hitch}
    \begin{aligned}
      F^K_{V\otimes L^{-1}}+(\tilde{\beta}_L)^{*_{K,k}}(\tilde{\beta}_L)+(\tilde{\gamma}_L)(\tilde{\gamma}_L)^{*_{K,k}} &=0\\
      F^k_{L}+(\tilde{\beta}_L)(\tilde{\beta}_L^{*_{K,k}})+(\tilde{\gamma}_L)^{*_{K,k}}(\tilde{\gamma}_L)
      &=0.
    \end{aligned}
  \end{equation}
\end{itemize}
\end{proposition}

\begin{proof}  Since $L^{2}=\det(V)$ it follows that
\begin{equation}\label{det=0}
\det(L\oplus V\otimes L^{-1})=\det(V)\otimes L^{-2}=\mathcal{O}\ .
\end{equation}
\noi and hence $(L,V\otimes L^{-1},\tilde{\beta}_L,\tilde{\gamma}_L)$ defines a $\SU(1,3)$-Higgs bundle.

Let $h_0$ be the Hermitian-Einstein metric on $L^{-1}$, so that the curvature
of the corresponding Chern connection satisfies
$F^{h_0}_L=\sqrt{-1}\deg(L)\omega$.  Given metrics $H$ and $h$ which satisfy
(A), define $K=H\otimes h_0$ on $V\otimes L^{-1}$ and $k=h\otimes h_0$ on
$L=\det(V)\otimes L^{-1}$.  Conversely, given metrics $K$ and $k$ which
satisfy (B), define $H=K\otimes h^{-1}_0$ on $V=V\otimes L^{-1}\otimes L$ and $h=k\otimes h^{-1}_0$ on $\det(V)=L^{2}$.
\end{proof}

\begin{remark}\label{rem:SU(13)}The equations \eqref{eqn:SU13Hitch} are not exactly the $\SU(1,3)$-Hitchin equations. If $(L,W,b,c)$ is any $\SU(1,3)$-Higgs bundle, the Hitchin equations for metrics $k$ and $K$ on $L$ and $W$ respectively are equivalent to the condition
\begin{equation}\label{eqn:trueSU(13)Hitch}
\begin{bmatrix}F^K_{W}+b^{*_{K,k}}b+cc^{*_{K,k}}&0\\
0&F^k_{L}+bb^{*_{K,k}}+c^{*_{K,k}}c\end{bmatrix}_0=0
\end{equation}
\noi where $[A]_0$ denotes the trace free part of the matrix $[A]$.  The pair
\eqref{eqn:SU13Hitch}  (for the $\SU(1,3)$-Higgs bundle $(\det(V),V,
\tilde{\beta},\tilde{\gamma})$) is equivalent to \eqref{eqn:trueSU(13)Hitch}
together with the extra condition $\Tr(F^K_{V\otimes L^{-1}})+F^k_L=0$. In fact this condition can always be achieved by a simultaneous conformal transformation of the metrics $K$ and $k$, as in \eqref{Kandk}. As explained above, such conformal transformations affect only the curvature terms in the equation but do not change the trace-free parts  of those terms.
\end{remark}

\begin{remark}\label{remark:SU} By defining $V=W\otimes L$, any
  $\SU(1,3)$-Higgs bundle $(L,W,\beta,\gamma)$ can be represented by a tuple
  $(L, V\otimes L^{-1},\tilde{\beta}_L,\tilde{\gamma}_L)$, where the Higgs
  fields are maps $\tilde{\beta}_L: V\otimes L^{-1}\rightarrow L\otimes K$ and
  $\tilde{\gamma}_L: L\rightarrow V\otimes L^{-1}\otimes K$. Notice that 

\begin{itemize}
\item $L^2=\det(V)$, and hence 
\item $\tilde{\beta}_L: V\otimes L^{-1}\rightarrow L\otimes K$ defines $\beta\in
  H^0(X,V^*\det(V)\otimes K)\simeq H^0(X,\Lambda^2V\otimes K)$,
\item $\tilde{\gamma}_L:L\rightarrow V\otimes L^{-1}\otimes K$ defines
  $\gamma\in H^0(X,V\otimes\det(V)^*\otimes K)\simeq H^0(X,\Lambda^2V^*\otimes K)$.
\end{itemize}

\end{remark}

\begin{corollary}\label{prop: SUtoSO*} With notation as in Remark \ref{remark:SU},  the map 
\begin{equation}\label{eqn:StoU}
(L,V\otimes L^{-1},\tilde{\beta}_L,\tilde{\gamma}_L)\mapsto (L^2,V,{\beta},{\gamma})
\end{equation}
\noi defines a map 
\begin{equation}\label{map:SUtoU}
\mathcal{M}_l(\SU(1,3))\rightarrow\mathcal{M}_{2l,2l}(\U(1,3))\ ,
\end{equation}
\item and the map 
\begin{equation}\label{eqn:even3}
(L,V\otimes L^{-1},\tilde{\beta}_L,\tilde{\gamma}_L)\mapsto (V,{\beta},{\gamma})
\end{equation}
\noi defines a $2^{2g}:1$ surjective map
%
%
\begin{equation}\label{eqn:modspaces}
\mathcal{M}_l(\SU(1,3))\rightarrow\mathcal{M}_{2l}(\SO^*(6)).
\end{equation}
\noi Here $l=\deg(L)$, $\tau$ denotes the Toledo invariant, and the notation for the moduli spaces is as in (4) of Remark~\ref{rem:defs}.
\end{corollary}

\begin{proof} The tuple $(L^2,V,{\beta},{\gamma})$ clearly defines a
  $\U(1,3)$-Higgs bundle with $\deg(L^2)=\deg(V)=2l$, while remark
  \ref{remark:SU} shows that $(V,\beta,\gamma)$ defines a
  $\SO^*(6)$-Higgs bundle.  In order to show that the given maps
  induces maps between the indicated moduli spaces we need to show
  that the maps preserve polystability.  We do this by invoking the
  Hitchin--Kobayashi correspondences for $\SU(1,3)$-, $\U(1,3)$-, and
  $\SO^*(6)$-Higgs bundles, i.e. we show that the map preserves the
  conditions for existence of solutions to the Hitchin equations for
  the Higgs bundles.  Moreover, by Proposition \ref{lemma:SU13HK} together
  with remark \ref{rem:SU(13)}, $(L, V\otimes
  L^{-1},\tilde{\beta}_L,\tilde{\gamma}_L)$ admits a solution to the
  $\SU(1,3)$-Hitchin equations if and only if $(L^2=\det(V),
  V,\beta,\gamma)$ admits a solution to the $\U(1,3)$-Hitchin
  equations; and by Proposition \ref{prop:eqequiv}, $(\det(V),
  V,\beta,\gamma)$ admits a solution to the $\U(1,3)$-Hitchin
  equations if and only if $(V,\beta,\gamma)$ admits a solution to the
  $\SO^*(6)$-Hitchin equations.

Finally, take any point in $\mathcal{M}_{2l}(\SO^*(6))$, represented say by
$(V,{\beta},{\gamma})$. For any $L$ such that $L^2=\det(V)$, the
$\SU(1,3)$-Higgs bundles $(L, V\otimes L^{-1},\tilde{\beta}_L,\tilde{\gamma}_L)$ is in the pre-image of $(V,{\beta},{\gamma})$ under the map.  This shows that the map is surjective. The multiplicity comes from choices of square roots of $\det(V)$.
\end{proof}

In addition to the maps \eqref{modSO6toU13}, \eqref{map:SUtoU}, and
\eqref{eqn:modspaces}, we have the surjective map (see
\cite{bradlow-garcia-prada-gothen:2003})
\begin{equation}\label{map:UtoPU}
  \begin{aligned}
    \mathcal{M}_{l,b}(\U(1,3)) &\to \mathcal{M}_{\tau}(\PU(1,3)),\\
    (L,W,\beta,\gamma) &\mapsto (\mathbb{P}(L\oplus W),\beta,\gamma),
  \end{aligned}
\end{equation}
\noi where $l=\deg(L), b=\deg(W), $ and $\tau= (3l-b)/2$.  Conversely, any $\PU(1,3)$-Higgs bundle in $\mathcal{M}_{\tau}(\PU(1,3))$ is in the image of such a map, where the degrees $(l,b)$ are determined only up to the $\Z$-action $(l,b)\mapsto (l+k, b+3k)$. This corresponds to twisting $L\oplus W$ by a line bundle of degree $k$. 

These maps lead to the following relations among Higgs bundles for the groups $\SO^*(6), \SU(1,3)$, and $\PU(1,3)$. 

\begin{proposition}\label{Higgs-lifts}\hfill

\begin{enumerate}
\item The composition of maps \eqref{map:UtoPU} and \eqref{modSO6toU13} defines a surjective map
\begin{equation}\label{map:SOtoPU}
\mathcal{M}_{d}(\SO^*(6))\mapsto \mathcal{M}_{d}(\PU(1,3))\ .
\end{equation}
Moreover a $\PU(1,3)$-Higgs bundle in $\mathcal{M}_{\tau}(\PU(1,3))$ is in the image of such a map if and only if $\tau$ is an integer.

\item The composition of maps \eqref{map:UtoPU} and \eqref{map:SUtoU} defines a surjective map
\begin{equation}\label{map:SUtoPU}
\mathcal{M}_{d}(\SU(1,3))\mapsto \mathcal{M}_{2d}(\PU(1,3))\ .
\end{equation}
Moreover a  $\PU(1,3)$-Higgs bundle in $\mathcal{M}_{\tau}(\PU(1,3))$ is in the image of such a map if and only if $\tau$ is an even integer.

\item A $\SO^*(6)$-Higgs bundle in $\mathcal{M}_d(\SO^*(6))$ lies in
  the image of a map of the form \eqref{eqn:modspaces} if and only if $d$ is an even integer.
\end{enumerate}

\end{proposition}

\begin{proof}

(1) The map to $\mathcal{M}_{\tau=d}(\PU(1,3))$ is surjective since
  $\PU(1,3)$-Higgs bundles with $\tau=d$ lift to $\U(1,3)$-Higgs bundles of
  the form $(L,W,\beta,\gamma)$ with $3\deg(L)-\deg(W)=2d$ (see (\ref{map:UtoPU})). After twisting
  with a line bundle if necessary, we can assume that
  $\deg(L)=\deg(W)=d$. Furthermore, we can assume that $L=\det(W)$ since if
  not, then twisting by a square root of $\det(V)\otimes L^{-1}$ will make it so. The assertion that $\tau$ must be an even integer is clear from the definitions of the maps  \eqref{map:UtoPU} and \eqref{map:SUtoU}.

(2) As in (1), any $\PU(1,3)$-Higgs bundles with $\tau=4d$ lift to $\U(1,3)$-Higgs bundles of the form $(\det(W),W,\beta,\gamma)$. Such a Higgs bundle is in the image of \eqref{map:SUtoU} if and only if $\deg(\det(W))$ is even. This condition is satisfied precisely when $\deg(W)=2d$.

(3) This follows from the fact that the map is defined by \eqref{eqn:even3} in which $\det(V)=L^2$ and hence $\deg(V)=2\deg(L)$. 
\end{proof}

Expressed in terms of the corresponding surface group representations, Proposition \ref{Higgs-lifts} gives conditions under which reductive surface group representations into $\PU(1,3), \SO^*(6)$ or $\SU(1,3)$ lift from one group to another.  

\begin{proposition}\label{Reps-lifts}\hfill

\begin{enumerate}
\item A reductive surface group representation into $\PU(1,3)$ lifts to a representation into $\SO^*(6)$ if and only if the Toledo invariant of the associated $\PU(1,3)$-Higgs bundle is an integer.

\item A reductive surface group representation into $\PU(1,3)$ lifts to a representation into $\SU(1,3)$ if and only if the Toledo invariant of the associated $\PU(1,3)$-Higgs bundle is an even integer.

\item A reductive surface group representation into $\SO^*(6)$ lifts to a representation into $\SU(1,3)$ if and only if the Toledo invariant of the associated $\SO^*(6)$-Higgs bundle is an even integer.
\end{enumerate}

\end{proposition}

\subsubsection{Maximal components} By Proposition
\ref{non-emptiness-higgs}, the moduli spaces
$\mathcal{M}_{d}(\SO^*(6))$ are non-empty for $|d|\le 2g-2$.   The
maximal components are thus those with $\abs{d}=2g-2$ (and these are
connected by Theorem~\ref{theorem: main1}). We discuss here only the case $d=2g-2$, but the case $d=-(2g-2)$ is analogous. 

By Theorem \ref{rigidityso*2n}, the moduli spaces  $ \mathcal{M}_{2g-2} (\SO^*(6))$ exhibit a rigidity which leads to the factorization
\begin{equation}\label{eqn: so6rigidity}
\mathcal{M}_{2g-2} (\SO^*(6))\simeq \mathcal{M}_{2g-2}(\SO^*(4))\times \mathrm{Jac}(X)
\end{equation}
\noi given by
\begin{equation}
(V,\beta,\gamma)=(V_{\perp},\beta,\gamma)\oplus \ker(\gamma)\ .
\end{equation}
%
\begin{equation}
T_4: M_0(2)\times \mathcal{M}_{g-1}(\SL(2,\R))\longrightarrow \mathcal{M}_{2g-2}(\SO^*(4))
\end{equation}
\noi given by
\begin{equation}
 (U, (K^{1/2},\beta,1_{K^{1/2}}))\mapsto (U\otimes K^{1/2},\Omega\otimes\beta,\Omega^*\otimes 1_{K^{1/2}})\ ,
\end{equation}
\noi where $\beta\in H^0(X,K^2)$, $1_{K^{1/2}}$ denotes the identity map on $K^{1/2}$, and $\Omega: U^*\simeq U$ is as in Lemma \ref{lemma:SL2}.  

We thus get a  $2^{2g}$-fold covering of $ \mathcal{M}_{2g-2} (\SO^*(6))$ 
\begin{equation}\label{eqn:direct}
T_6: M_0(2)\times \mathcal{M}_{g-1}(\SL(2,\R))\times\mathrm{Jac}(X)
\longrightarrow \mathcal{M}_{2g-2}(\SO^*(6))\ .
\end{equation}

\begin{remark} A choice of $K^{1/2}$ defines a section for the map $T_4$ --- and hence for $T_6$ --- and picks out a Teichm\"uller component of $\mathcal{M}_{g-1}(\SL(2,\R))$.
\end{remark}

We get a different description of the maximal components if we exploit
the embedding of $\mathcal{M}_{2g-2}(\SO^*(6))$ in
$\mathcal{M}_{(2g-2,2g-2)}(\U(1,3))$ given
by Corollary~\ref{cor:SO*toU13}:
$$(V,{\beta},{\gamma})\mapsto (\det(V),V, \tilde{\beta},\tilde{\gamma}).$$
\noi As shown in  \cite{bradlow-garcia-prada-gothen:2005}, the
component $\mathcal{M}_{(2g-2,2g-2)}(\U(1,3))$ has maximal Toledo
invariant for $\U(1,3)$-Higgs bundles and, moreover, this moduli space
itself exhibits a rigidity. Indeed (see Theorem 3.32 in \cite{bradlow-garcia-prada-gothen:2005}) the component $\mathcal{M}_{(2g-2,2g-2)}(\U(1,3))$  factors as
\begin{equation}\label{eqn:byU13rigid}
\mathcal{M}_{(2g-2,2g-2)} (\U(1,3))\simeq \mathcal{M}_{(2g-2,0)}(\U(1,1))\times M_{2g-2}(2),
\end{equation}
\noi where $\mathcal{M}_{(2g-2,0)}$ denotes the moduli space of
$\U(1,1)$-Higgs bundles $(L,M,\beta,\gamma)$ with $\deg(L)=2g-2$ and
$\deg(M)=0$, and $M_{d}(2)$ denotes the moduli space of polystable rank 2 bundles of degree $d$.  The factorization is given by
\begin{equation}\label{eqn:U13decomp}
 (L,W,\beta,\gamma)=(L,L\otimes K^{-1},\beta,1_L)\oplus Q
\end{equation}
\noi where $W=L\otimes K^{-1}\oplus Q$.  Notice that $L=\det(W)$ if and only
if $\det(Q)=K$.  In that case, for any choice of $K^{-1/2}$ the determinant of
$Q\otimes K^{1/2}$ is trivial and we can write $Q=U\otimes K^{1/2}$ with $\det(U)=\mathcal{O}$. 
The image of the embedding of  $\mathcal{M}_{2g-2}(\SO^*(6))$ in
$\mathcal{M}_{(2g-2,2g-2)}(\U(1,3))$ is thus characterized by the condition
that $Q=U\otimes K^{1/2}$ with $\det(U)=\mathcal{O}$ in \eqref{eqn:U13decomp}.   We define 
\begin{equation}\label{eqn:MK2}
M_K(2)=\{Q\in M_{2g-2}(2)\ |\ \det(Q)=K\}.
\end{equation}
\noi The Toledo invariant is maximal for
$\mathcal{M}_{(2g-2,0)}(\U(1,1))$ and hence, by Proposition 3.30 in
\cite{bradlow-garcia-prada-gothen:2005} we can identify
$\mathcal{M}_{(2g-2,0)}(\U(1,1))$ with the moduli space of degree
zero,  $K^2$-twisted $\C^*$-Higgs bundles\footnote{Note that
  $\C^*=\GL(1,\C)$ so a $K^2$-twisted $\C^*$-Higgs bundle is a pair $(L,\beta)$
  consisting of a line bundle $L$ and a section $\beta\in H^0(X,K^2)$}, i.e.
\begin{equation}\label{eqn:U11Cayley}
  \begin{aligned}
    \mathcal{M}_{(2g-2,0)}(\U(1,1)) &\xra{\simeq} \mathrm{Jac}(X)\times
    H^0(X,K^2), \\
    (L,M,\beta,\gamma) &\mapsto (L,\beta\circ\gamma).
  \end{aligned}
\end{equation}
\noi Putting together \eqref{eqn:MK2}, \eqref {eqn:U11Cayley} and  \eqref{eqn:byU13rigid} we thus get an identification of the image of $\mathcal{M}_{2g-2}(\SO^*(6))$ in $\mathcal{M}_{(2g-2,2g-2)}(\U(1,3))$ as
\begin{equation}\label{eqn:byU13}
\mathcal{M}_{2g-2}(\SO^*(6))\simeq \mathrm{Jac}(X)\times H^0(X,K^2)\times M_K(2).
\end{equation}

Comparing \eqref{eqn:direct} and \eqref{eqn:byU13} we see that the two descriptions match up via the map
$$(U, (K^{1/2},\beta, 1), L_0)\longrightarrow (L_0,\beta, Q=U\otimes K^{1/2}).$$
\noi The fibers of this map are the $2^{2g}$ points of order 2 in $\mathrm{Jac}(X)$.

\bigskip

We note finally that the dimension of $\mathcal{M}_{\pm(2g-2)}(\SO^*(6))$ can be computed from the isomorphism \eqref{eqn: so6rigidity}. We find $\dim(\mathcal{M}_{\pm(2g-2)}(\SO^*(6))=7g-6$ whereas the expected dimension is $15(g-1)$.

\appendix

\section{$G$-Higgs bundles for other groups}
\label{appendix:G-Higgs}

We collect here some basic results about $G$-Higgs bundles for groups
other than $\SO^*(2n)$ which play a role in our analysis of
$\SO^*(2n)$-Higgs bundles.  The groups include three complex reductive
groups ($\GL(n,\C)$, $\SL(n,\C)$ and $\SO(n,\C)$) and two non-compact real forms
($\U(p,q)$ and $\U^*(2n)$).  
In all cases the basic
definitions of stability properties follow from the general definition
formulated for $G$-Higgs bundles in \cite{garcia-prada-gothen-mundet:2009a}.

\subsection{The groups $\GL(n,\C), \SL(n,\C)$ and $\SO(n,\C)$}\hfill

We begin by recalling how the notion of $G$-Higgs bundle specializes
when $G$ is a complex group. In this case, the complexified isotropy
representation is just the adjoint representation of $G$ on
$\lie{g}$. Thus, a $G$-Higgs bundle for a complex group $G$ is a pair
$(E,\varphi)$, where $E \to X$ is a holomorphic principal $G$-bundle
and $\varphi\in H^0(X,\Ad E \otimes K)$; here $\Ad E =
E\times_{\Ad}\lieg$ is the adjoint bundle of $E$. We shall use this
observation for all three groups considered in this section.

Consider first the case of $G=\GL(n,\C)$. A  $\GL(n,\CC)$-Higgs
bundle may be viewed as a pair consisting of a rank $n$ holomorphic vector
bundle $E$ over $X$ and a holomorphic section
$$
 \Phi \in H^0(X,K\otimes \End E).
$$
We refer the reader to \cite{garcia-prada-gothen-mundet:2009a} for the
general statement of the stability conditions for 
$\GL(n,\CC)$-Higgs bundles. The notions of (semi-,poly-)stability  
in this case are  equivalent
to the  original notions given by Hitchin in \cite{hitchin:1987}
(see \cite{garcia-prada-gothen-mundet:2009a}).
Denote by $\mu(E) = \deg(E) / \rk(E)$ the slope of $E$.

\begin{proposition}
  \label{thm:sl(n,C)-stability} A $\GL(n,\CC)$-Higgs bundle $(E,\Phi)$
  is {semistable} if and only if for any subbundle $E'\subset
  E$ such that $\Phi(E')\subset E' \otimes K$ we have $\mu(E')\leq
  \mu(E)$.  Furthermore, $(E,\Phi)$ is {stable} if for any nonzero and
  strict subbundle $E'\subset E$ such that $\Phi(E')\subset E' \otimes
  K$ we have $\mu(E')<\mu(E)$. Finally, $(E,\Phi)$ is {polystable} if it
  is semistable and for each subbundle $E'\subset E$ such that
  $\Phi(E')\subset E' \otimes K$ and $\mu(E')=\mu(E)$ there is another
  subbundle $E''\subset E$ satisfying $\Phi(E'')\subset E'' \otimes K$
  and $E=E'\oplus E''$.  As a consequence $(E,\Phi)=\oplus (E_i,\Phi_i)$
  where $(E_i,\Phi_i)$ is a stable $\GL(n_i,\C)$-Higgs bundle with
  $\mu(E_i)=\mu(E)$.
\end{proposition}

The group $\SL(n,\C)$ is the subgroup of $\GL(n,\C)$ defined by the
usual condition on the determinant.  A $\SL(n,\CC)$-Higgs bundle may
thus be viewed as a $\GL(n,\C)$-Higgs bundle $(E,\Phi)$ with the extra
conditions that $E$ is endowed with a trivialization $\det E\simeq\cO$
and $\Phi \in H^0(X,K\otimes \End_0 E)$ where $\End_0 E$ denotes the
bundle of traceless endomorphisms of $E$. The (semi-,poly-)stability
condition is the same as the one for $\GL(n,\CC)$-Higgs bundles given
in Proposition~\ref{thm:sl(n,C)-stability}.

Finally we consider the case $G=\SO(n,\C)$. A principal
$\SO(n,\CC)$-bundle on $X$ corresponds to a rank $n$ holomorphic
orthogonal vector bundle $(E,Q)$, where $E$ is a rank $n$
vector bundle and $Q$ is a holomorphic section of $S^2E^*$ whose
restriction to each fibre of $E$ is non degenerate. The adjoint bundle
can be identified with $\Lambda^2_QE \subset \End(E)$, the subbundle
of $\End(E)$ consisting of endomorphisms which are skew-symmetric with
respect to $Q$. A $\SO(n,\CC)$-Higgs bundle is thus a pair consisting
of a rank $n$ holomorphic orthogonal vector bundle $(E,Q)$ over $X$
and a section
$$
\Phi \in H^0(X,\Lambda^2_Q E\otimes K).
$$

The general notions of (semi-,poly-)stability specialize in the case
of $\SO(n,\CC)$-Higgs bundles to the following (see
\cite{aparicio,garcia-prada-aparicio}).

\begin{proposition}
\label{thm:sp(2n,C)-stability} A $\SO(n,\CC)$-Higgs
bundle $((E,Q),\Phi)$ is {semistable} if and only if for any
isotropic subbundle $E'\subset E$ such that $\Phi(E')\subset
K\otimes E'$ we have $\deg E'\leq 0$. Furthermore,
$((E,Q),\Phi)$ is {stable} if for any nonzero and strict
isotropic subbundle $0\neq E'\subset E$ such that $\Phi(E')\subset
K\otimes E'$ we have $\deg E'<0$. Finally, $((E,Q),\Phi)$ is
{polystable} if it is semistable and for any nonzero and strict
isotropic subbundle $E'\subset E$ such that $\Phi(E')\subset
K\otimes E'$ and $\deg E'=0$ there is a coisotropic subbundle
$E''\subset E$ such that $\Phi(E'')\subset K\otimes E''$ and
$E=E'\oplus E''$.
\end{proposition}


\begin{remark}
Recall that if $(E,Q)$ is an orthogonal
vector bundle, a subbundle $E'\subset E$ is said to be isotropic
if the restriction of $Q$ to $E'$ is identically zero, and 
coisotropic if $E'^{\perp_Q}$ is isotropic.
\end{remark}

\begin{remark}  
  For complex groups $G$, Definition \ref{def:simple} implies that a
  $G$-Higgs bundle $(E,\varphi)$ is simple if
  $\Aut(E,\varphi)=Z(H^{\C})$.  For $G=\GL(n,\C)$ or $\SL(n,\C)$ it is
  well known that stability implies simplicity. This is not so for
  $\SO(n,\C)$-Higgs bundles.  For instance it is possible for a
  stable $\SO(n,\C)$-Higgs bundle to decompose as sum of stable
  $\SO(n_i,\C)$-Higgs bundles (with $\Sigma n_i=n$).  In all cases
  though, the Higgs bundles which are stable and simple represent
  smooth points in their moduli spaces (see Proposition
  \ref{prop:smoothpoint}).
\end{remark}

\subsection{The groups $\U(p,q)$ and $\U^*(2n)$}
\subsubsection{$\U(p,q)$-Higgs bundles}\label{A:U(p,q)}


The maximal compact subgroups of $\U(p,q)$ are isomorphic to
$H=\U(p)\times \U(q)$ and hence $H^{\C}=\GL(p,\C)\times\GL(q,\C)$.
The complexified isotropy representation space is
$\liemc=\Hom(\C^q,\C^p)\oplus\Hom(\C^p,\C^q)$. A $\U(p,q)$-Higgs
bundle may thus be described by the data $(V,
W,\varphi=\beta+\gamma)$, where $V$ and $W$ are vector bundles of rank
$p$ and $q$, respectively, $\beta\in H^0(X,\Hom(W,V)\otimes K)$ and
$\gamma\in H^0(X,\Hom(V,W)\otimes K)$. 

The following proposition gives the simplified stability conditions
for $\U(p,q)$-Higgs bundles. It can be proved using arguments similar to the
ones for other real groups (cf.\ Section~\ref{sec:spnr-higgs} and
\cite[Section~4]{garcia-prada-gothen-mundet:2009a}).

\begin{proposition}
\label{prop:Upq-simplification}
A $\U(p,q)$-Higgs bundle $(V,W,\varphi=\beta+\gamma)$  is {semistable} if 
$$
\mu(V'\oplus W') \leq \mu(V\oplus W),
$$ 
is satisfied for all 
$\varphi$-invariant pairs of subbundles $V'\subset V$ and 
$W'\subset W$, i.e.\ for pairs such that 
\begin{align*}
\beta &:W'\longrightarrow V'\otimes K\nonumber\\
\gamma &:V'\longrightarrow W'\otimes K.
\end{align*}

A $\U(p,q)$-Higgs bundle $(V,W,\varphi)$  is {stable} if 
the slope  inequality is strict whenever  $V'\oplus W'$ is a proper
non-zero $\varphi$-invariant subbundle of $V\oplus W$. 

A $\U(p,q)$-Higgs bundle $(V,W,\varphi)$ is {polystable} if it is
semistable and for any $\varphi$-invariant pair of subbundles
$V'\subset V$ and $W'\subset W$ satisfying $\mu(V'\oplus
W')=\mu(V\oplus W)$ there is another $\varphi$-invariant pair of
subbundles $V''\subset V$ and $W''\subset W$ such that $ V=V'\oplus
V''\ \mathrm{and}\ W=W'\oplus W''$.  As a consequence there is a decomposition
$$(V,W,\beta,\gamma)=\bigoplus (V_i,W_i,\beta_i,\gamma_i),$$
where $V=\bigoplus V_i$, $W=\bigoplus W_i, \beta=\Sigma\beta_i
,\gamma=\Sigma\gamma_i$ and $(V_i,W_i,\beta_i,\gamma_i)$ is a stable
$\U(p_i,q_i)$-Higgs bundle with $\mu(V_i\oplus W_i)=\mu(V\oplus W)$.
\end{proposition}

\begin{remark}\label{rem:U-p-qzero}
  In the case $q=0$, the group is $\U(p)$ and hence $\varphi =
  0$. Thus a $\U(p)$-Higgs bundle is an ordinary vector bundle. 
  Proposition~\ref{prop:Upq-simplification}
  shows that in this case the $\U(p,q)$-Higgs bundles stability condition coincides with the usual one for vector bundles.
\end{remark}

\subsubsection{$\U^*(2n)$-Higgs bundles}\label{sect: U*}
The group $\U^*(2n)$ is a non-compact real form of $\GL(2n,\CC)$ consisting of 
matrices $M$  verifying that $\bar  MJ_n=J_nM$ where
$J_n=\begin{pmatrix}
     0 & I_n \\
     -I_n & 0
\end{pmatrix}.$
A maximal compact subgroup of $\U^*(2n)$ is the compact symplectic group
$\Sp(2n)$  (or, equivalently, the group of $n\times n$ quaternionic unitary
matrices),  whose complexification is $\Sp(2n,\CC)$, the complex symplectic 
group. The group $\U^*(2n)$  is the 
non-compact dual of $\U(2n)$, in the sense that 
the non-compact  symmetric space $\U^*(2n)/\Sp(2n)$ is the dual of the 
compact symmetric space $\U(2n)/\Sp(2n)$ in Cartan's classification
of symmetric spaces (cf. \cite{helgason}).

The corresponding Cartan decomposition of the complex Lie algebra is 
$$\mathfrak{gl}(2n,\CC)=\mathfrak{sp}(2n,\CC)\oplus\liemc,$$
where $\liemc=\{A\in\mathfrak{gl}(2n,\CC)\st A^tJ_n=J_nA\}$. Hence
a $\U^*(2n)$-Higgs bundle over $X$ is a pair $(E,\varphi)$, where $E$ is a
holomorphic $\Sp(2n,\CC)$-principal bundle and the Higgs field $\varphi$ is a
global holomorphic section of $E\times_{\Sp(2n,\CC)}\liemc\otimes K$.


Given a symplectic vector bundle $(W,\Omega)$, denote by $S^{2}_{\Omega}W$ the
bundle of endomorphisms $\xi$ of $W$ which are symmetric with respect to
$\Omega$ i.e. such that $\Omega(\xi\,\cdot,\cdot)=\Omega(\cdot,\xi\,\cdot)$. 
In terms of vector bundles, we have 
that a {$\U^*(2n)$-Higgs bundle} over $X$ is a triple $(W,\Omega,\varphi)$,
where $W$ is a holomorphic vector bundle of rank $2n$, $\Omega\in
H^0(X,\Lambda^2W^*)$  is a symplectic form on $W$, and the Higgs field
$\varphi\in H^0(X,S_{\Omega}^2 W\otimes K)$ is a $K$-twisted endomorphism $W\to
W\otimes K$, symmetric with respect to $\Omega$.

Given the symplectic form $\Omega$, we have the usual skew-symmetric
isomorphism 
$$\omega:W\stackrel{\simeq}{\longrightarrow}W^*$$ given by $$\omega(v)=\Omega(v,-).$$ 
The map $f\mapsto f\omega^{-1}$ defines  an isomorphism between $S^2_\Omega W$
and $\Lambda^2W$. Hence we can think of a 
$\U^*(2n)$-Higgs bundle as a triple $(W,\Omega,\varphi)$ with $\varphi\in H^0(X,S_{\Omega}^2W\otimes K)$ or as a triple $(W,\Omega,\tilde\varphi)$ with $\tilde\varphi\in H^0(X,\Lambda^2W\otimes K)$ given by
\begin{equation}\label{tildevarphi}
\tilde\varphi=\varphi\omega^{-1}.
\end{equation}

The general (semi-,poly-)stability conditions for $\U^*(2n)$-Higgs
bundles are studied in \cite{garcia-prada-oliveira:2010}, where
simplified conditions (similarly to the case of other groups) are
given. We have the following
(\cite[Proposition~3.6]{garcia-prada-oliveira:2010}).

\begin{proposition}\label{prop:orthogonal-stability}
  A $\U^*(2n)$-Higgs bundle $(W,\Omega,\varphi)$ {semistable} if
  and only if $ \deg W'\leq 0$ for any isotropic and
  $\varphi$-invariant subbundle $W'\subset W$.

  A $\U^*(2n)$-Higgs bundle $(W,\Omega,\varphi)$ is {stable} if
  and only if it is semistable and $\deg W'<0$ for any isotropic and
  $\varphi$-invariant strict subbundle $0\neq W' \subset W$.

  The $\U^*(2n)$-Higgs bundle $(W,\Omega,\varphi)$  is {polystable} if and  only if it is semistable
  and, for any isotropic (respectively coisotropic) and $\varphi$-invariant
  strict subbundle $0\neq W'\subset W$ such that $\deg W'=0$, there is
  another coisotropic (respectively  isotropic) and $\varphi$-invariant
  subbundle $0\neq W''\subset W$ such that $W\simeq W'\oplus W''$.
\end{proposition}







\providecommand{\bysame}{\leavevmode\hbox to3em{\hrulefill}\thinspace}

\end{document}